\documentclass[a4paper, 11pt, onecolumn]{article}

\usepackage[utf8]{inputenc}
\usepackage[T1]{fontenc}
\usepackage{graphicx}
\usepackage[english]{babel}
\usepackage{textcomp}
\usepackage{layout}
\usepackage[top=2.2cm, bottom=2.2cm, left=2.8cm, right=2.8cm]{geometry}
\usepackage{setspace}
\usepackage{listings}
\usepackage{color}
\usepackage[dvipsnames]{xcolor}
\usepackage{lscape}
\usepackage{csquotes}
\usepackage{braket}
\usepackage{amsmath}
\usepackage{amssymb}
\usepackage{amsthm}
\usepackage{thmtools}
\usepackage{amsfonts}
\usepackage{mathrsfs}
\usepackage{xfrac}
\usepackage{faktor}
\usepackage{esint}
\usepackage{mathtools}
\usepackage{stmaryrd}
\usepackage{array}
\usepackage{bbm}
\usepackage{hyperref}
\hypersetup{
    colorlinks=true,
    linkcolor=red,
    filecolor=magenta,      
    urlcolor=blue,
    citecolor=BlueViolet,
    }
\usepackage{bbold}
\usepackage{algorithm}
\usepackage{algpseudocode}

\usepackage[backend = bibtex, style=alphabetic, giveninits=true]{biblatex}
\usepackage{csquotes}

\usepackage{tikz}
\usetikzlibrary{shadows}

\renewbibmacro{in:}{} %supprime le in dans les citations d'articles

\AtEveryBibitem{\clearfield{month}}
\AtEveryBibitem{\clearfield{day}}
\AtEveryBibitem{\clearfield{eprinttype}}
\AtEveryBibitem{\clearfield{eprint}}
\AtEveryBibitem{\clearfield{url}}
\AtEveryBibitem{\clearfield{urlyear}}
\AtEveryBibitem{\clearfield{review}}
\AtEveryBibitem{\clearfield{isbn}}
\AtEveryBibitem{\clearfield{issn}}
\AtEveryBibitem{\clearfield{pagetotal}}
\AtEveryBibitem{\clearfield{doi}}
\AtEveryBibitem{\clearfield{note}}
\AtEveryBibitem{\clearname{editor}}
\AtEveryBibitem{\clearlist{location}}
\AtEveryBibitem{\clearfield{number}}
\AtEveryBibitem{\clearfield{volume}}
\AtEveryBibitem{\clearfield{pages}}

\newtheorem{theorem}{Theorem}[section]
\newtheorem{corollary}{Corollary}[section]
\newtheorem{lemma}{Lemma}[section]
\newtheorem{proposition}{Proposition}[section]

\theoremstyle{remark}
\newtheorem{remark}{Remark}

\newenvironment{syst}
    {\left\{ \begin{aligned}}{ \end{aligned} \right.}

\newenvironment{nalign}{
    \begin{equation}
    \begin{aligned}
}{
    \end{aligned}
    \end{equation}
    \ignorespacesafterend
}

\newcommand{\subsecnonewline}[1]{\vspace{0.5cm}\noindent\textbf{\large{#1.}}}

\renewcommand{\le}{\leqslant}
\renewcommand{\ge}{\geqslant}
\renewcommand{\epsilon}{\varepsilon}

\newcommand{\qqtqq}[1]{\qquad \text{#1} \qquad}

\newcommand{\qqtq}[1]{\qquad \text{#1} \quad}
\newcommand{\qtq}[1]{\quad \text{#1} \quad}

\newcommand{\R}{\ensuremath{\mathbb{R}}}

\newcommand{\N}{\ensuremath{\mathbb{N}}}
\newcommand{\Q}{\ensuremath{\mathbb{Q}}}
\newcommand{\Z}{\ensuremath{\mathbb{Z}}}
\newcommand{\C}{\ensuremath{\mathbb{C}}}

\newcommand{\F}{\ensuremath{\mathbb{F}}}

\newcommand{\T}{\ensuremath{\mathbb{T}}}

\newcommand{\CA}{\mathcal{A}}
\newcommand{\CB}{\mathcal{B}}
\newcommand{\CC}{\mathcal{C}}

\newcommand{\CF}{\mathcal{F}}

\newcommand{\CH}{\mathcal{H}}
\newcommand{\CL}{\mathcal{L}}
\newcommand{\CO}{\mathcal{O}}
\newcommand{\CP}{\mathcal{P}}

\newcommand{\CS}{\mathcal{S}}
\newcommand{\CU}{\mathcal{U}}

\newcommand{\CM}{\mathcal{M}}

\newcommand{\eint}[2]{\left\llbracket #1 ,#2 \right\rrbracket}

\newcommand{\floor}[1]{\left\lfloor #1 \right\rfloor}

\newcommand{\sprod}[2]{\left\langle #1, #2 \right\rangle }
\renewcommand{\set}[1]{ \left \lbrace #1 \right \rbrace  }
\newcommand{\norm}[1]{\left \Vert #1 \right \Vert}

\newcommand{\jap}[1]{\left \langle #1 \right \rangle }

\newcommand*\diff{\mathop{}\!\mathrm{d}} %élément différentiel droit

\newcommand{\modulo}[1]{\left[#1\right]}

\DeclareMathOperator{\tr}{tr}

\DeclareMathOperator{\supp}{supp}

\DeclareMathOperator{\Vect}{Vect}
\DeclareMathOperator{\Op}{Op}
\DeclareMathOperator{\Sp}{Sp}
\DeclareMathOperator{\GL}{GL}

\DeclareMathOperator{\Leb}{Leb}
\DeclareMathOperator{\spectre}{spt}
\DeclareMathOperator{\Sym}{Sym}

\DeclareMathOperator{\SL}{SL}
\DeclareMathOperator{\adj}{adj}

\title{Semiclassical measures for quantized elliptic and parabolic symplectic linear maps of the multidimensional torus}
\author{Vincent Ferrari-Dominguez}

\date{\vspace{-5ex}}

\begin{document}
\maketitle

\begin{abstract}
	We study the semiclassical measures associated to the quantum map on the torus obtained by quantizing integer symplectic linear maps, focusing on periodic matrices and certain parabolic matrices. More precisely, we consider the time-one map of the completely integrable flows generated by Hamiltonians of the form $H(x,\xi) = Q(\xi)/2$ where $Q$ is an integral quadratic form. We prove that the set of semiclassical measures associated to a periodic matrix $A$ coincides with the set of invariant probability measures of $A$. The semiclassical measures of parabolic matrices on $\T^{2}$ are explicitly characterized by proving that they have at most three nonzero Fourier coefficients in position. A partial quantum ergodicity statement for parabolic matrices is obtained by proving a sharp bound on the quantum variance and some examples are built in higher dimensions. In particular, starting from $\T^{4}$ there always exist semiclassical measures whose projection on the position has only nonzero Fourier coefficients. Moreover, if $Q$ is isotropic over $\Q$, then there exist semiclassical measures whose projection on the position concentrate on some subtori. A link is made with the semiclassical measures of the Laplacian on the torus.
\end{abstract}

\section{Introduction}

Does the chaoticity of a classical dynamical system impact the high-frequency
behavior of its quantum analogue? To study this question, Hannay and Berry
introduced in \cite{hannayQuantizationLinearMaps1980a} a toy model for
quantum chaos by building a quantum analogue of symplectic integer matrices
acting on the torus $\T^2$. Since their seminal work, the quantization of
hyperbolic\footnote{A matrix $A\in \Sp_2(\Z) = \SL_2(\Z)$ is said elliptic if
	$|\tr(A)|<2$, which implies that $A$ is periodic. It is said hyperbolic if
	$|\tr(A)| >2$, which implies that the transformation induced by $A$ on $\T^2$
	is ergodic for the Lebesgue measure. It is said parabolic if $|\tr(A)|=2$.}
symplectic matrices, called quantum cat maps in reference to the Arnold cat map,
has been extended to higher dimension and extensively studied. We refer the
reader to \cite[sec. 5]{dyatlovQuantumErgodicity2022} for an overview and to
\cite{riviereLebesgueComponentSemiclassical2025,kimCharacterizingSupportSemiclassical2025a,kurlbergQuantumErgodicityHigher2025a}
for examples of recent works on this topic. While other types of parabolic maps
such as skew translations
\cite{marklofQuantumUniqueErgodicity2000,schubertRateQuantumErgodicity2008}
and perturbed Kronecker maps \cite{rosenzweigQuantumUniqueErgodicity2006a} have
been analyzed, the semiclassical measures of elliptic and parabolic symplectic
integer matrices have not been studied to the best of the author's knowledge.

The goal of this article is to fill this gap in the literature by studying the case
of elliptic and parabolic symplectic integer matrices. Another motivation is
the link that can be drawn between elliptic, resp. parabolic, symplectic
matrices and the geodesic flow on the spheres, resp. on the torus, as well as
between the quantum analogues of these dynamical systems. This will be a
recurring theme throughout this introduction.

\subsection{Symplectic integer matrices on the torus and their quantization}

We briefly recall the quantization procedure of symplectic integer matrices on
the torus. The first step consists in building a quantum phase space associated
to the torus $\T^{2d}$. The usual choice is to consider a finite vector space
$\CH_N(\theta)$ of distributions on $\R^d$ that are quasi-periodic both in
position and frequency and to fix a Hilbert space structure on it. The
parameter $\theta \in \T^{2d}$ corresponds to a choice of phase for the
quasi-periodicity, and $N \in \N^*$ is the semiclassical parameter that can be
thought of as the inverse of a discretization length of $\T^{2d}$.

The second step consists in quantizing the classical observables. To each
observable $a \in \CC^\infty(\T^{2d})$ is associated a family of operators
$(\Op_N(a))_{N \in \N^*}$ acting on the spaces $\CH_N(\theta)$, built from the
Weyl quantization on $\R^{2d}$.

In the third and final step, the quantum analogue $M_N(A)$ of a symplectic
integer matrix $A \in \Sp_{2d}(\Z)$ acting on the torus is built. The operator
$M_N(A)$ is a unitary operator on $\CH_N(\theta)$ that satisfies an Egorov
theorem
\begin{equation*}
	\norm{M_N(A)^{-1} \Op_N(a) M_N(A) - \Op_N(a \circ A)}_{\CB(\CH_N)} = \CO_a(\frac{1}{N}) \qqtq{for} a \in \CC^\infty(\T^{2d}).
\end{equation*}
These operators are built by restricting the metaplectic operators to the spaces $\CH_N(\theta)$ for adapted values of $\theta$. The metaplectic operators can be built either by studying the Schrödinger representation of the Heisenberg group or as an evolution operator for the Schrödinger equation associated to quadratic homogeneous Hamiltonians.

Once the quantum analogue of the symplectic matrix $A$ is constructed, the goal
is then to understand the spectral properties of $M_N(A)$ in the semiclassical limit $N
	\to + \infty$. More precisely, we seek to understand the possible concentration
behavior of a sequence $(\psi_N)_{N}$ of eigenstates of $M_N(A)$. The \textit{Wigner or Husimi distributions} are
tools that describe the behavior of a quantum state $\psi_N \in \CH_N(\theta)$ and
are defined on a classical observable $a \in \CC^\infty(\T^{2d})$ by
\begin{equation*}
	Wig(\psi_N)(a) = \sprod{\Op^w_N(a)\psi_N}{\psi_N} \qqtqq{and} Hus(\psi_N)(a) = \sprod{\Op^{AW}_N(a)\psi_N}{\psi_N},
\end{equation*}
where the superscript $w$, resp. $AW$, refers to the Weyl, resp. the Anti-Wick, quantization. They are distributions on the classical phase space $\T^{2d}$ and the Husimi distribution is, in fact, a positive measure.

A \textit{semiclassical measure} of a sequence of uniformly bounded quantum
states $(\psi_N)_N$ is then defined as an accumulation point of the sequence of
Husimi distributions $(Hus(\psi_N))_N$, equivalently of the Wigner
distributions $(Wig(\psi_N))_N$. The notion does not depend on the choice of
the quantization and a semiclassical measure is always a finite positive
measure on $\T^{2d}$. The set of semiclassical measures associated to a
symplectic integer matrix $A$, denoted by $\CS \CM(A)$, is thus defined as the
set of all semiclassical measures associated to a sequence $(\psi_N)_N$ where
each $\psi_N$ is a normed eigenvector of $M_N(A)$. The aim of this paper is to
describe the set $\CS \CM(A)$ when $A$ is either parabolic or elliptic. We call
\textit{quantum limit associated to a symplectic matrix $A$} the projection on
the position coordinates of $\T^{2d}$ of any semiclassical measure of $A$.

\subsection{The elliptic case}

We first fix the terminology and decide to call a matrix $A \in \Sp_{2d}(\Z)$
\textit{elliptic} if it is periodic, meaning that there exists some integer $T
	\in \N$ such that $A^T = I$. This is consistent with the usual notion of elliptic
symplectic matrices in dimension two as they are indeed periodic. Moreover, if
$A \in \Sp_{2d}(\Z)$, then $A$ is periodic if and only if the orbit of each
point in $\C^{2d}$ is bounded (see Proposition
\ref{pro_characterization_elliptic_matrices}), a fitting dynamical behavior for
an elliptic matrix.

It is standard that the semiclassical measures associated to a symplectic
integer matrix $A$ are invariant under $A$. The reverse inclusion holds when
$A$ is elliptic, as stated in the following theorem.
\begin{theorem}\label{thm_semiclassical_measures_elliptic_matrix}
	Let $A \in \Sp_{2d}( \Z)$ be elliptic, then
	\begin{equation*}
		\CS \CM(A) = \CP_A(\T^{2d}),
	\end{equation*}
	where $\CP_A(\T^{2d})$ is the set of probability measures on $\T^{2d}$ invariant under $A$.
\end{theorem}

The proof is similar in spirit to that of \cite[thm.
	1.1]{jakobsonClassicalLimitsEigenfunctions1999} or \cite[thm.
	1]{azagraConcentrationSymmetricEigenfunctions2010} and consists in averaging a
coherent state over a period of $M_N(A)$. It can be adapted, mutatis mutandis, to
any periodic quantum map or evolution operator for a Schrödinger equation whose
period is independent of the semiclassical parameter. We call \textit{quantum
	propagator} the evolution operator solving a Schrödinger equation.

A standard example of a periodic quantum propagator is given by $t \mapsto
	e^{\frac{it}{h}D_h}$ where $D_h = \sqrt{-h^2\Delta_{S^d} + (h\frac{d-1}{2})^2}
	- h\frac{d-1}{2}$ on the sphere $S^d$. Jakobson and Zelditch
\cite{jakobsonClassicalLimitsEigenfunctions1999} proved using this property
that the semiclassical measures for the Laplacian on the sphere $S^d$ are
exactly the invariant measures for the geodesic flow on the cosphere bundle $S^*
	S^d$. This result was extended to all CROSSes\footnote{Compact Rank One
	Symmetric Spaces, namely the spheres, the real, complex and quaternionic
	projective spaces and the Cayley plane. See \cite[chap.
		3]{besseManifoldsAllWhose1978a} for details.} by Macia \cite[thm.
	4]{maciaRemarksQuantumLimits2008} and to compact manifolds of constant positive
sectional curvature by Azagra and Macia \cite[thm.
	3]{azagraConcentrationSymmetricEigenfunctions2010}. It is also possible in the
case of CROSSes to define a periodic quantum propagator by modifying the
constant $\frac{d-1}{2}$ in the definition of $D$. The values can be determined
using the table in \cite[sec 8.8]{besseManifoldsAllWhose1978a}.

Another example of a periodic quantum propagator is given by harmonic oscillators
on $\R^{2d}$ whose frequencies are rational multiples of one another. Arnaiz
and Macia \cite{arnaizLocalizationDelocalizationEigenmodes2022} proved that
in this context, the semiclassical measures of the corresponding quantum
harmonic oscillator are exactly the invariant measures of the classical one.
They also proved that this is no longer true if some of the frequencies are
independent over $\Q$.

A periodic quantum propagator also appears in the work of Charles and Lefeuvre
\cite[sec. 4.1.1, the operator $A_k$]{charlesSemiclassicalDefectMeasures2025}
on the semiclassical measure of the magnetic Laplacian on closed hyperbolic
surfaces when the magnetic field goes to $+\infty$ in the low energy regime.
They prove that in this regime all measures on an energy shell invariant under the magnetic flow are semiclassical measures for the magnetic Laplacian.

Finally, we point that it is not enough to suppose that the classical dynamic is periodic to ensure that all invariant measure are semiclassical measures. Indeed, Macia and Rivière \cite[thm. 1.3]{maciaConcentrationNonconcentrationSchrodinger2016a} obtained counterexamples to such a statement on Zoll manifolds.

\begin{remark}
	Kurlberg and Rudnick \cite{kurlbergHeckeTheoryEquidistribution2000} introduced, under some arithmetic condition on $A$, a family of operators on $\CH_N(0)$, called Hecke operators, that commute with $M_N(A)$. Then they considered joint eigenvectors of $M_N(A)$ and the Hecke operators, the so-called Hecke eigenfunctions, and proved that the only semiclassical measure of Hecke eigenfunctions is the Lebesgue measure, a phenomenon called arithmetic quantum unique ergodicity.
	Theorem \ref{thm_semiclassical_measures_elliptic_matrix} highlights the different behavior of Hecke eigenfunctions and general eigenvectors. Indeed, the matrix $A = \begin{pmatrix}
			1 & 1 \\ -2 & 1
		\end{pmatrix}$ satisfies arithmetic quantum unique ergodicity as already observed by Kelmer \cite[rem. 0.3]{kelmerArithmeticQuantumUnique2010a}, while Theorem \ref{thm_semiclassical_measures_elliptic_matrix} implies that all $A$-invariant measures are semiclassical measures, since $A$ is $4$-periodic.
\end{remark}

\subsection{The parabolic case}

As in the elliptic case, we first need to specify the set of parabolic matrices
we consider. We decide to call a matrix $A \in \Sp_{2d}(\Z)$ \textit{parabolic}
if it is of the form
\begin{equation*}
	P_B := \begin{pmatrix} I & B \\ 0 & I \end{pmatrix} \qqtq{where} B^T = B.
\end{equation*}

The symplectic matrices of the form $(P_B)_{B \in \Sym_d(\Z)}$ are easy to
quantize, form a subgroup of $\Sp_{2d}(\Z)$ and each of them has unbounded
orbits but no exponential growth. Moreover, there is a direct link between the
matrices of the form $P_B$ and quadratic Hamiltonians of the form
\begin{equation*}
	p_B(x, \xi) = \frac{1}{2} \sprod{B \xi}{\xi} \qqtq{for} (x,\xi) \in T^* \T^d = \T^d \times \R^d.
\end{equation*}
The Hamiltonian flow $(\phi_t)_{t \in \R}$ corresponding to $p_B$ is
\begin{equation*}
	\phi_t(x,\xi) = (x + tB\xi, \xi) = P_B^t(x, \xi) \qqtq{for} (x,\xi) \in \T^d \times \R^d, \; t \in \R.
\end{equation*}
This link also exists at the quantum level as the metaplectic operator $M_h(P_{B})$ is the quantum propagator associated to the Hamiltonian $\Op^w_h(p_B)$, see Theorem \ref{thm_metaplectic_operators_quadratic_hamiltonians} for a generalization of this phenomenon.
We note that the case $B=I$ corresponds to the geodesic flow on the torus and it is thus no surprise that the semiclassical measures of the Laplacian on the torus, denoted $\CS \CM(-\Delta_{\T^{d}})$, are also semiclassical measures for $P_{2I}$. The converse is, however, false, as shown by the following corollary, which is obtained by explicitly associating to each eigenvector of the Laplacian an eigenvector of $M_N(P_{2I})$, see Proposition \ref{pro_link_semiclassical_laplacian_P_2I} for details.
\begin{corollary}\label{cor_semicla_laplacian_torus_included_semicla_P_2I}
	We have
	\begin{equation*}
		\CS \CM (-\Delta_{\T^d}) \subsetneq \CS \CM (P_{2I}).
	\end{equation*}
\end{corollary}
The inclusion is strict by Theorem \ref{thm_full_characterization_parabolic_dimension_1} in dimension $1$, and by the example built in Theorem \ref{thm_parabolic_case_infinite_fourier_coef} in higher dimensions. This corollary follows from Proposition \ref{pro_link_semiclassical_laplacian_P_2I}, which associates to any eigenvector of the Laplacian an eigenvector of $M_N(P_{2I})$ on $\CH_N(0)$.

As claimed above, parabolic matrices are easy to quantize, particularly if $N$ is even. If $N$ is even, then there exists an eigenbasis $(f^0_j)_{j \in \eint{0}{N-1}^d}$ of $\CH_N(0)$ such that
\begin{equation*}
	M_N(P_B) f^0_j = e^{-\frac{\pi i}{N}\sprod{Bj}{j}} f^0_j.
\end{equation*}
We observe from this that the eigenvalues and the structure of the eigenspaces of $M_N(P_B)$ are dictated by the level sets of $j \mapsto \sprod{B j}{j}$ on $(\Z/N\Z)^d$. As the Fourier coefficients of the Wigner distribution are given by the formula
\begin{equation*}
	Wig(\sum_{j \in (\Z/N\Z)^d} \psi_j f^0_j)(e^{2 \pi i \sprod{(m,n)}{\cdot}}) = e^{\frac{\pi i}{N} \sprod{n}{m}} \sum_{j \in (\Z / N \Z)^d} e^{\frac{2 \pi i }{N}\sprod{n}{j}} \psi_j \overline{\psi_{j+m}},
\end{equation*}
we see that the Fourier coefficients of semiclassical measures of $P_B$ are intimately linked to solutions of equations of the form
\begin{equation*}
	\sprod{Bj}{j} \equiv \sprod{B(j+m)}{j+m} \equiv a \quad \modulo{N},
\end{equation*}
where $m \in \Z^d$ corresponds to the Fourier coefficient $(m,n)$ under scrutiny and $a \in \Z/N\Z$ can be interpreted as the radius of a sphere defined by the quadratic form $j \mapsto \sprod{B j}{j}$. 

\begin{remark}
	The main difference between the semiclassical measures of the Laplacian on the torus and those of $P_{2I}$ stems from the fact that the frequencies of an eigenvector of the Laplacian are all on a sphere, while those of eigenvectors of the metaplectic operator $M_N(P_{2I})$ only need to be on the same sphere modulo $N$. In particular, while the radius of the sphere must go to $+\infty$ for the Laplacian to obtain semiclassical measures, spheres of any radius matter for the semiclassical measures of $P_{2I}$.
\end{remark}

\subsubsection{Semiclassical measures of parabolic matrices on \texorpdfstring{$\T^2$}{T²}}

We first completely describe the semiclassical measures for parabolic matrices
on $\T^2$. As there is only one parameter for parabolic matrices in dimension
$1$, we denote
\begin{equation*}
	P_b = \begin{pmatrix}
		1 & b \\ 0 & 1
	\end{pmatrix} \qqtqq{for} b \in \Z^*.
\end{equation*}
As $\CS \CM(P_b) = \CS \CM(P_b^{-1}) = \CS \CM(P_{-b})$ per Lemma \ref{lem_basic_pro_semicla_meas_symp_matrix}, we restrict ourselves to the case $b \in \N^*$.

\begin{theorem}\label{thm_full_characterization_parabolic_dimension_1}
	Let $b \in \N^*$. We have
	\begin{equation*}
		\CS \CM (P_b) = \set{\gamma \Leb_x \otimes \nu + \sum_{l=0}^{bm-1} \alpha_l (1 + \beta_l \cos(2 \pi m x + \phi_l)) \Leb_x  \otimes \delta_{\frac{l}{bm}} }_{\gamma, \nu, m, \alpha_l, \beta_l, \phi_l},
	\end{equation*}
	where the parameters satisfy
	\begin{equation*}
		\nu \in \CP(\T), \, \gamma \in [0,1],\,  m \in \N^*, \, (\alpha_l, \beta_l, \phi_l)_{l \in \eint{0}{bm-1}} \in ([0,1]^2 \times \R)^{bm} \qtq{with} \gamma + \sum_{l=0}^{bm-1} \alpha_l = 1.
	\end{equation*}
\end{theorem}

This theorem implies that quantum limits of a parabolic matrix $P_b$ have a
density against Lebesgue and this density is a trigonometric polynomial whose
support in frequency is of the form $\set{-m,0,m}$ for $m \in \N^*$. It can
thus be seen as an analogue of a celebrated theorem of Jakobson \cite[thm.
	1.2]{jakobsonQuantumLimitsFlat1997}: any quantum limit of the Laplacian on the
torus $\T^2$ is absolutely continuous with respect to the Lebesgue measure and
its density is a trigonometric polynomial whose support in frequency is
included in two circles centered at the origin.

We highlight that the set of semiclassical measures obtained in Theorem
\ref{thm_full_characterization_parabolic_dimension_1} is path-connected but not
convex, as is also the case for the semiclassical measures for the Laplacian on
the torus $\T^2$. The work of Arnaiz and Macia
\cite{arnaizLocalizationDelocalizationEigenmodes2022} on harmonic
oscillators provides another example of a non-convex set of semiclassical
measures.

\begin{remark}\label{rem_semicla_parabolic_dim_1_all_matrices_trace_2}
	The set of semiclassical measures associated to any symplectic integer matrix of trace $2$ can be computed using Theorem \ref{thm_full_characterization_parabolic_dimension_1}. Indeed, for any $A \in \Sp_2(\Z)$ with $\tr(A) = 2$, there exists $C \in \Sp_2(\Z)$ and $b \in \Z$ such that $A = C P_b C^{-1}$. This is a consequence of the characterization of the conjugacy classes in $\SL_2(\Z)$ \cite[thm. 11.6]{karpenkovGeometryContinuedFractions2013}\footnote{We warn the reader that Karpenkov classifies the conjugacy classes in $\GL_2(\Z)$, therefore if needed one should multiply by $\begin{pmatrix}1 & 0 \\ 0 & -1 \end{pmatrix}$ to obtain a conjugacy $C \in \SL_2(\Z)$. Looking at the characteristic polynomial of $A \in \SL_2(\Z)$, it is straightforward that if $|\tr(A)| = 2$ then the eigenvalues of $A$ are either $1$ or $-1$.} and the coincidence $\Sp_2(\Z) = \SL_2(\Z)$. As Lemma \ref{lem_basic_pro_semicla_meas_symp_matrix} implies that $\CS \CM(A) = C_{\#} \CS \CM(P_b)$, there remain three cases to consider:
	\begin{itemize}
		\item if $b = 0$, then $A = I$ and $\CS \CM(A) = \CP(\T^{2d})$ by Theorem
		      \ref{thm_semiclassical_measures_elliptic_matrix},
		\item if $b > 0$, then Theorem \ref{thm_full_characterization_parabolic_dimension_1}
		      and the relation $\CS \CM(A) = C_{\#} \CS \CM(P_b)$ concludes,
		\item if $b < 0$, then $\CS \CM(A) = \CS \CM(A^{-1}) = \CS \CM( C P_{-b} C^{-1})$
		      brings us back to the previous case.
	\end{itemize}
	If the trace of $A$ is $-2$, Theorem \ref{thm_full_characterization_parabolic_dimension_1} does not completely characterize $\CS \CM(A)$, but it constrains it. Indeed, the only eigenvalue of $A$ is $-1$ and thus $\tr(A^2) = 2$. Since $\CS \CM(A) \subset \CS \CM(A^2)$ according to Lemma \ref{lem_basic_pro_semicla_meas_symp_matrix}, the discussion above restricts heavily the possible semiclassical measures for $A$. In particular, if $A = \begin{pmatrix}
			-1 & b \\ 0 & -1
		\end{pmatrix}$, then $A^2 = P_{2b}$ and Theorem \ref{thm_full_characterization_parabolic_dimension_1} implies that any quantum limit of $A$ has at most three nonzero Fourier coefficients and is thus absolutely continuous with respect to the Lebesgue measure.
\end{remark}

\subsubsection{Partial quantum ergodicity}

The case of $\T^2$ settled, we sail to higher dimensions. A first result tells
us that measures of the form $\nu \otimes \Leb$ are still semiclassical
measures for parabolic matrices.

\begin{proposition}\label{pro_parabolic_nu_times_lebesgue}
	Let $B \in \Sym_d(\Z)$. Then
	\begin{equation*}
		\set{ \Leb_{\T^d} \otimes \nu  \, | \, \nu \in \CP(\T^d)} \subset \CS \CM (P_B).
	\end{equation*}
\end{proposition}

In particular, the full Lebesgue measure on $\T^{2d}$ is a semiclassical
measure. This is also the case for elliptic matrices as a direct consequence of
Theorem \ref{thm_semiclassical_measures_elliptic_matrix} and for ergodic
symplectic matrices thanks to the quantum ergodicity theorem, proven for
quantum maps on the torus $\T^2$ by Bouzouina and De Bièvre
\cite{bouzouinaEquipartitionEigenfunctionsQuantized1996a} with a proof that
generalizes to higher dimensions. 

\begin{remark}
	An analog of this fact on compact Riemannian manifold $(M,g)$ does not hold. More precisely, the Liouville measure on the cosphere bundle $S^*M$ is not always a semiclassical measure for the Laplacian $-\Delta_g$ on $M$. An example is given by the Laplacian on the torus $\T^{2}$ with the metric $d g_{\sqrt{2}, \sqrt{3}} = \sqrt{2} d x^2 + \sqrt{3} d y^2$. The associated Laplacian is
	\begin{equation*}
		-\Delta_{\sqrt{2}, \sqrt{3}} = \frac{1}{\sqrt{2}} \partial_x^2 + \frac{1}{\sqrt{3}} \partial_y^2.
	\end{equation*}
	An eigenbasis for $-\Delta_{\sqrt{2}, \sqrt{3}}$ is given by complex exponentials $(e^{2\pi i \sprod{n}{\cdot}})_{n \in \Z^2}$ and the multiplicity of a positive eigenvalue is at most $4$, as $\sqrt{2}$ and $\sqrt{3}$ are irrational and independent over $\Q$. This implies that the semiclassical measures of $-\Delta_{\sqrt{2}, \sqrt{3}}$ are of the form
	\begin{equation*}
		\sum_{i=1}^4 P_i(x,y) \Leb_{x,y} \otimes \delta_{z_i},
	\end{equation*}
	where $P_1, \dots, P_4$ are trigonometric polynomials and $z_1, \dots, z_4$ are points on the sphere. In particular, the Liouville measure on $\T^2 \times S^1$ is not a semiclassical measure for $-\Delta_{\sqrt{2}, \sqrt{3}}$.
\end{remark}

Proposition \ref{pro_parabolic_nu_times_lebesgue} prevents any hope for some
regularity of semiclassical measures in momentum, but raises the question of
the possible behavior for the projection on position. The next corollary tells
us that, up to some exotic and rare sequences of eigenvectors, all
semiclassical measures project to the Lebesgue measure on the position
coordinate. In other words, the generic quantum limit is the Lebesgue measure
in position.
\begin{corollary}[Partial quantum ergodicity in position]\label{cor_partial_quantum_ergodicity_position}
	Let $B \in \Sym_d(\Z) \cap \GL_d(\Q)$. Consider a sequence $(\theta_N)_N$ such that $M_N(P_B)$ is well defined on $\CH_N(\theta_N)$ and fix for each $N \ge 0$ an orthonormal basis $(\psi_{N,j})_{j \in \eint{1}{N}^d}$ of $\CH_N(\theta_N)$ composed of eigenvectors of $M_N(P_B)$. Then there exists a sequence $(E_N)_N$ of subsets of $\eint{1}{N}^d$ that satisfies
	\begin{equation*}
		\frac{|E_N|}{N^d} \xrightarrow[N \to + \infty]{} 1 \qqtqq{and} \pi_\# Wig(\psi_{N,j}) \xrightharpoonup[N \to + \infty, j \in E_N]{}^* \Leb,
	\end{equation*}
	where $\pi : (x,\xi) \in \T^{2d} \mapsto x \in \T^d$ is the projection in position.
\end{corollary}

This statement resembles the quantum ergodicity theorem, except that here only
a projection of the Wigner distribution converges to the Lebesgue measure.
Similar results on partial quantum ergodicity were proven by Marklof and
Rudnick \cite{marklofAlmostAllEigenfunctions2012a} for the Dirichlet
Laplacian on rational polygons and investigated by Rivière
\cite{riviereRemarksQuantumErgodicity2013}.

The proof of Corollary \ref{cor_partial_quantum_ergodicity_position}
follows the same reasoning as that of the quantum ergodicity theorem
\cite{snirelcprimemanErgodicPropertiesEigenfunctions1974,colindeverdiereErgodiciteFonctionsPropres1985,zelditchUniformDistributionEigenfunctions1987b}
and can be deduced in a standard way from information about the quantum
variance. The quantum variance associated to an orthonormal basis $(\psi_j)_{j
			\in \eint{1}{N}^d}$ of $\CH_N(\theta)$ and an observable $a \in
	\CC^\infty(\T^{2d})$ is given by
\begin{equation*}
	V(a,N,(\psi_j)_{j}) := \frac{1}{N^d} \sum_{j\in \eint{1}{N}^d} \left| \sprod{\Op^w_N(a) \psi_j }{\psi_j} - \bar{a} \right|^2 \qqtq{where} \bar{a} = \int_{\T^{2d}} a(x,\xi) \diff x \diff \xi.
\end{equation*}

We obtain a precise bound for the quantum variance of an eigenbasis of $M_N(P_B)$
if we restrict to functions $a \in \CC^\infty(\T^{2d})$ that only depend on
the position.
\begin{theorem}\label{thm_upper_bound_quantum_variance_position_parabolic_matrices}
	Let $B \in \Sym_d(\Z) \cap \GL_d(\Q)$ and $N \ge 1$. Consider an orthonormal basis $(\psi_{j})_{j \in \eint{1}{N}^d}$ of $\CH_N(\theta_N)$ composed of eigenvectors of $M_N(P_B)$ and an observable $a \in \CC^\infty(\T^{2d})$ that depends only on the position, i.e. $a(x, \xi) = a(x)$, then for any $s > \frac{d}{2}$, we have
	\begin{equation*}
		V(a,N, (\psi_j)_{j}) \le C_{s,d} \frac{|\det(B)|}{N} \norm{a}_{H^{s+1}(\T^d)} \norm{a}_{H^s(\T^d)},
	\end{equation*}
	where $C_{s,d}$ is a constant that depends only on $s$ and $d$.
\end{theorem}

The proof follows almost the standard line of reasoning that consists in
controlling the remainders in the proof of quantum ergodicity as initiated by
Zelditch \cite{zelditchRateQuantumErgodicity1994} and perfected by Schubert
\cite{schubertUpperBoundsRate2006a} for the geodesic flow on negatively
curved manifolds. There is a small twist in that we let the averaging time go
to $+\infty$ before taking the semiclassical limit $N \to + \infty$, while the
standard proof does the inverse. This idea was already used by Hezari and
Rivière \cite{hezariEquidistributionToralEigenfunctions2020} to obtain
a similar bound for the partial quantum variance for the Laplacian on the torus. They
prove a bound of the form $V(a) \le \norm{a}_{L^2} h$ improving on their
previous bound of order $h^{2/3}$ obtained by the standard semiclassical
method \cite{hezariQuantitativeEquidistributionProperties2017a}. A bound on
the quantum mean for the Laplacian on the torus $\T^d$ was derived by
arithmetic means by Lester and Rudnick
\cite{lesterSmallScaleEquidistribution2017a} to prove small-scale
equidistribution results. McKenzie and Sabri
\cite{mckenzieQuantumErgodicityPeriodic2023} also used the same twist to
prove a general criterion ensuring the quantum ergodicity of Schrödinger
operators on large periodic graphs. They also obtained a rate of convergence of
order at least $N^{-1}$ under some hypotheses on the graph (see Remark 1.7 in
their paper).

The quantum variance was also studied for other parabolic maps on the torus.
Marklof and Rudnick \cite{marklofQuantumUniqueErgodicity2000} quantized the
skew translations and proved that the Wigner distribution of any eigenstate
converges to the Lebesgue measure at an optimal rate of $N^{-1/2}$. Schubert
adapted his method to quantum maps and obtained a control of order
$N^{-1+\epsilon}$ on the partial quantum variance for skew translations
\cite[thm. 6]{schubertRateQuantumErgodicity2008}. His method would also work
in our case and give the same rate of $N^{-1+\epsilon}$. Perturbed Kronecker
maps were studied by Rosenzweig \cite{rosenzweigQuantumUniqueErgodicity2006a}
and he proved that the Wigner measures of eigenstates converge to the Lebesgue
measure at a rate of at least $N^{-2}$. Both the skew translations and the
perturbed Kronecker maps are uniquely ergodic but non-mixing.

We prove that the bound obtained in Theorem
\ref{thm_upper_bound_quantum_variance_position_parabolic_matrices} is optimal
in its dependence on $N$. It is, however, likely that the Sobolev norms used in
Theorem \ref{thm_upper_bound_quantum_variance_position_parabolic_matrices} are
not optimal, see Remark \ref
{rem_proof_thm_upper_bound_quantum_variance_position_parabolic_matrices}.
\begin{theorem}\label{thm_lower_bound_quantum_variance_position_parabolic_matrices}
	Let $B \in \Sym_d(\Z) \cap \GL_d(\Q)$. There exists an increasing sequence $(N_k)_k$ of integers, an orthonormal eigenbasis $(\psi_{k,j})_{j \in \eint{1}{N_k}^d}$ of $M_{N_k}(P_B)$ on $\CH_{N_k}(0)$ and a vector $n \in \Z^d  \times \set{0} $ such that
	\begin{equation*}
		V(e^{2 \pi i \sprod{n}{\cdot}}, N_k,(\psi_{k,j})_{j}) \ge \frac{1}{2^{d+1}N_k}.
	\end{equation*}
\end{theorem}

\subsubsection{Examples of semiclassical measures in higher dimensions}

We also investigate the possible behavior of the exotic sequences of
eigenvectors missed by Corollary
\ref{cor_partial_quantum_ergodicity_position}. The next result tells us that
there always exists a quantum limit that has all its Fourier coefficients
strictly positive in dimension greater than or equal to $2$.
\begin{theorem}\label{thm_parabolic_case_infinite_fourier_coef}
	Let $d \ge 2$ and $B \in \Sym_d(\Z) \cap \GL_d(\Q)$. Then there exists $\mu \in \CS \CM(P_{2B})$ such that
	\begin{equation*}
		\forall n \in \Z^d, \quad \widehat{\mu}(n,0) = \widehat{\pi_\# \mu}(n) > 0.
	\end{equation*}
\end{theorem}

In the case of the Laplacian on the torus $\T^d$, it was shown by Jakobson
\cite[sec. 5]{jakobsonQuantumLimitsFlat1997} that there exist quantum limits
whose density is not a trigonometric polynomial if $d \ge 4$. This result was
recently extended by McGrath \cite[thm.
	1.5]{mcgrathVariationPrimeKtuples2022} to $d=3$. It is not true for $d=2$ as
Jakobson \cite[thm. 1.2]{jakobsonQuantumLimitsFlat1997} showed that any
quantum limit on $\T^{2}$ has a density against the Lebesgue measure which is a
trigonometric polynomial.

\begin{remark}\label{rem_infinite_fourier_coeff_parabolic}
	We point out that the quantum limit corresponding to $\mu$ given by Theorem \ref{thm_parabolic_case_infinite_fourier_coef} cannot be a quantum limit for the Laplacian on the torus $\T^{d}$. Indeed, for $n \in \Z^d$, the $n$-th Fourier coefficient of a quantum limit can be non-null if and only if there exists a sequence of spheres of increasing radius such that each sphere contains at least two integer points whose difference is $n$. However, if $n$ is the difference between two integer points on the same sphere, then $\norm{n}^2$ is even. In particular, a quantum limit cannot have only nonzero Fourier coefficients, as is the case for the projection of $\mu$.
\end{remark}

The last results of the paper study the link between the integral quadratic
form $Q$ associated to $B$ and the semiclassical measures of $P_B$. Before
stating them, we need to introduce some terminology about quadratic forms and
symplectic linear algebra. Let $k$ be a field and $(E, \sigma)$ a symplectic
vector space on $k$. A quadratic form $Q$ on $k^d$ is said \textit{isotropic}
if there exists a vector $x \neq 0$ such that $Q(x) =0$, otherwise it is said
\textit{anisotropic}. A vector subspace $F$ of $k^d$ is called \textit{totally
	isotropic} if the restriction of $Q$ to $F$ is zero and the \textit{isotropy
	index} of $Q$ is the maximal dimension of a totally isotropic subspace for $Q$.

The symplectic orthogonal of a vector subspace $F$ of $E$, denoted
$F^\sigma$, is given by
\begin{equation*}
	F^\sigma = \set{x \in E \; | \; \forall y \in F, \; \sigma(x,y) = 0}.
\end{equation*}
The subspace $F$ is said \textit{symplectic} if $F \cap F^\sigma = \set{0}$, \textit{isotropic} if $F \subset F^\sigma$, \textit{coisotropic} if $F^\sigma \subset F$ and \textit{Lagrangian} if $F = F^\sigma$. Two vector subspaces $F,G$ of $E$ are \textit{symplectic, resp. Lagrangian, supplementaries}, if they are symplectic, resp. Lagrangian, and supplementary in $E$.

If the quadratic form $Q$ is isotropic over $\Q$, then there exist
semiclassical measures that are concentrated on a submanifold of $\T^{2d}$.
Such submanifolds are called \textit{superscars}.
\begin{theorem}\label{thm_weird_measure_parabolic_isotropic}
	Let $B \in \Sym(\Z^d)$ and $Q$ be the associated quadratic form. If $Q$ is isotropic over $\Q$ with isotropy index $r$, then there exists a semiclassical measure $\mu$ of $P_{2B}$ of the form
	\begin{equation*}
		\mu = \delta_{\T^r_{E}} \otimes \Leb_{\T^{d-r}_{E^\bot}} \otimes \delta_{\T^r_{F}} \otimes \Leb_{\T^{d-r}_{F^\bot}},
	\end{equation*}
	where $\T^{l}_{V}$ denote the image in the torus of $V \otimes \R$ for $V$ a subspace of $\Q^{2d}$ of dimension $l$. The subspace $E$ is a maximal totally isotropic subspace for $Q$. The vector subspaces $E \oplus F$ and $E^\bot \oplus F^\bot$ are symplectic supplementaries in $\R^{2d}$ and $E$ and $F$ are supplementary Lagrangian in $E \oplus F$.
\end{theorem}

In this example, the semiclassical measure exhibits superscars on a symplectic
submanifold consisting of fixed points for $P_B$. This phenomenon was already
noticed in an almost identical setting by Anantharaman, Fermanian-Kammerer and
Macia \cite[example 1.6,
	1.7]{anantharamanSemiclassicalCompletelyIntegrable2015a}. They observed that
some semiclassical measures associated to the Hamiltonian $H(\xi) = \xi_1^2 -
	\xi_2^2$ or $H(\xi) = \xi_1^2 + \xi_2^2 - \xi_3^2$ exhibited superscars. Note
that these two Hamiltonians are isotropic quadratic forms over $\Q$.

Other examples of superscars for symplectic integer matrices were known \cite{gurevichWeilRepresentationDeligne2006,kelmerArithmeticQuantumUnique2010a}, but
with a different flavor. If the matrix $A$ had an isotropic rational invariant
vector subspace $E$, then Kelmer \cite{kelmerArithmeticQuantumUnique2010a}
proved that some semiclassical measures of $A$ exhibit superscars on the
subtorus $\T_{E^\sigma}$. He also showed that isotropic rational invariant
subspaces are the only obstruction to arithmetic quantum unique ergodicity for
symplectic integer matrices. Note that the subtorus $\T_{E^\sigma}$ is a
coisotropic submanifold of $\T^{2d}$. Another example of scars was given by De
Bièvre, Faure and Nonnenmacher \cite{faureScarredEigenstatesQuantum2003}
who proved that $\frac{1}{2} \Leb + \frac{1}{2} \delta_0$ is a semiclassical
measure for any hyperbolic symplectic integer matrix on $\T^{2}$.

\begin{remark}
	The problem of determining whether a given integral quadratic form $Q$ in $d$ variables is isotropic over $\Q$ is made tractable by the Hasse-Minkowski theorem \cite[sec. 4.3.2, thm. 8]{serreCourseArithmetic1973a}. This theorem states that $Q$ is isotropic over $\Q$ if and only if it is isotropic over $\R$ and all the $p$-adic fields $\Q_p$. It implies in particular that, if $d \ge 5$ or if the discriminant\footnote{The discriminant of a non-degenerate quadratic form associated to a symmetric matrix $B$ is the image of $\det(B)$ in $\Q^*/(\Q^*)^2$.} of $Q$ is $\pm 1$, then $Q$ is isotropic over $\Q$ if and only if it is isotropic over $\R$, hence if it is neither positive nor negative definite \cite[chap. 4.3.2, cor. 3 and chap. 5.2.2, thm. 3]{serreCourseArithmetic1973a}. In the other cases, determining whether $Q$ is isotropic over $\Q_p$ can be done via finite computations of Hilbert symbols involving the discriminant and the Hasse invariant of $Q$. A standard reference is \cite[chap. 3, 4, 5]{serreCourseArithmetic1973a}.
\end{remark}

The next theorem provides a bound on the partial sums of the Fourier
coefficients of quantum limits in terms of the maximal number $H(M)$ of
equations of the form $Q(j) = Q(j-n)$, where $n \in \Z^d$ with $0< \norm{n} \le
	M$, that admit a joint complex solution $j\in \C^d$. A simple bound for $H(M)$
is given when $Q$ is anisotropic over $\Q$.
\begin{theorem}\label{thm_bound_sum_fourier_coeff_parabolic}
	Let $d \ge 2$ and $B \in \Sym_d(\Z) \cap \GL_d(\Q)$. Denote by $Q$ the quadratic form associated to $B$ and $B_0(M)^* = \set{n \in \Z^d \setminus \set{0} \, | \, \norm{n} \le M}$ for $M \in \N$. Then for any $M \ge 0$,
	\begin{equation*}
		\sup_{\mu \in \CS \CM(P_{B})} \sum_{n \in B_0(M)^*} |\widehat{\pi_\# \mu}(n)| \le H(M),
	\end{equation*}
	where
	\begin{equation*}
		H(M) = \max_{X \subset B_0(M)^*} |X| \text{ s.t. } \exists j \in \C^d, \, \forall n \in X, \, Q(j) = Q(j-n).
	\end{equation*}
	Furthermore, if $Q$ is anisotropic over $\Q$ then
	\begin{equation*}
		H(M) \le \frac{|B_0(M)^*|}{2}.
	\end{equation*}
\end{theorem}

The bound on $H(M)$ is clearly suboptimal but it immediately implies that the
Dirac mass at $0$ cannot be a quantum limit if $Q$ is anisotropic over $\Q$.
\begin{corollary}\label{cor_Q_anisotropic_dirac_not_quantum_limit}
	Let $d \ge 2$ and $B \in \Sym_d(\Z)$. Denote by $Q$ the quadratic form associated to $B$ and suppose that it is anisotropic over $\Q$. Then
	\begin{equation*}
		\forall \mu \in \CS \CM(P_B) \qquad \pi_\# \mu \neq \delta_0.
	\end{equation*}
\end{corollary}

Theorem \ref{thm_bound_sum_fourier_coeff_parabolic} is probably more
interesting for its proof strategy than for the bound obtained. The idea is to
bound the sum of Fourier coefficients of the Wigner distribution by the number
of equations $Q(j) = Q(j-n)$ that have a joint solution $j$ modulo $N$ and then
to bound this number when $N$ goes to infinity by the number $H(M)$ that counts
the same quantity but with solutions in $\C^d$.

\begin{remark}
	Consider the matrix $B = \begin{pmatrix} 1 & 0 \\ 0 & -2 \end{pmatrix}$, whose associated quadratic form is $Q(x,y) = x^2 -2 y^2$. While $Q$ is isotropic over $\R$, it is anisotropic over $\Q$ as $\sqrt{2}$ is irrational. Corollary \ref{cor_Q_anisotropic_dirac_not_quantum_limit} thus implies that $\delta_0$ is not a quantum limit for $P_B$. This may come as a surprise as $Q$ is non-definite and one could expect $P_B$ to have wildly singular quantum limits.
\end{remark}

Stronger results on the regularity of the quantum limits hold for the Laplacian on the torus. Indeed, Bourgain
\cite{jakobsonQuantumLimitsFlat1997} proved that all quantum limits of the
Laplacian on the torus have a density against the Lebesgue measure. This was
also proved by Anantharaman and Macia
\cite{anantharamanSemiclassicalMeasuresSchrodinger2014} via semiclassical
two-microlocal methods and extended to more general completely integrable
systems with Fermanian-Kammerer
\cite{anantharamanSemiclassicalCompletelyIntegrable2015a}. They also
obtained observability estimates, improving on previous work by Jaffard
\cite{jaffardControleInterneExact1990}, Macia
\cite{maciaSchrodingerFlowCompact2011} and Burq and Zworski
\cite{burqControlSchrodingerOperators2012}. 

Finally, we mention that the summability properties of the sequence of Fourier
coefficients of quantum limits of the Laplacian on the torus were studied,
first by Jakobson \cite{jakobsonQuantumLimitsFlat1997}, then by Aïssiou
\cite{aissiouSemiclassicalLimitsEigenfunctions2013a} and by McGrath
\cite{mcgrathVariationPrimeKtuples2022}.

\subsecnonewline{Structure of the paper} The Weyl and Anti-Wick quantization on $\R^{2d}$ and on the torus $\T^{2d}$ are presented in Section \ref{section_quantization}. The symplectic
automorphisms of the torus are quantized via the metaplectic representation and the link between the quantum propagators of quadratic
homogeneous Hamiltonians and the metaplectic operators is presented in Theorem \ref{thm_metaplectic_operators_quadratic_hamiltonians}. The Wigner
and Husimi distributions are introduced in Section
\ref{section_semiclassical_measures} and used to define semiclassical measures.
Technical lemmas about them needed in the rest of the paper are also
proved.

Section \ref{section_elliptic_matrices} is dedicated to elliptic
matrices and contains the proof of Theorem
\ref{thm_semiclassical_measures_elliptic_matrix} and of a characterization of
elliptic matrices, Proposition \ref{pro_characterization_elliptic_matrices}. A
few arithmetical lemmas needed for the proofs on semiclassical measures of
parabolic matrices are presented in Section \ref{section_arithmetic_interlude}. The final section \ref{section_parabolic_matrices} deals with parabolic
matrices. The quantizability conditions for them are first presented, then
Proposition \ref{pro_parabolic_nu_times_lebesgue} and Theorems 
\ref{thm_full_characterization_parabolic_dimension_1}, 
\ref{thm_upper_bound_quantum_variance_position_parabolic_matrices}, 
\ref{thm_lower_bound_quantum_variance_position_parabolic_matrices}, \ref{thm_parabolic_case_infinite_fourier_coef},  
\ref{thm_weird_measure_parabolic_isotropic} and \ref{thm_bound_sum_fourier_coeff_parabolic} are proved. Finally, the link between the
semiclassical measures of parabolic matrices and those of the Laplacian on
the torus is detailed in Proposition
\ref{pro_link_semiclassical_laplacian_P_2I}.

\subsecnonewline{Notations, definitions and conventions} We fix some notations, definitions and conventions that will be used in the
rest of the paper. The spaces $\R^{2d}$ and $\T^{2d}$ are endowed with a
symplectic structure given by the natural symplectic form $\sigma$ defined by

\begin{equation*}
	\sigma(z,z') = \sprod{Jz}{z'} \qqtq{for} z,z' \in \R^{2d} \text{ or } \T^{2d} \qqtq{where} J = \begin{pmatrix*}
		O & I \\ - I & 0
	\end{pmatrix*}.
\end{equation*}
The corresponding Heisenberg group of dimension $2d+1$, denoted $H_{2d+1}$, is $\R^{2d+1}$ endowed with the group law
\begin{equation*}
	(z, t) \cdot (z, t') = (z + z', t + t' + \frac{1}{2} \sigma(z,z')) \qqtq{for} z,z' \in \R^{2d}, \, t,t' \in \R.
\end{equation*}
Let $R$ be a ring. We denote by $M_d(R)$ the square matrices of size $d$ with coefficients in $R$, by $\GL_d(R)$ the invertible matrices in $M_d(R)$, by $\Sym_d(R)$ the symmetric matrices and by $\Sp_{2d}(R)$ the symplectic matrices. A matrix $A \in M_{2d}(R)$ is said \textit{symplectic} if it satisfies $A^T J A = J$.
Let $h >0$. The $h$-rescaled Fourier transform $\CF_h$ is defined by
\begin{equation*}
	\CF_h(a)(\xi) = \frac{1}{(2\pi h)^{d/2}} \int_{\R^{d}} e^{-\frac{i}{h} \sprod{x}{\xi}} a(x) \diff x \qqtq{for} a \in L^1(\R^{d}).
\end{equation*}
The set of Borel probability measures on a topological space $X$ is denoted $\CP(X)$. If $\phi : X \to X$ is a continuous transformation of $X$, the set of Borel probability measures invariant under $\phi$ is denoted $\CP_\phi(X)$. If $\mu$ is a measure on the torus $\T^d$, its Fourier coefficients are given by
\begin{equation*}
	\forall k \in \Z^d, \quad \widehat{\mu}(k) = \int_{\T^d} e^{-2 i \pi \sprod{k}{x}} \diff \mu(x).
\end{equation*}
Let $a \in \CC^\infty(\T^d)$ and $s \ge0$. The $H^s$ Sobolev norm of $a$ is defined as
\begin{equation*}
	\norm{a}_{H^s(\T^d)}^2 = \sum_{n \in \Z^d} (1 + \norm{n}^2)^{s} |a_n|^2 \qqtq{where} a = \sum_{n \in \Z^d} a_n e^{2 \pi i \sprod{n}{\cdot}}.
\end{equation*}
The Sobolev space $H^s(\T^d)$ is then defined as the completion of $\CC^\infty(\T^d)$ for this norm.

\subsecnonewline{Acknowledgement} I warmly thank my advisor Nalini Anantharaman for introducing me to quantum cat maps, and for her support and advice during the writing of this work. I am also grateful to Thomas Agugliaro for numerous discussions about arithmetic, quadratic forms and varieties. I used a Large Language Model, Coder Gwen 3.8 27B NVFP4, only to correct typos and English mistakes, find references in the literature and review the results once they were written.

\section{Quantization on the torus}\label{section_quantization}

We recall in this section the quantization procedure of integer symplectic
matrices on $\T^{2d}$ as introduced originally in
\cite{hannayQuantizationLinearMaps1980a} and extended by
\cite{bouzouinaEquipartitionEigenfunctionsQuantized1996a}. Our presentation
is based on the higher-dimensional version done in
\cite{riviereEntropySemiclassicalMeasuresSymplecticMapTorus2011} and
\cite{dyatlovSemiclassicalMeasuresHigherdimensional2024}. The goal is to build
a quantum version of the classical dynamics on $\T^{2d}$ and thus to construct
\begin{enumerate}
	\item a family of Hilbert spaces $(\CH_N)_{N \in \N^*}$ playing the role of the quantum phase space,
	\item a quantization procedure that associates to a classical observable $a \in \CC^\infty(\T^{2d})$ a family of operators $(\Op_N(a))_{N \in \N^*}$ acting on $(\CH_N)_{N \in \N^*}$,
	\item a family of operators $(M_N(\phi))_{N \in \N^*}$ acting on $(\CH_N)_{N \in \N^*}$ that quantize the classical dynamical system $\phi : \T^{2d} \to\T^{2d}$, meaning that they satisfy an Egorov theorem:
	\begin{equation*}
		\forall a \in \CC^\infty(\T^{2d}), \qquad \norm{M_N(\phi)^{-1} \Op_N(a) M_N(\phi) - \Op_N(a \circ \phi)}_{\CH_N} \xrightarrow[N \to + \infty]{} 0.
	\end{equation*}
\end{enumerate}

\subsection{Weyl quantization and Schrödinger representations of the Heisenberg group}

The Weyl quantization of a classical observable $a \in \CS(\R^{2d})$ is defined
as

\begin{equation*}
	\forall u \in L^2(\R^d), \; \forall x \in \R^d  \quad \left(\Op^w_h(a)u \right) (x) = \frac{1}{(2\pi h)^d} \int_{\R^{2d}} e^{\frac{i}{h}\sprod{x-y}{\xi}} a\left(\frac{x+y}{2}, \xi \right) u(y) \diff y \diff \xi.
\end{equation*}

\noindent This quantization procedure can be extended to a large class of symbols $S(1)$, defined by
\begin{equation*}
	S(1) = \set{(a_h)_{h >0} \in \CC^\infty(\R^{2d}) \; | \; \forall \alpha \in \N^{2d}, \, \exists C_\alpha \ge 0 \; \norm{\partial^\alpha a}_{\infty} \le C_\alpha },
\end{equation*}
on which it satisfies the Calderón-Vaillancourt theorem: there exists $C,M >0$ such that for all symbols $a \in S(1)$
\begin{equation}\label{Calderon-Vaillancourt_R2d}
	\norm{\Op^w_h(a)}_{L^2(\R^d \to \R^d)} \le C \sum_{|\alpha| \le M} h^{|\alpha|/2} \norm{\partial^\alpha a}_\infty.
\end{equation}
We refer the reader to \cite[chap. 4]{zworskiSemiclassicalAnalysis2012} for precise statements and proofs on the Weyl quantization.

Let $h >0$. The quantum translations $(U_h(z))_{z \in \R^{2d}}$ are unitary
operators on $L^2(\R^d)$ that act on a function $\psi \in L^2(\R^d)$ by
\begin{equation}\label{eq_action_quantum_translation_on_functions}
	\forall (x, \xi) \in \R^{2d}, \, \forall y \in \R^d, \quad (U_h(x,\xi) \psi)(y) = e^{\frac{i}{h}\sprod{y-x/2}{\xi}} \psi(y-x).
\end{equation}
They do not commute but they satisfy the relations
\begin{equation}\label{eq_commutation_quantum_translation}
	U_h(z) U_h(z') = e^{\frac{i}{2h}\sigma(z, z')} U_h(z + z') = e^{\frac{i}{h}\sigma(z, z')} U_h(z') U_h(z) \qtq{for} z,z' \in \R^{2d}.
\end{equation}
Hence, they generate a representation $\rho_h$ of the Heisenberg group $H_{2d+1}$ given by
\begin{equation*}
	\rho_h : (x, \xi, t) \in H_{2d+1} \mapsto e^{\frac{i}{h}t} U_h(x,\xi).
\end{equation*}

The representations $(\rho_h)_{h >0}$ are the irreducible
representations of the Heisenberg group which are non-trivial on the center.
This is the content of the Stone-von Neumann theorem and we refer the reader to
\cite[chap. 1]{follandHarmonicAnalysisPhase1989a} for more
details\footnote{We warn the reader that Folland chooses a different convention
	for the Schrödinger representation, see \cite[chap. 1, sec. 2, eq
		(1.27)]{follandHarmonicAnalysisPhase1989a} for a remark on this choice.}.
Finally, we remark that
\begin{equation}\label{eq_U_h(z)=quantif_e(isigma(z,.))}
	U_h(z) = \Op^w_h(e^{\frac{i}{h}\sigma(z, \cdot)}) \qqtq{for} z \in \R^{2d}.
\end{equation}

\subsection{Weyl quantization on the torus}

The Weyl quantization on the torus $\T^{2d}$ is built directly from the
Weyl quantization on $\R^{2d}$. The first step is to define the Hilbert space
of states. Intuitively, our quantum states should be invariant under quantum
phase space translations of length $1$ since the classical phase space is
invariant under such classical translations. Since two quantum states that
differ by a global phase are physically identical, we seek tempered
distributions $\psi \in \CS'(\R^d)$ such that
\begin{equation*}
	\forall w \in \Z^{2d} ,\; \exists \theta \in \R \quad U_h(w) \psi = e^{i \theta} \psi.
\end{equation*}

A necessary condition for such a tempered distribution $\psi$ to exist is that
all the operators $(U_h(w))_{w \in \Z^{2d}}$ must commute. According to the
commutation relations \eqref{eq_commutation_quantum_translation} of the $U_h$,
this implies that there exists an integer $N \in \N^*$ such that
\begin{equation*}
	\frac{1}{h} = 2 \pi N.
\end{equation*}
We suppose in the following that such a choice of $h >0$ has been made and we define for $\theta \in \T^{2d}$ the space $\CH_N(\theta)$ by
\begin{equation*}
	\CH_N(\theta) = \set{\psi \in \CS'(\R^d) \; | \; \forall w \in \Z^{2d}, \; \; U_h(w)\psi = e^{2 \pi i \sigma(\theta,w) + \pi i N P(w) } \psi}.
\end{equation*}
where $P : \R^{2d} \to \R$ is the polarization defined by
\begin{equation*}
	P(z) = \sprod{y}{\eta} \qqtqq{if} z = (y, \eta) \in \R^{2d}.
\end{equation*}

The important properties of $\CH_N(\theta)$ are explained in the following
theorem \cite[section
	2.2]{dyatlovSemiclassicalMeasuresHigherdimensional2024}.

\begin{theorem}\label{thm_proprietes_CH_N(theta)}
	Let $\theta = (\theta_x, \theta_\xi) \in \T^{2d}$. The dimension of $\CH_N(\theta)$ is $N^d$ and a basis is given by $(e^\theta_j)_{j \in \eint{0}{N-1}^d}$ where
	\begin{equation*}
		e^\theta_j(x) = \frac{1}{N^{d/2}} \sum_{k \in \Z^d} e^{-2 \pi i \sprod{k}{\theta_\xi}} \delta\left(x - \frac{j+kN - \theta_x}{N}\right).
	\end{equation*}
	If the scalar product on $\CH_N(\theta)$ is fixed such that $(e^\theta_j)_{j \in \eint{0}{N-1}^d}$ is an orthonormal basis, there is a unitary isomorphism $\Phi_N : L^2(\R^d) \simeq L^2(\T^{2d}, \CH_N(\theta))$ whose expression on $\CS(\R^d)$ is
	\begin{equation*}
		\Phi_N(f)(\theta) = \sum_{j \in \eint{0}{N-1}^d} \sprod{f}{e^\theta_j}_{L^2} e^\theta_j \qquad f \in \CS(\R^d), \; \theta \in \T^{2d}.
	\end{equation*}
\end{theorem}
The space $\CH_N(\theta)$ will therefore always be endowed with the scalar product making $(e^\theta_j)_{j\in \eint{0}{N-1}^d}$ an orthonormal basis. The existence of the unitary isomorphism $\Phi_N$ can be restated as an integral decomposition of $L^2(\R^d)$
\begin{equation}\label{integral_decomposition_L2_Rn}
	L^2(\R^d) = \int^{\oplus}_{\theta \in \T^{2d}} \CH_N(\theta) \diff \theta.
\end{equation}
The orthonormal basis $(e_j^\theta)_{j \in \eint{0}{N-1}^d}$ of $\CH_N(\theta)$ satisfies the relations
\begin{equation}\label{eq_transformation_rule_e_theta_j}
	e^{\theta + (y, \eta) }_j = e^\theta_{j-y} \qqtqq{and}
	e^\theta_{j+N y} = e^{2\pi i \sprod{y}{\theta_\xi}} e^\theta_j \qqtq{for} (y,\eta) \in \Z^{2d},
\end{equation}
and the quantum phase space translation acts on them by
\begin{equation}\label{eq_action_quantum_translation_e_theta_j}
	U_h(z) e^\theta_j = e^{2 \pi i \sprod{\xi}{j - \theta_x} + \pi i N P(z)}e^{\theta -N z}_j \qqtq{for} z = (x, \xi) \in \R^{2d}.
\end{equation}
This implies that
\begin{equation*}
	U_h(z) \CH_N(\theta) = \CH_N(\theta - Nz).
\end{equation*}
In particular, we find that $U_h(n/N) \CH_N(\theta) = \CH_N(\theta)$ for any $n \in \Z^{2d}$. The restriction of $U_h(n/N)$ to $\CH_N(\theta)$ is denoted $U_{N, \theta}(n/N)$, or $U_N(n/N)$ if the value of $\theta$ is clear from the context. Those operators act on the basis $(e^\theta_j)_{j \in \eint{0}{N-1}^d}$ by

\begin{equation}\label{eq_action_U_N((m,n)/N)_on_e_theta_j}
	U_{N}(\frac{m}{N}, \frac{n}{N}) \, e^\theta_j = e^{\frac{\pi i}{N} \sprod{n}{m-2 \theta_x}} e^{\frac{2 \pi i }{N} \sprod{n}{j}} e^\theta_{j+m} \qqtq{for} (m,n) \in \Z^{2d}.
\end{equation}
We record another useful relation.
\begin{equation}\label{eq_projection_on_torus_of_translation}
	\forall \theta \in \T^{2d}, \; n \in \Z^{2d}, \; g \in \CS(\R^d) \qquad \Phi_N (U_h(n) g) (\theta) = e^{2 \pi i \sigma(\theta,n) - \pi i N P(n)} \Phi_N(g) (\theta).
\end{equation}

We now turn to the quantization of observables on the torus. Any function $a
	\in \CC^\infty(\T^{2d})$ can be seen as a symbol in $S(1)$ on which we can use
the Weyl quantization. Using a Fourier decomposition of $a$ and the formula
\eqref{eq_U_h(z)=quantif_e(isigma(z,.))}, we get
\begin{equation}\label{eq_quantization_obs_torus}
	\Op^w_h(a) = \sum_{n \in \Z^{2d}} a_n U_h(\frac{n}{N}) \qqtqq{where} a(x) = \sum_{n \in \Z^{2d}} a_n e^{2 i \pi \sigma(n,x)}.
\end{equation}
The series of operators converges since the Fourier coefficients of $a$ satisfy $a_n = \CO(n^{-\infty})$ as $a \in \CC^{\infty}(\T^{2d})$. Since the operators $(U_h(n/N))_{n \in \Z^{2d}}$ stabilize $\CH_N(\theta)$, the Weyl quantization of $a$ can be restricted to those spaces. This restriction is denoted $\Op^w_{N, \theta}(a)$ or $\Op^w_N(a)$ if the value of $\theta$ is clear. The direct integral decomposition \eqref{integral_decomposition_L2_Rn} of $L^2(\R^d)$ implies \cite[thm. XIII.83]{reedMethodsModernMathematical4Operators1978}
\begin{equation*}
	\sup_{\theta \in \T^{2d}} \norm{\Op^w_{N, \theta}(a)}_{\CL(\CH_N(\theta))} = \norm{\Op^w_h(a)}_{\CL(L^2(\R^d))}.
\end{equation*}
This equality and the Calderón-Vaillancourt theorem \eqref{Calderon-Vaillancourt_R2d} lead to a bound on the operator norm of the Weyl quantization on the torus:
\begin{equation}\label{Calderon-Vaillancourt_torus}
	\forall a \in \CC^\infty(\T^{2d}), \quad \norm{\Op^w_{N, \theta}(a)}_{\CL(\CH_N(\theta))} \le C \sum_{|\alpha| \le M} N^{-|\alpha|/2} \norm{\partial^\alpha a}_{\infty}.
\end{equation}

The Weyl quantization on the torus of an observable $a \in \CC^\infty(\T^{2d})$
could also have been defined with the relation
\begin{equation*}
	\forall g \in L^2(\T^{2d}, \CH_N(\theta)) \; \forall \theta \in \T^{2d} \quad \Op^W_{N, \theta}(a) g (\theta) = \Phi_N \Op^W_h(a) \Phi_N^* g \, (\theta),
\end{equation*}
as is done in \cite{dyatlovSemiclassicalMeasuresHigherdimensional2024}.

We compute the trace of the quantum translation on $\CH_N(\theta)$ and the
Hilbert-Schmidt norm of $\Op^w_N(a)$.
\begin{lemma}\label{lem_trace_quantum_translation_HS_norm_Weyl_quantization}
	Let $n \in \Z^{2d}$ and $a \in \CC^\infty(\T^{2d})$. We have
	\begin{equation*}
		\tr(U_N(\frac{n}{N})) = \begin{syst}
			&e^{\frac{2 \pi i}{N} \sigma(\theta, n) + \frac{\pi i}{N} P(n)} N^d \qtq{if} n \equiv 0 \quad \modulo{N} ,\\
			&0 \qtq{otherwise}
		\end{syst},
	\end{equation*}
	and the Hilbert-Schmidt norm of $\Op^w_N(a)$ is given by
	\begin{equation*}
		\frac{\norm{\Op^w_N(a)}_{HS}^2}{N^d} = \frac{\tr(\Op^w_N(a)^* \Op^w_N(a))}{N^d} = \sum_{\substack{n,k \in \Z^{2d} ,\\ n \equiv k \modulo{N}}} \overline{a_{k}} a_{n} e^{\frac{\pi i}{N} \sigma(2\theta-n, n-k) + \frac{\pi i}{N} P(n-k)}.
	\end{equation*}
\end{lemma}
\begin{proof}
	Write $n=(k,l)$. We deduce from \eqref{eq_transformation_rule_e_theta_j} and \eqref{eq_action_U_N((m,n)/N)_on_e_theta_j} that
	\begin{align*}
		\tr(U_N(\frac{k}{N}, \frac{l}{N})) & = \sum_{j \in \eint{0}{N-1}^d} \sprod{U_N(\frac{k}{N}, \frac{l}{N})e^\theta_j}{e^\theta_j} = \sum_{j \in \eint{0}{N-1}^d} e^{\frac{\pi i}{N} \sprod{l}{k-2 \theta_x}} e^{\frac{2 \pi i }{N} \sprod{l}{j}} \sprod{e^\theta_{j+k}}{e^{\theta}_j} ,\\
		& = \mathbb{1}_{N \Z^d}(k) e^{2 \pi i \sprod{\frac{k}{N}}{\theta_\xi}} e^{\frac{\pi i}{N} \sprod{l}{k-2 \theta_x}} \sum_{j \in \eint{0}{N-1}^d} e^{\frac{2 \pi i}{N} \sprod{l}{j}} ,\\
		& = \mathbb{1}_{N \Z^{2d}}(k,l) e^{\frac{2 \pi i}{N} \sigma(\theta, (k,l)) + \frac{\pi i}{N} \sprod{k}{l}} N^d.
	\end{align*}
	Write $a = \sum_{n \in \Z^{2d}} a_{n} e^{2 \pi i \sigma(n, \cdot)}{}$. The formulas \eqref{eq_commutation_quantum_translation} and \eqref{eq_quantization_obs_torus} imply
	\begin{align*}
		\tr(\Op^w_N(a)^* \Op^w_N(a)) & = \sum_{n, k \in \Z^{2d}} \overline{a_{k}} a_{n} \tr(U_N(\frac{k}{N})^* U_N(\frac{n}{N})) ,\\
		& = \sum_{n, k \in \Z^{2d}} \overline{a_{k}} a_{n} e^{\frac{\pi i}{N} \sigma(-k,n)} \tr(U_N(\frac{n-k}{N})).
	\end{align*}
	The computation of the trace of the quantum translation done above then gives
	\begin{equation*}
		\tr(\Op^w_N(a)^* \Op^w_N(a)) = \sum_{n,k \in \Z^{2d}} \overline{a_{k}} a_{n} e^{\frac{\pi i}{N} \sigma(n,k)} \mathbb{1}_{N \Z^{2d}}(n-k) e^{\frac{2 \pi i}{N} \sigma(\theta, n-k) + \frac{\pi i}{N} P(n-k)}.
	\end{equation*}
	We simplify the phase to obtain
	\begin{equation*}
		\frac{\tr(\Op^w_N(a)^* \Op^w_N(a))}{N^d} = \sum_{\substack{n,k \in \Z^{2d} ,\\ n \equiv k \modulo{N}}} \overline{a_{k}} a_{n} e^{\frac{\pi i}{N} \sigma(2\theta-n, n-k) + \frac{\pi i}{N} P(n-k)}.
	\end{equation*}
\end{proof}

\subsection{Anti-Wick quantization}

We will also use another quantization based on coherent states, the
Anti-Wick quantization. We refer the reader to
\cite{robertCoherentStatesApplications2021} for a presentation of coherent
states in mathematical physics. The standard coherent states are Gaussian
functions centered around a point in phase space. More precisely, the coherent
state centered at $0$ in $\R^{2d}$, denoted $\varphi_{0,h}$, is given by
\begin{equation*}
	\varphi_{0,h}(x) = \frac{1}{(\pi h)^{d/4}}e^{-\norm{x}^2/2h}.
\end{equation*}
and its translates centered at $\rho \in \R^{2d}$ are obtained via the action of the quantum translation $(U_h(\rho))_{\rho \in \R^{2d}}$
\begin{equation*}
	\varphi_{\rho,h} = U_h(\rho) \varphi_{0,h} \qqtq{for} \rho \in \R^{2d}.
\end{equation*}
On the one hand, they are not orthogonal and their overlap is
\begin{equation}\label{eq_overlap_coherent_states}
	\sprod{\varphi_{\rho,h}}{\varphi_{\rho',h}} = e^{\frac{i}{2h} \sigma(\rho,\rho')} e^{-\frac{\norm{\rho - \rho'}^2}{4h}}.
\end{equation}
This can be obtained by direct computation of Gaussian integrals or in a more conceptual fashion using the creation and annihilation operators for the harmonic oscillator \cite[see][sec. 1.3.1 ,prop. 2]{robertCoherentStatesApplications2021}.

On the other hand, they satisfy a resolution of the identity
\begin{equation}\label{eq_resolution_identity_R2d}
	Id_{L^2(\R^d)} = \int_{\R^{2d}} \ket{\varphi_{\rho,h}}\bra{\varphi_{\rho,h}} \frac{\diff \rho}{(2\pi h)^d},
\end{equation}
where $\ket{\varphi_{\rho,h}}\bra{\varphi_{\rho,h}}$ is the rank-1 projector on $\varphi_{\rho,h}$ defined by
\begin{equation*}
	\forall u \in L^2(\R^d), \qquad \ket{\varphi_{\rho,h}}\bra{\varphi_{\rho,h}} u = \sprod{u}{\varphi_{\rho,h}} \varphi_{\rho,h}.
\end{equation*}

Let $a \in S(1)$, its Anti-Wick quantization is defined as
\begin{equation}\label{eq_def_Anti-Wick_R2d}
	\Op^{AW}_h(a) = \int_{\R^{2d}} a(\rho) \ket{\varphi_{\rho,h}}\bra{\varphi_{\rho,h}} \frac{\diff \rho}{(2\pi h)^d}.
\end{equation}
The main feature of this quantization is positivity, meaning that if $a \ge 0$ is a positive bounded symbol then its Anti-Wick quantization $\Op^{AW}_h(a)$ is a positive self-adjoint operator on $L^2(\R^d)$.

The Anti-Wick quantization is linked to the Weyl one by the formula
\begin{equation}\label{eq_link_anti_wick_weyl}
	\forall a \in S(1), \quad \Op^{AW}_h(a) = \Op^w_h(a * G_h) \qqtq{where} G_h(x, \xi) = \frac{1}{(\pi h)^d} e^{-\frac{1}{h}(\norm{x}^2 + \norm{\xi}^2)}.
\end{equation}
This expression and the Calderón-Vaillancourt \eqref{Calderon-Vaillancourt_R2d} allow one to prove that
\begin{equation}\label{eq_bound_Wey_vs_Anti_Wick_R2d}
	\forall a \in S(1), \qquad \norm{\Op^{AW}_h(a) - \Op^w_h(a)}_{\CL(L^2(\R^d))} = O_a(h).
\end{equation}

To descend the Anti-Wick quantization to the torus, one simply applies the
unitary isomorphism $\Phi_N$ between $L^2(\R^d)$ and $L^2(\T^{2d},
	\CH_N(\theta))$ given by Theorem \ref{thm_proprietes_CH_N(theta)}. We project
the coherent states on $\R^{2d}$ onto the torus and define
\begin{equation*}
	\varphi_{\rho,N}^\theta = \Phi_N \left(\varphi_{\rho, (2\pi N)^{-1}}\right)(\theta) \in \CH_N(\theta) \qqtq{for} \theta \in \T^{2d}, \; \rho \in \R^{2d}.
\end{equation*}
The relation \eqref{eq_projection_on_torus_of_translation} implies that, up to a phase, the coherent states $\varphi_{\rho,N}^\theta$ only depend on the equivalence class of $\rho$ modulo $\Z^{2d}$.

This leads to the definition of the Anti-Wick quantization of an observable $a
	\in \CC^\infty(\T^{2d})$ as
\begin{equation*}
	\Op^{AW}_{N, \theta}(a) = \int_{\T^{2d}} a(\rho) \ket{\varphi_{\rho,N}^\theta} \bra{\varphi_{\rho,N}^\theta} N^d \diff \rho \quad \in \CL(\CH_N(\theta)).
\end{equation*}
As usual, we will denote $\varphi_{\rho,N}^\theta$ as $\varphi_{\rho,N}$ and $\Op^{AW}_{N, \theta}(a)$ as $\Op^{AW}_N(a)$ when $\theta$ is clear from the context. This quantization is again positive.

Let $\theta \in \T^{2d}$. The resolution of identity
\eqref{eq_resolution_identity_R2d} descends to the quantum space $\CH_N(\theta)$
leading to
\begin{equation}\label{eq_resolution_identity_torus}
	Id_{\CH_N(\theta)} = \int_{\T^{2d}} \ket{\varphi_{\rho,N}^\theta} \bra{\varphi_{\rho,N}^\theta} N^d \diff \rho.
\end{equation}
This equality implies the following bound on the $L^2$ operator norm of the Anti-Wick quantization of an observable
\begin{equation}\label{ineq_bound_norm_Anti-Wick_torus}
	\forall a \in \CC(\T^{2d}), \qquad \norm{\Op^{AW}_{N, \theta}(a)}_{\CL(\CH_N(\theta))} \le \norm{a}_\infty.
\end{equation}
The difference between the Anti-Wick and the Weyl quantization on the torus is of order $1/N$ as a consequence of \eqref{eq_bound_Wey_vs_Anti_Wick_R2d}
\begin{equation}\label{eq_bound_Wey_vs_Anti_Wick_torus}
	\forall a \in \CC^\infty(\T^{2d}) \quad \norm{\Op^{AW}_N(a) - \Op^w_N(a)}_{\CL(\CH_N(\theta))} = O_a(\frac{1}{N}).
\end{equation}

We present estimates on the norm and on the overlap of the coherent states
$\varphi_{\rho,N}$ on the torus that are a direct transposition to higher
dimensions of the estimates for $d=1$ of \cite[sec. 6.4, prop.
	61]{robertCoherentStatesApplications2021}.
\begin{lemma}[Overlap and norm of coherent states on the torus]\label{lem_prop_coherent_states_torus}
	Let $\rho, \rho' \in \R^{2d}$, $\theta \in \T^{2d}$ and $N \in \N^*$. Then
	\begin{equation*}
		|\sprod{\varphi^\theta_{\rho, N}}{\varphi^\theta_{\rho',N}}| \le C_d e^{-\frac{\pi N}{2} d_{\T^{2d}}(\rho, \rho')^2} \qqtqq{and} \norm{\varphi^\theta_{\rho', N}}^2 = 1 + \CO_d(e^{-\frac{\pi}{4}N}).
	\end{equation*}
	The constants depend only on the dimension $d$ and $d_{\T^{2d}}(\rho,\rho')$ is the distance on $\T^{2d}$ defined by
	\begin{equation*}
		d_{\T^{2d}}(\rho,\rho') = \inf_{m \in \Z^{2d}} \norm{\rho - \rho' -m}.
	\end{equation*}
\end{lemma}

\begin{proof}
	The proof for $d=1$ of \cite[sec. 6.4, prop. 61]{robertCoherentStatesApplications2021} adapts mutatis mutandis to higher dimensions. There is however a small error in the expression for the overlap of coherent states used in the proof and accounting for this error leads to the lemma above, with no $\sqrt{N}$ and a different exponent. Note that we only need the case $\gamma = i$ of the coherent states defined in \cite{robertCoherentStatesApplications2021}.
\end{proof}

\subsection{Symplectic automorphisms of the torus and metaplectic representation}

Let $A \in \Sp_{2d}( \R)$ be a real symplectic matrix. The map defined
by
\begin{equation*}
	(z, t) \in H_{2d+1} \mapsto (Az,t) \in H_{2d+1},
\end{equation*}
is an automorphism of the Heisenberg group that preserves the center $\set{(0,t) \in H_{2d+1} \; | \: t \in \R}$. We continue to denote it $A$. We can twist the irreducible representation $\rho_h$ by $A$ to get the representation $\rho_h \circ A$. It is still irreducible and acts on the center in the same manner as $\rho_h$. As a consequence of the Stone-von Neumann theorem, $\rho_h$ and $\rho_h \circ A$ are conjugate, meaning that there exists a unitary operator $M_h(A)$ on $L^2(\R^d)$ such that
\begin{equation*}
	\forall z \in \R^{2d}, \; \forall t \in \R, \qquad \rho_h (Az, t) = M_h(A) \rho_h(z,t) M_h(A)^{-1}.
\end{equation*}
This relation is equivalent to
\begin{equation}\label{eq_exact_Egorov_R2d}
	\forall z \in \R^{2d}, \qquad U_h (Az) = M_h(A) U_h(z) M_h(A)^{-1},
\end{equation}
which is an exact Egorov theorem for the metaplectic operators.
In other words, the operator $M_h(A)$ acting on $L^2(\R^d)$ is the quantum analogue of the diffeomorphism $A$ acting on $\R^{2d}$.
According to Schur's lemma, the operator $M_h(A)$ is unique up to a phase and the morphism
\begin{equation*}
	A \in \Sp_{2d}( \R) \mapsto M_h(A) \in \CU(L^2(\R^d))/\set{cId \, | \,  |c|=1 },
\end{equation*}
is thus a projective representation of the symplectic group. This representation can be lifted into a representation of the metaplectic group, which is the double cover of the symplectic group, and is generally called the metaplectic representation. We refer the reader to \cite[chap. 4]{follandHarmonicAnalysisPhase1989a} for more information on the metaplectic representation.

The metaplectic operators associated to a set of generators of the symplectic
group admit simple formulas.
\begin{lemma}[Metaplectic operators for generators of the symplectic group]\label{lem_metaplectic_op_generator_symp_group}
	Let $E \in \GL_d(\R)$, $B \in \Sym_d(\R)$ and $h >0$. We have, for
	\begin{align*}
		M_h\left(\begin{pmatrix}
				         E & 0 \\ 0 & E^{-T}
			         \end{pmatrix}\right)u = u \circ E^{-1} \quad & \quad M_h\left(\begin{pmatrix}
				                                                                       I & 0 \\ B & I
			                                                                       \end{pmatrix}\right)u(x) = e^{\frac{i}{2h} \sprod{Bx}{x}} u(x) \\ M_h\left(J\right) = \CF_h&
	\end{align*}
	The equalities are to be understood modulo a multiplication by $\pm Id$.
\end{lemma}
\begin{proof}
	For $A \in \Sp_{2d}(\R)$, a representative $M_h(A)$ of the metaplectic operator associated to $A$ is characterized by the relation \eqref{eq_exact_Egorov_R2d}. Thus, it suffices to check that the expression in the lemma verifies it. The explicit computations are done at the beginning of \cite[chap. 4, sec. 2]{follandHarmonicAnalysisPhase1989a}.
\end{proof}

There is a direct link between metaplectic operators and the quantum
propagator solving the Schrödinger equation associated to quadratic
Hamiltonians. The next theorem is a reformulation of statements in \cite[sec.
	11.3]{zworskiSemiclassicalAnalysis2012} and we refer the reader to it for a
different construction of the metaplectic operators via quantum propagators.

\begin{theorem}[Metaplectic operators and quadratic Hamiltonians]\label{thm_metaplectic_operators_quadratic_hamiltonians}
	Let $A_1 \in \GL_d(\R)$ and $A_2, A_3 \in \Sym_d(\R)$. Define $K=\exp(A) \in \Sp_{2d}(\R)$ and a quadratic Hamiltonian $p \in \CC^\infty(\R^{2d})$ by
	\begin{equation*}
		p(x, \xi) = \frac{1}{2} \sprod{A_2\xi}{\xi} + \sprod{A_1x}{\xi} + \frac{1}{2} \sprod{A_3x}{x} \qtq{and} A = \begin{pmatrix}
			A_1 & A_2 \\ A_3 & A_1^{-T}
		\end{pmatrix}.
	\end{equation*}
	Then the equation below has a unique continuous solution $(F_t)_{t \in \R}$ of unitary operators on $L^2(\R^d)$ and this solution verifies $F(t) = \pm M_h(K^t)$ for all $t \in \R$.
	\begin{equation*}
		\frac{h}{i} \frac{d}{dt} F(t) + F(t) \Op^w_h(p) = 0 \qqtq{with} F(0) = Id_{L^2(\R^d)}.
	\end{equation*}
\end{theorem}
\begin{proof}
	Note that $JA + A^TJ = 0$. As this equation defines the tangent space of the symplectic group at the identity, we deduce that $K \in \Sp_{2d}(\R)$.
	The existence and uniqueness of $F$ solving the equation of the theorem is exactly \cite[lem. 11.8]{zworskiSemiclassicalAnalysis2012}. Then the proof of \cite[thm. 11.9]{zworskiSemiclassicalAnalysis2012} shows that for any $t \in \R$, the quantum propagator $F$ satisfies
	\begin{equation*}
		\forall a \in S(1), \qquad F(t)^{-1} \Op^w_h(a) F(t) = \Op^w_h(a \circ K^t).
	\end{equation*}
	Taking $a = e^{\frac{i}{h}\sigma(z, \cdot)}$ and using \eqref{eq_U_h(z)=quantif_e(isigma(z,.))} gives
	\begin{equation*}
		\forall z \in \R^{2d}, \qquad F(t)^{-1} U_h(z) F(t) = \Op^w_h(e^{\frac{i}{h}\sigma(z, K^t\cdot)}) = \Op^w_h(e^{\frac{i}{h}\sigma(k^{-t}z, \cdot)}) = U_h(K^{-t}z).
	\end{equation*}
	Replacing $t$ by $-t$, we see that this is exactly \eqref{eq_exact_Egorov_R2d}, which means that $F(t) = M_h(K^t)$.
\end{proof}

Since any matrix with integer coefficients stabilizes the torus $\T^{2d}$, any
symplectic integer matrix $A \in \Sp_{2d}( \Z)$ induces an automorphism of
$\T^{2d}$, still denoted $A$. It is quantized by restricting the quantum
propagator $M_h(A)$ to $\CH_N(\theta)$, where $\theta$ is chosen to ensure that
$\CH_N(\theta)$ is invariant under $M_h(A)$. According to \cite[sec.
	2.2.4]{dyatlovSemiclassicalMeasuresHigherdimensional2024}, such a choice can be
made as long as $1$ is not an eigenvalue of $A$ or $N$ is even.
\begin{lemma}[Quantizability conditions for metaplectic operators]\label{lem_conditions_quantizability_metaplectic_operators}
	Let $A \in \Sp_{2d}( \Z)$. There exists a unique $\phi_A \in \set{0,1}^{2d}$ such that
	\begin{equation*}
		M_{h}(A) \CH_N(\theta) = \CH_N(A\theta + \frac{N \phi_A}{2}) \qqtqq{where} h = \frac{1}{2 \pi N}.
	\end{equation*}
	The angle $\phi_A$ is characterized by the relation
	\begin{equation*}
		\forall w \in \Z^{2d}, \qquad P(A^{-1}w) - P(w) \equiv \sigma(\phi_A,w) \quad \modulo{2\Z}.
	\end{equation*}
	In particular,
	\begin{itemize}
		\item If $N$ is even, then $\CH_N(0)$ is invariant under $M_h(A)$.
		\item If $1$ does not belong to the spectrum of $A$, then, for any $N \in \N^*$,
		      there exists an angle $\theta_N(A) \in \T^{2d}$ such that $\CH_N(\theta_N(A))$
		      is invariant under $M_h(A)$
	\end{itemize}
\end{lemma}

Applying this lemma to the matrix $J$ gives $\phi_J = 0$ and thus
\begin{equation}\label{eq_action_J_CH_N(theta)}
	M_h(J) \CH_N(\theta) = \CH_N(J\theta) \qqtq{for} \theta \in \T^{2d}, \, h = \frac{1}{2 \pi N}.
\end{equation}

The restriction of $M_h(A)$ to $\CH_N(\theta_N(A))$ is denoted $M_N(A)$. It
satisfies an exact Egorov theorem for the Weyl quantization as a direct
consequence of \eqref{eq_exact_Egorov_R2d}
\begin{equation}\label{eq_exact_Egorov_torus}
	\forall n \in \Z^{2d}, \qquad U_N (\frac{An}{N}) = M_N(A) U_N(\frac{n }{N}) M_N(A)^{-1},
\end{equation}
which, together with \eqref{eq_quantization_obs_torus}, implies
\begin{equation}\label{eq_exact_Egorov_observable_torus}
	\forall a \in \CC^\infty(\T^{2d}), \qquad M_N(A)^{-1} \Op^w_N(a) M_N(A) = \Op^w_N(a \circ A).
\end{equation}

We highlight that the exactness of the Egorov theorem for $A$ is specific to
the Weyl quantization. The Anti-Wick quantization satisfies an Egorov theorem
but with an error of order $1/N$:
\begin{equation*}
	\forall a \in \CC^\infty(\T^{2d}), \qquad \norm{M_N(A) \Op^{AW}_N(a) M_N(A)^{-1} - \Op^{AW}_N(a \circ A)}_{\CL(\CH_N(\theta))} = O_{a}(\frac{1}{N}).
\end{equation*}

\section{Semiclassical measures}\label{section_semiclassical_measures}

The main objects of interest of this article are semiclassical measures. They
are measures on $\T^{2d}$ that encode the asymptotic behavior
of a sequence of quantum states $(\psi_N)_N \in \CH_N(\theta_N)$ in the
semiclassical limit $N \to \infty$. Let $(\phi_{N})_N$ be another sequence of
quantum states in $\CH_N(\theta_N)$. We associate to them two sequences of
distributions corresponding to either the Weyl or the Anti-Wick quantization
and given by
\begin{equation*}
	\forall a \in \CC^\infty(\T^{2d}), \qquad \begin{aligned}
		Hus(\psi_N, \phi_N)(a) & = \sprod{\Op^{AW}_N(a)\psi_N}{\phi_N} ,\\
		Wig(\psi_N, \phi_N)(a) & = \sprod{\Op^w_N(a)\psi_N}{\phi_N}.
	\end{aligned}
\end{equation*}
These distributions are called the Husimi and Wigner distributions of the pair of quantum states $(\psi_N, \phi_N)$. The Husimi distribution is equal to
\begin{equation*}
	Hus(\psi_N, \phi_N)(a) = \int_{\T^{2d}} a(\rho) \sprod{\psi_N}{\varphi_{\rho,N}^{\theta_N}} \overline{\sprod{\phi_N}{\varphi_{\rho,N}^{\theta_N}}} N^d \diff \rho \qqtq{for} a \in \CC^\infty(\T^{2d}),
\end{equation*}
and can thus be extended to a signed finite measure on $\T^{2d}$ thanks to \eqref{eq_resolution_identity_torus}. The Calderón-Vaillancourt theorem \eqref{Calderon-Vaillancourt_torus} implies that the Wigner distribution is of finite order
\begin{equation}\label{ineq_wigner_distribution_finite_order}
	\forall a \in \CC^\infty(\T^{2d}), \qquad |Wig(\psi_N, \phi_N)(a)| \le \norm{\psi_N} \norm{\phi_N} C \sum_{|\alpha| \le M} N^{-|\alpha|/2} \norm{\partial^\alpha a}_{\infty}.
\end{equation}
We denote by $Hus(\psi_N)$ and $Wig(\psi_N)$ the distributions $Hus(\psi_N, \psi_N)$ and $Wig(\psi_N, \psi_N)$. The Husimi distribution $Hus(\psi_N)$ is a positive measure on $\T^{2d}$.

\begin{remark}
	The fact that the Husimi distribution is a positive measure while the Wigner distribution is signed is a manifestation of the uncertainty principle. The formula \eqref{eq_link_anti_wick_weyl} tells us that the Anti-Wick quantization is the Weyl quantization where the classical observable has been smoothed with a Gaussian of variance $\sqrt{h}/2$. This is just enough to evade the uncertainty principle that prevents concentration on sets of characteristic size $\sqrt{h}$ in phase space. We point the reader to the discussion in \cite[chap. 1 sec. 8]{follandHarmonicAnalysisPhase1989a} about the positivity of the Wigner distribution and a remarkable theorem of Hudson \cite{hudsonWhenWignerQuasiprobability1974} that characterizes the functions with a non-negative Wigner distribution.
\end{remark}

We suppose in the following that the norm of the quantum states $(\psi_N)_N$ is
bounded independently of $N$. The set of semiclassical measures associated to
$(\psi_N)_N$, noted $\CS \CM ((\psi_N)_N)$, is the set of accumulation points
of the sequence $(Hus(\psi_N))_N$ in the space of finite measures on $\T^{2d}$.
It is a non-empty compact subset of the set of finite measures on $\T^{2d}$. We
obtain the same set if we consider the accumulation points of the sequence
$(Wig(\psi_N))_N$ in the space of distributions, thanks to
\eqref{eq_bound_Wey_vs_Anti_Wick_torus} which implies
\begin{equation*}
	\forall a \in \CC^\infty(\T^{2d}), \qquad |Hus(\psi_N)(a) - Wig(\psi_N)(a)| = O_a \left( \frac{1}{N} \right).
\end{equation*}
We record that for any finite measure $\mu$ on $\T^{2d}$ and any increasing sequence of integers $(N_j)_{j}$, we have the equivalence
\begin{equation}\label{equivalence_semiclassical_measure_Wigner_Husimi}
	Hus(\psi_{N_j}) \rightharpoonup^* \mu \iff Wig(\psi_{N_j}) \rightharpoonup^* \mu.
\end{equation}
Note that the convergence on the left is for the weak-star topology on the space of finite measures on $\T^{2d}$, while the convergence on the right is the weak-star convergence of distributions.

To illustrate the notion of semiclassical measure, we prove that the only
semiclassical measure associated to the coherent state on the torus
$\varphi^{\theta_N}_{\rho_0,N}$ is the Dirac measure at $\rho_0$.
\begin{lemma}[Semiclassical measure of coherent states]\label{lem_semiclassical_measure_coherent_states}
	Let $\rho_0 \in \T^{2d}$ and $(\theta_N)_N$ a sequence of elements of $\T^{2d}$. Then
	\begin{equation*}
		\CS \CM (\varphi^{\theta_N}_{\rho_0,N}) = \set{\delta_{\rho_0}}.
	\end{equation*}
\end{lemma}

\begin{proof}
	The angles $\theta_N$ will not play any role in the proof and we will not mention them. We deduce from Lemma \ref{lem_prop_coherent_states_torus} that
	\begin{equation*}
		\norm{\varphi_{\rho_0,N}} \xrightarrow[N \to + \infty]{} 1.
	\end{equation*}
	Hence, the norms of the states $(\varphi_{\rho_0,N})_N$ are bounded independently of $N$ by some $M >0$, the positive measure $(Hus(\varphi_{\rho_0,N}))_N$ has mass at most $M$ and any semiclassical measure of $(\varphi_{\rho_0,N})_N$ has mass $1$.
	Let $\mu$ be a semiclassical measure for $\varphi_{\rho_0,N}$. This means that, up to a subsequence that we do not mention, we have
	\begin{equation*}
		Hus(\varphi_{\rho_0,N}) \xrightharpoonup[N \to \infty]{} \mu.
	\end{equation*}
	We prove that the support of $\mu$ is restricted to $\set{\rho_0}$, which implies that $\mu = \delta_{\rho_0}$ as $\mu$ is a probability measure.

	Let $\chi \in \CC_c(\T^{2d}\setminus \set{\rho_0})$. In particular, there
	exists $\epsilon >0$ such that the support of $\chi$ is included in the
	complement of the ball $B(\rho_0,\epsilon)$. Lemma
	\ref{lem_prop_coherent_states_torus} gives
	\begin{align*}
		|Hus(\varphi_{\rho_0,N})(\chi)| & = \int_{\T^{2d}} |\chi(\rho)| |\sprod{\varphi_{\rho_0,N}}{\varphi_{\rho,N}}|^2 N^d \diff \rho ,\\
		& \le \norm{\chi}_{\infty} \int_{\T^{2d} \setminus B(\rho_0,\epsilon)}  e^{-\frac{\pi N}{2} d_{\T^{2d}}(\rho_0, \rho)^2} N^d \diff \rho ,\\
		& \le C_d \norm{\chi}_{\infty} e^{-\frac{\pi N}{2} \epsilon^2 } N^d \xrightarrow[N \to + \infty]{} 0.
	\end{align*}
	Therefore $\mu(\chi) = 0$. As this is the case for any continuous function compactly supported in $\T^{2d} \setminus \set{\rho_0}$, we conclude that the support of $\mu$ is included in $\set{\rho_0}$.
	This proves that the only accumulation point of the sequence $(Hus(\varphi_{\rho_0,N}))_{N}$ is $\delta_{\rho_0}$. This concludes the proof as the space of positive measures on $\T^{2d}$ of mass at most $M$ is compact.
\end{proof}

The metaplectic operators act as expected on semiclassical measures thanks to
the Egorov theorem.
\begin{lemma}\label{lem_action_metaplectic_semiclassical_measures}
	Let $A \in \Sp_{2d}( \Z)$ and $(\psi_{N})_N$ a sequence of normed vectors in $\CH_{N}(\theta_N)$. Then, for any probability measure $\mu \in \CP(\T^{2d})$ and any $s \in \Z$, we have
	\begin{equation*}
		Wig(\psi_{N}) \xrightharpoonup[N \to  + \infty]{}^* \mu \implies Wig(M_{(2 \pi N)^{-1}}(A)^s \psi_{N}) \xrightharpoonup[N \to  + \infty]{}^* A^s_{\#} \mu.
	\end{equation*}
\end{lemma}

\begin{proof}
	Let $a \in \CC^\infty(\T^{2d})$. The exact Egorov theorem for the Weyl quantization on the torus \eqref{eq_exact_Egorov_observable_torus} leads to
	\begin{align*}
		Wig(M_{(2 \pi N)^{-1}}(A)^s \psi_N)(a) & = \sprod{\Op^w_N(a) M_{(2 \pi N)^{-1}}(A)^s \psi_N}{M_{(2 \pi N)^{-1}}(A)^s \psi_N},\\
		& = \sprod{M_{(2 \pi N)^{-1}}(A)^s \Op^w_N(a) M_{(2 \pi N)^{-1}}(A)^s \psi_N}{ \psi_N} ,\\
		& = \sprod{\Op^w_N(a \circ A^s) \psi_N}{ \psi_N} ,\\
		& \xrightharpoonup[N \to  + \infty]{}^* \mu(a \circ A^s) = (A^s_{\#} \mu)(a),
	\end{align*}
	which implies the lemma.
\end{proof}

We now present the definition of the semiclassical measure associated to a
symplectic integer matrix. Let $A \in \Sp_{2d}(\Z)$ be a symplectic matrix with
integer coefficients. A measure $\mu \in \CP(\T^{2d})$ is a semiclassical
measure associated to $A$ if there exist an increasing sequence of integers
$(N_k)_k$, a sequence of angles $(\theta_k)_k$ in $\T^{2d}$ and a sequence
$(\psi_k)_k$ of normed eigenvectors of $M_{N_k}(A)$ on $\CH_{N_k}(\theta_k)$
such that
\begin{equation*}
	Wig(\psi_k) \xrightharpoonup[k \to \infty]{} \mu.
\end{equation*}
The set of all semiclassical measures associated to $A$ is denoted $\CS \CM(A)$.

The next lemma records a few simple facts about the semiclassical measures
associated to a symplectic matrix that will be used extensively in the rest of
the paper.
\begin{lemma}[Basic properties of $\CS \CM(A)$]\label{lem_basic_pro_semicla_meas_symp_matrix}
	Let $A \in \Sp_{2d}(\Z)$, then
	\begin{equation*}
		\emptyset \neq \CS \CM(A) \subset \CP_A(\T^{2d}).
	\end{equation*}
	Moreover, $\CS \CM(A)$ is weak-star closed in the set of probability measures on $\T^{2d}$ and thus compact for the weak-star topology. For any $0 \le s \le t \in \N$, we have
	\begin{equation*}
		\CS \CM(A) = \CS \CM(A^{-1}) \qqtqq{and} \CS \CM(A^s) \subset \CS \CM(A^t).
	\end{equation*}
	Finally, for any $C \in \Sp_{2d}(\Z)$, we have
	\begin{equation*}
		\CS \CM (C A C^{-1}) = C_{\#} \CS \CM(A).
	\end{equation*}
\end{lemma}
\begin{proof}
	Denote $E_N(A)$ the set of normed eigenvectors of $M_N(A)$ in $\CH_N(\theta)$ for any $\theta$ for which $\CH_N(\theta)$ is stable under $M_N(A)$. This set is non-empty if $N$ is even as $\CH_N(0)$ is then stable under $M_N(A)$ according to Lemma \ref{lem_conditions_quantizability_metaplectic_operators} and the diagonalizability of unitary operators on finite dimensional spaces. The definition of the set of semiclassical measures associated to $A$ and the usual set description of accumulation points of sequences gives
	\begin{equation*}
		\CS \CM(A) = \cap_{k \ge 0} \overline{\cup_{N \ge k} \set{Hus(\psi) \, | \, \psi \in E_N(A)}}.
	\end{equation*}
	This is a decreasing intersection of non-empty compact sets in the space $\CP(\T^{2d})$ of probability measures on $\T^{2d}$, hence it is compact and non-empty.
	The invariance is a consequence of Lemma \ref{lem_action_metaplectic_semiclassical_measures} and the fact that if $\psi_N \in E_N(A)$ then $Wig(M_N(A) \psi_N) = Wig(\psi_N)$ since $M_N(A)$ is unitary.
	The property for different powers of the matrix comes from
	\begin{equation*}
		E_N(A) = E_N(A^{-1}) \qqtqq{and} E_N(A^s) \subset E_N(A^t).
	\end{equation*}
	which are immediate as $M_N(A^s) = \pm M_N(A)^s$ and an eigenvector for a given matrix is always an eigenvector for any power of this matrix.

	As the metaplectic operators form a representation of the metaplectic group, we
	have that $M_h(C A C^{-1}) = \pm M_h(C) M_h(A) M_h(C)^{-1}$. This implies that
	\begin{equation*}
		E_N(C A C^{-1}) = M_h(C) E_N(A) = \set{M_h(C) \psi \; | \; \psi \in E_N(A)}.
	\end{equation*}
	The last property of the lemma then follows from Lemma \ref{lem_action_metaplectic_semiclassical_measures}.
\end{proof}

We compute an explicit expression of the Fourier coefficients of the Wigner
distribution.
\begin{lemma}[Fourier coefficients of the Wigner distribution]\label{lem_fourier_coefficients_wigner_distribution}
	Let $\psi = \sum_{j \in \eint{0}{N-1}^d} \psi_j e^\theta_j \in \CH_N(\theta)$. Then for any $(m,n) \in \Z^{2d}$, we have
	\begin{equation*}
		Wig(\psi)(e^{2 \pi i \sprod{(m,n)}{\cdot}}) = e^{\frac{2 \pi i}{N} \sigma(\theta,(-n,m)) -\frac{\pi i}{N} P(m,n)} \sum_{j \in \eint{0}{N-1}^d} \psi_j \overline{\psi_{\modulo{j-n}_N}} e^{\frac{2 \pi i }{N}\sprod{m}{j}} e^{\frac{2 \pi i}{N} \sprod{\theta_\xi}{j-\modulo{j-n}_N}},
	\end{equation*}
	where $\modulo{x}_N$ is the representant in $\eint{0}{N-1}^d$ of $x \in \Z^{d}$ modulo $N \Z^d$.
\end{lemma}

\begin{proof}
	We have
	\begin{align*}
		Wig(\psi)(e^{2 \pi i \sprod{(m,n)}{\cdot}}) & = Wig(\psi)(e^{2 \pi i \sigma((-n,m),\cdot)}) = \sprod{\Op^w_N(e^{2 \pi i \sigma((-n,m),\cdot)}) \psi}{\psi} ,\\
		& = \sprod{U_N(\frac{-n}{N}, \frac{m}{N}) \psi}{\psi} = \sum_{j,k \in \eint{0}{N-1}^d} \psi_j \overline{\psi_k} \sprod{U_N(\frac{-n}{N}, \frac{m}{N}) e^\theta_j}{e^\theta_k},\\
		& = e^{-\frac{\pi i}{N} \sprod{m}{n+2 \theta_x}} \sum_{j,k \in \eint{0}{N-1}^d} \psi_j \overline{\psi_k} e^{\frac{2 \pi i }{N}\sprod{m}{j}}\sprod{ e^\theta_{j-n}}{e^\theta_k}.
	\end{align*}
	Writing the Euclidean division of $j-n$ by $N$ as $j-n = a+bN$, the relation \eqref{eq_transformation_rule_e_theta_j} gives that
	\begin{equation*}
		\sprod{ e^\theta_{j-n}}{e^\theta_k} = e^{2 \pi i \sprod{\theta_\xi}{b}} \sprod{e^\theta_a}{e^\theta_k} = e^{2 \pi i \sprod{\theta_\xi}{b}} \delta_{a,k}.
	\end{equation*}
	Hence, as $a = \modulo{j-n}_N$ and $b = \frac{1}{N}(j-n-\modulo{j-n}_N)$, we get
	\begin{align*}
		Wig(\psi)(e^{2 \pi i \sprod{(m,n)}{\cdot}}) & = e^{-\frac{\pi i}{N} \sprod{m}{n+2 \theta_x}} \sum_{j \in \eint{0}{N-1}^d} \psi_j \overline{\psi_{\modulo{j-n}_N}} e^{\frac{2 \pi i }{N}\sprod{m}{j}} e^{\frac{2 \pi i}{N} \sprod{\theta_\xi}{j-n-\modulo{j-n}_N}},\\
		& = e^{\frac{2 \pi i}{N} \sigma(\theta,(-n,m)) -\frac{\pi i}{N} P(m,n)} \sum_{j \in \eint{0}{N-1}^d} \psi_j \overline{\psi_{\modulo{j-n}_N}} e^{\frac{2 \pi i }{N}\sprod{m}{j}} e^{\frac{2 \pi i}{N} \sprod{\theta_\xi}{j-\modulo{j-n}_N}}.
	\end{align*}
\end{proof}

Note that if $\theta_\xi =0$, then the relation
\eqref{eq_transformation_rule_e_theta_j} implies that $e^\theta_j$ only depends
on the residue modulo $N$ of $j$. In this case, the formula in Lemma
\ref{lem_fourier_coefficients_wigner_distribution} simplifies significantly into
\begin{equation}\label{eq_Fourier_coefficients_Wigner_distribution}
	Wig(\psi)(e^{2 \pi i \sprod{(m,n)}{\cdot}}) = e^{-\frac{\pi i}{N} \sprod{m} {n +2 \theta_x}} \sum_{j \in (\Z/N\Z)^d} e^{\frac{2\pi i}{N} \sprod{m}{j} } \psi_j \overline{\psi_{j-n}}.
\end{equation}

We use the previous lemma to determine the semiclassical measure associated to
the vectors $(e^\theta_j)_{j, \theta}$.

\begin{lemma}\label{lem_semiclassical_measure_e_theta_j}
	Let $(j_N)_N$ be a sequence of elements in $\eint{0}{N-1}^d$ and $(\theta_N)_N$ be a sequence of elements in $\T^{2d}$. Then
	\begin{equation*}
		\frac{j_N}{N} \xrightarrow[N \to \infty]{} x \implies \CS \CM(e^{\theta_N}_{j_N}) = \set{\delta_x \otimes \Leb}.
	\end{equation*}
\end{lemma}

\begin{proof}
	Let $(m,n) \in \Z^{2d}$. Lemma \ref{lem_fourier_coefficients_wigner_distribution} gives that $Wig(e^{\theta_N}_{j_N})(e^{2 \pi i \sprod{(m,n)}{\cdot}})$ is equal to 
	\begin{equation*}
		e^{\frac{2 \pi i}{N} \sigma(\theta_N,(-n,m)) -\frac{\pi i}{N} P(m,n)} \sum_{j \in \eint{0}{N-1}^d} \delta_{j, j_N} \delta_{\modulo{j-n}_N, j_N} e^{\frac{2 \pi i }{N}\sprod{m}{j}} e^{\frac{2 \pi i}{N} \sprod{\theta_{N,\xi}}{j-\modulo{j-n}_N}} ,
	\end{equation*}
	which simplifies into
	\begin{equation*}
		e^{\frac{2 \pi i}{N} \sigma(\theta_N,(-n,m)) -\frac{\pi i}{N} P(m,n)} \delta_{\modulo{j_N-n}_N, j_N} e^{\frac{2 \pi i }{N}\sprod{m}{j_N}} e^{\frac{2 \pi i}{N} \sprod{\theta_{N, \xi}}{j_N-\modulo{j_N-n}_N}},
	\end{equation*}
	where $\delta_{j,k}$ is the Kronecker delta. Note that $\modulo{j_N-n}_N = j_N$ is only possible if $n$ is a multiple of $N$. This is not the case for $N > |n|$. Hence, for $N > |n|$, we obtain
	\begin{align*}
		Wig(e^{\theta_N}_{j_N})(e^{2 \pi i \sprod{(m,n)}{\cdot}}) & = \delta_{n,0} e^{\frac{2 \pi i}{N} \sigma(\theta,(-n,m)) -\frac{\pi i}{N} P(m,n)} e^{\frac{2 \pi i }{N}\sprod{m}{j_N}} e^{\frac{2 \pi i}{N} \sprod{\theta_\xi}{j_N-\modulo{j_N-n}_N}} ,\\
		 & = \delta_{n,0} e^{\frac{2 \pi i}{N} \sigma(\theta,(0,m))} e^{\frac{2 \pi i }{N}\sprod{m}{j_N}} ,\\
		 & \xrightarrow[N \to + \infty]{} \delta_{n,0} e^{2 \pi i \sprod{m}{x}} = \delta_x \otimes \Leb_\xi(e^{2 \pi i \sprod{(m,n)}{\cdot}}).
	\end{align*}
	We conclude thanks to the density of the trigonometric polynomials in $\CC^\infty(\T^{2d})$.
\end{proof}

We prove a technical lemma to understand the semiclassical measure associated
to a sum of vectors.
\begin{lemma}\label{lem_convex_combination_orthogonal_semiclassical_measures}
	Let $(\psi^1_N)_N, \dots, (\psi^m_N)_N$ be sequences of vectors in $\CH_N(\theta)$ and $\mu_1, \dots, \mu_m$ be probability measures on $\T^{2d}$ such that
	\begin{equation*}
		Hus(\psi^i_N) \rightharpoonup^* \mu_i \qqtqq{and} \mu_i \bot \mu_j \qtq{for} i \neq j.
	\end{equation*}
	Then for any $(\alpha_i)_{1 \le i \le m} \in \R_+^m$ such that $\sum_{i=1}^m \alpha_i = 1$, we have
	\begin{equation*}
		Hus(\frac{\sum_{i=1}^{m} \sqrt{\alpha_i} \psi^i_N}{\norm{\sum_{i=1}^{m}\sqrt{\alpha_i} \psi^i_N}}) \rightharpoonup^* \sum_{i=1}^m \alpha_i \mu_i.
	\end{equation*}
\end{lemma}

\begin{proof}
	Remark that
	\begin{equation*}
		\norm{\psi^i_N}^2 = Hus(\psi^i_N, \psi^i_N)(1) \xrightarrow[N \to + \infty]{} \mu_i(1) = 1,
	\end{equation*}
	as $\mu_i$ is a probability measure on $\T^{2d}$. Hence, if we suppose that $N$ is large enough, we can suppose that $\norm{\psi^i_N}^2 \le 2$.
	We first prove that
	\begin{equation}\label{limit_proof_lem_convex_combination_orthogonal_semiclassical_measures}
		\forall i \neq j \in \eint{1}{m}, \quad Hus(\psi^i_N,\psi^j_N)  \xrightarrow[N \to + \infty]{} 0.
	\end{equation}
	Fix $i \neq j \in \eint{1}{m}$ and remark that the distribution $Hus(\psi^i_N,\psi^j_N)$ is a complex measure on $\T^{2d}$ of norm bounded by $2$ for all $N \in \N$ large enough. By compactness of the set of complex measures on $\T^{2d}$ of norm at most $2$, we consider a subsequence $(N_k)_k$ such that
	\begin{equation*}
		Hus(\psi^i_{N_k},\psi^j_{N_k})\xrightarrow[k \to + \infty]{} \nu,
	\end{equation*}
	where $\nu$ is a complex finite measure on $\T^{2d}$. Let $a \in \CC^0(\T^{2d})$ be a periodic continuous function. We have
	\begin{align*}
		|Hus(\psi^i_{N_k},\psi^j_{N_k})(a)| & = | \int_{\T^{2d}} a(\rho) \sprod{\psi^i_{N_k}}{\varphi_{\rho, N_k}} \sprod{\varphi_{\rho, N_k}}{\psi^j_{N_k}}  N_k^d \diff \rho | ,\\
		& \le \sqrt{\int_{\T^{2d}} |a(\rho)| |\sprod{\psi^i_{N_k}}{\varphi_{\rho, N_k}}|^2 N_k^d \diff \rho} \sqrt{\int_{\T^{2d}} |a(\rho)|  |\sprod{\varphi_{\rho, N_k}}{\psi^j_{N_k}}|^2 N_k^d \diff \rho} ,\\
		& \le \sqrt{Hus(\psi^i_{N_k})(|a|)} \sqrt{Hus(\psi^j_{N_k})(|a|)}.
	\end{align*}
	Letting $N_k$ go to $+\infty$, we obtain
	\begin{equation*}
		|\nu(a)| \le \sqrt{\mu_i(|a|)} \sqrt{\mu_j(|a|)}.
	\end{equation*}
	Since this is valid for any $a \in \CC^0(\T^{2d})$, we deduce that $\nu$ is absolutely continuous with respect to both $\mu_i$ and $\mu_j$. As these two measures are mutually singular, the complex measure $\nu$ is null. Since this argument is valid for any accumulation point of the sequence $(Hus(\psi^i_N,\psi^j_N))_N$ and the space of complex measures on $\T^{2d}$ of norm at most $2$ is metrizable, we deduce that
	\begin{equation*}
		Hus(\psi^i_N,\psi^j_N) \xrightarrow[N \to + \infty]{} 0.
	\end{equation*}
	As the Husimi distribution is bilinear in its inputs, we get
	\begin{align*}
		Hus(\sum_{i=1}^{m} \sqrt{\alpha_i} \psi^i_N) & = \sum_{i = 1}^m \alpha_i Hus(\psi^i_N, \psi^i_N) + 2 \sum_{1 \le i < j \le m} \sqrt{\alpha_i \alpha_j} Hus(\psi^i_N, \psi^j_N) \xrightharpoonup[N \to + \infty]{}^* \sum_{i = 1}^m \alpha_i \mu_i,
	\end{align*}
	thanks to \eqref{limit_proof_lem_convex_combination_orthogonal_semiclassical_measures}. This limit also implies that
	\begin{align*}
		\norm{\sum_{i=1}^{m}\sqrt{\alpha_i} \psi^i_N}^2 & = Hus(\sum_{i=1}^{m} \sqrt{\alpha_i} \psi^i_N)(1) \xrightarrow[N \to + \infty]{} \sum_{i=1}^{m} \alpha_i \mu_i(1) = 1.
	\end{align*}
	Hence,
	\begin{equation*}
		Hus(\frac{\sum_{i=1}^{m} \sqrt{\alpha_i} \psi^i_N}{\norm{\sum_{i=1}^{m}\sqrt{\alpha_i} \psi^i_N}}) = \frac{1}{\norm{\sum_{i=1}^{m}\sqrt{\alpha_i} \psi^i_N}^2} Hus(\sum_{i=1}^{m} \sqrt{\alpha_i} \psi^i_N)  \xrightharpoonup[N \to + \infty]{}^* \sum_{i = 1}^m \alpha_i \mu_i.
	\end{equation*}
\end{proof}

\section{Elliptic symplectic matrices}\label{section_elliptic_matrices}

This section is dedicated to the proof of Theorem
\ref{thm_semiclassical_measures_elliptic_matrix} and the characterization of
elliptic matrices claimed in the introduction.

\begin{proposition}\label{pro_characterization_elliptic_matrices}
	Let $A \in \Sp_{2d}(\Z)$. Then the following statements are equivalent
	\begin{enumerate}
		\item $A$ is periodic, meaning that there exists $T \in \N$ such that $A^T = I$,
		\item $A$ is diagonalizable over $\C$ and all its eigenvalues are in the unit circle,
		\item for any point $x \in \C^{2d}$, its orbit $\CO_A(x) = \set{A^t x \, | \, t \in
				      \Z}$ is bounded.
	\end{enumerate}
\end{proposition}

\begin{proof}
	We prove $1 \implies 3 \implies 2 \implies 1$.
	\begin{itemize}
		\item $\underline{1 \implies 3} : $ It is trivial as any finite set is bounded.
		\item $\underline{3 \implies 2} : $ We write $A = P M P^{-1}$ where $M$ is in Jordan form:
		      \begin{equation*}
			      M = \begin{pmatrix}
				      J_{k_1}(\lambda_1) &&\\
				 & \ddots &\\
				 && J_{k_r}(\lambda_r)
			      \end{pmatrix} \qqtq{where} J_k(\lambda) = \begin{pmatrix}
				      \lambda & 1 & \;& \;\\
				      \;& \lambda & \ddots & \;\\
				      \;& \;& \ddots & 1 \\
				      \;& \;& \;& \lambda
			      \end{pmatrix} \in M_k(\C),
		      \end{equation*}
		      and $\lambda_1, \dots, \lambda_r$ are the eigenvalues of $A$. The powers of a Jordan block are
		      \begin{equation*}
			      J_k(\lambda)^t =
			      \begin{pmatrix}
				      \lambda^t & \binom{t}{1}\lambda^{t-1} & \binom{t}{2}\lambda^{t-2} & \cdots & \cdots & \binom{t}{k-1}\lambda^{t-k+1} \\
				 & \lambda^t& \binom{t}{1}\lambda^{t-1} & \cdots & \cdots & \binom{t}{k-2}\lambda^{t-k+2} \\
				 && \ddots& \ddots & \vdots & \vdots\\
				 &&& \ddots & \ddots & \vdots\\
				 &&&& \lambda^t & \binom{t}{1}\lambda^{t-1} \\
				 &&&&& \lambda^t
			      \end{pmatrix} \qqtq{for} t \in \N.
		      \end{equation*}
		    One readily sees that if $|\lambda| > 1$ or $|\lambda|=1$ and $k > 1$, then the sequence 
			  $$(\norm{J_{k}(\lambda)^t (0, \dots, 0, 1 )})_{t \in \N},$$
			is unbounded. Therefore, we must have $|\lambda_i| \le 1$ and in case of equality $k_i = 1$. Applying the same reasoning to $A^{-1}$, we see that all the eigenvalues of $A$ must have modulus one and that the corresponding Jordan blocks must have size $1$. Therefore, $A$ is diagonalizable.
		\item $\underline{2 \implies 1} : $ The roots of the minimal polynomial of $A$ are algebraic integers and they are all on the unit circle. By a theorem of Kronecker \cite[original article]{kroneckerZweiSaetzeUeber1857} \cite[a quick modern proof]{greiterSimpleProofTheorem1978}, they are all roots of unity. It is then clear that $A^T = I$ where $T$ is the least common multiple of the orders of the eigenvalues of $A$.
	\end{itemize}
\end{proof}

We now prove the theorem \eqref{thm_semiclassical_measures_elliptic_matrix}.
\begin{proof}
	Fix $A \in \Sp_{2d}( \Z)$ periodic of minimal period $T>0$ and consider the sequence of even integers $N_k = 2k$. According to Lemma \ref{lem_conditions_quantizability_metaplectic_operators}, the Hilbert space $\CH_{N_k}(0)$ is invariant under $M_{N_k}(A)$. To simplify the notations, we drop the indices $k$ and do not mention $\theta = 0$ in the following. Let $x_1, \dots, x_n$ be periodic points for $A$ of period $T$ such that
	\begin{equation}\label{non_eq_proof_theorem_elliptic_symplectic_matrices}
		\forall (i,t) \neq (j,s) \in \eint{1}{n} \times \eint{0}{T-1}  \qquad A^t x_i \neq A^s x_j,
	\end{equation}
	and fix $\alpha_i \ge 0$ such that $\sum_{i=1}^n \alpha_i = 1$. As $A^T = I$, we have that $M_{N}(A)^T = \pm Id_{\CH_N}$. Define an eigenvector $\psi_N$ of $M_N(A)$ by
	\begin{equation*}
		\psi_N = \begin{syst}
		&\sum_{i=1}^n \frac{\alpha_i}{T}\sum_{s=0}^{T-1}  M_N(A)^s \varphi_{x_i,N} \qtq{if} M_N(A)^T = Id_{\CH_N} \\
		&\sum_{i=1}^n \frac{\alpha_i}{T}\sum_{s=0}^{T-1}  e^{\frac{\pi i}{T}s} M_N(A)^s \varphi_{x_i,N} \qtq{if} M_N(A)^T = - Id_{\CH_N}
		\end{syst}
	\end{equation*}
	In the first case, $\psi_N$ is an eigenvector with eigenvalue $1$ of $M_N(A)$. In the second one, $e^{\frac{\pi i}{T}} M_N(A)$ is $T$-periodic and $\psi_N$ is an eigenvector with eigenvalue $e^{-\frac{\pi i}{T}}$ of $M_N(A)$.
	Lemma \ref{lem_semiclassical_measure_coherent_states} and Lemma \ref{lem_action_metaplectic_semiclassical_measures} imply that
	\begin{equation*}
		Hus(M_N(A)^s \varphi_{x_i,N}) \xrightharpoonup[N \to + \infty]{}^* \delta_{A^s x_i}.
	\end{equation*}
	According to the condition \eqref{non_eq_proof_theorem_elliptic_symplectic_matrices}, we find that
	\begin{equation*}
		\forall (i,t) \neq (j,s) \in \eint{1}{n} \times \eint{0}{T-1}  \qquad \delta_{A^t x_i} \bot \delta_{A^s x_j}.
	\end{equation*}
	We denote for $x \in \T^{2d}$ its orbit under $A$ by $\CO(x)$ and the counting measure supported on $\CO(x)$ by $\delta_{\CO(x)}$. It is given by
	\begin{equation*}
		\delta_{\CO(x)} = \frac{1}{T_x} \sum_{s=0}^{T_x-1} \delta_{A^s x} \qqtq{where} T_x \text{ is the period of } x.
	\end{equation*}
	Hence, we deduce from Lemma \ref{lem_convex_combination_orthogonal_semiclassical_measures} that
	\begin{equation*}
		Hus(\frac{\psi_N}{\norm{\psi_N}}) \xrightharpoonup[N \to + \infty]{}^* \sum_{i=1}^n \frac{\alpha_i}{T}\sum_{s=0}^{T-1} \delta_{A^s x_i} = \sum_{i=1}^n \alpha_i \delta_{\CO(x_i)}.
	\end{equation*}
	Thus, the measure on the right-hand side of the last equation is a semiclassical measure for $A$. In other words, any convex combination of counting measures on the orbits of points of period $T$ for $A$ is a semiclassical measure of $A$.

	We now leverage the fact that the set $\CS \CM(A)$ of semiclassical measures of
	$A$ is weak-star closed in the set of probability measures on $\T^{2d}$. It
	allows us to remove the conditions on the period of the points $x_i$. Indeed,
	as $T$ is the minimal period of $A$, the set of points of period strictly smaller
	than $T$ is contained in a finite union of tori of dimension strictly smaller
	than $2d$. Hence, any point $y$ of period $T_1 < T$ can be approximated by a
	sequence of points $(y_m)_{m}$ of period $T$. Moreover, we get
	\begin{equation*}
		\delta_{\CO(y_m)}\ = \frac{1}{T} \sum_{s=0}^{T-1} \delta_{A^s y_m} \xrightharpoonup[N \to + \infty]{}^* \frac{1}{T_1} \sum_{s=0}^{T_1-1} \delta_{A^s y} = \delta_{\CO(y)}.
	\end{equation*}
	Hence,
	\begin{equation*}
		X := \set{\sum_{i=1}^n \alpha_i \delta_{\CO(x_i)} \; | \; \sum_{i=1}^n \alpha_i =1 \qtq{and} x_i \in \T^{2d}} \subset \CS \CM(A) \subset \CP_A(\T^{2d}).
	\end{equation*}
	We can conclude as long as we prove that $X$ is dense in $\CP_A(\T^{2d})$ for the weak-star topology. Define a continuous projector $\pi$ from $\CP(\T^{2d})$ onto $\CP_A(\T^{2d})$ by
	\begin{equation*}
		\pi(\mu) = \frac{1}{T} \sum_{s=0}^{T-1} A^s_\# \mu.
	\end{equation*}
	This projector sends the space of convex combinations of Dirac measures onto $X$. Since the Dirac measures are the extreme points of the set of probability measures, the Krein-Milman theorem implies that the convex combinations of Dirac measures are dense in $\CP(\T^{2d})$. As the image of a dense subset under a continuous surjective map is dense, we deduce that $X$ is dense in $\CP_A(\T^{2d})$.
\end{proof}

\section{An arithmetic interlude}\label{section_arithmetic_interlude}

This section is dedicated to the proof of some arithmetic lemmas that will be
needed in the proofs for parabolic symplectic matrices.

We first present a lemma that controls the existence of solutions of
integer polynomial equations over $\Z/N\Z$ via the existence of a solution over
$\C$. It is a standard consequence of the Weak Nullstellensatz.
\begin{lemma}\label{lem_empty_over_C_empty_over_ZN}
	Let $(P_1, \dots, P_m) \in \Z[X_1, \dots X_n]$ be integer polynomials and denote by $(E)$ the system
	\begin{equation*}
		(E) \quad \begin{syst}
			P_1(X_1, \dots, X_n) &= 0 ,\\
			&\vdots ,\\
			P_m(X_1, \dots, X_n) &= 0
		\end{syst}
	\end{equation*}
	If the system $(E)$ has no solution over $\C$, then there exists an integer $N_0 \in \Z$ such that the system $(E)$ has no solution in $\Z / N\Z$ for any $N > N_0$.
\end{lemma}
\begin{proof}
	As the system $(E)$ has no solution in $\C$, it has no solution in the algebraic closure $\overline{\Q}$ of $\Q$. A variant of the Nullstellensatz \cite[chap. 15.3, cor. 33]{dummitAbstractAlgebra2004} implies that $1$ belongs to the ideal generated by $P_1, \dots, P_m$ in $\Q[X_1, \dots, X_n]$. In other words, there exist polynomials $R_1, \dots, R_m \in \Q[X_1, \dots, X_n]$ such that
	\begin{equation}\label{eq_lem_empty_over_C_empty_over_ZN}
		R_1 P_1 + \dots + R_m P_m = 1.
	\end{equation}
	We multiply this equation by the least common multiple of the denominators of the coefficients of all the polynomials involved to obtain
	\begin{equation}\label{eq_lem_empty_over_C_empty_over_ZN_final}
		R_1 P_1 + \dots + R_m P_m = N_0,
	\end{equation}
	with $R_1, \dots, R_m \in \Z[X_1, \dots, X_n]$ and $N_0 \in \N$. Suppose now that the system $(E)$ admits a solution $(x_1, \dots, x_n)$ in $\Z/N \Z$. Evaluating the equality \eqref{eq_lem_empty_over_C_empty_over_ZN_final} at $(x_1, \dots, x_n)$ thus gives
	\begin{equation*}
		N_0 \equiv \sum_{i=1}^m R_i(x_1, \dots, x_n) P_i(x_1, \dots, x_n) \equiv 0 \modulo{N}.
	\end{equation*}
	Hence, $N$ must divide $N_0$ and thus $N \le N_0$.
\end{proof}

We deduce from this lemma the emptiness of the solutions of a particular family
of systems of quadratic integer equations over $\Z/N\Z$ for $N$ large enough.
This will be used in the proof of Proposition
\ref{pro_projection_parabolic_case_dim1}.
\begin{lemma}\label{lem_no_solutions_system_2_Fourier_coeff_parabolic_1d}
	Let $a \in \Z, b \in \Z^*$ and $n_1 \neq \pm n_2 \in \Z^*$. Then there exists a finite number of values of $N \in \N$ such that the following system admits solutions $(j,l) \in \Z/ N \Z$
	\begin{equation}\label{syst_lem_arithmetic_parabolic_1d}
		\begin{syst}
			bj^2 - a j &\equiv  b(j-n_1)^2 - a (j-n_1) \quad \modulo{2N} ,\\
			bl^2 -  a l &\equiv  b(l-n_2)^2 - a (l-n_2) \quad \modulo{2N} ,\\
			bj^2 - a j &\equiv bl^2 - a jl \quad \modulo{2N}.
		\end{syst}
	\end{equation}
\end{lemma}

\begin{proof}
	We prove that the system has no solution over $\C$ and then Lemma \ref{lem_empty_over_C_empty_over_ZN} concludes, as it implies that a system with integer coefficients and no solutions in $\C$ has no solutions in $\Z/N\Z$ for $N$ large enough. Suppose that there exists $(j,l) \in \C^2$ that is a solution to the system \eqref{syst_lem_arithmetic_parabolic_1d}. The first two equations lead to
	\begin{equation*}
		j = \frac{a}{2b} + \frac{n_1}{2} \qqtqq{and} l = \frac{a}{2b} + \frac{n_2}{2}.
	\end{equation*}
	Hence, the last equation is equivalent to
	\begin{align*}
		b(\frac{a}{2b} + \frac{n_1}{2})^2 - a (\frac{a}{2b} + \frac{n_1}{2}) = b(\frac{a}{2b} + \frac{n_2}{2})^2 - a (\frac{a}{2b} + \frac{n_2}{2}).
	\end{align*}
	which simplifies into
	\begin{align*}
		b\frac{n_1^2}{4} = b \frac{n_2^2}{4} \iff n_1 = \pm n_2.
	\end{align*}
	As we supposed that $n_1 \neq \pm n_2$, the system has no solution over $\C$.
\end{proof}

The following lemma will be used in the proof of Theorem
\ref{thm_upper_bound_quantum_variance_position_parabolic_matrices}.
\begin{lemma}\label{lem_limit_average_indicator_bound_quantum_variance}
	Let $B \in M_d(\Z) \cap \GL_d(\Q)$, $N \in \N$ and $n \in \Z^d \setminus \set{0}$. Then
	\begin{equation}
		\limsup_{T \to + \infty} \frac{1}{T^2} \sum_{s,t = 0}^{T-1} \mathbb{1}_{N \Z^d}((s-t)B n) \le \frac{N \wedge gcd(\det(B) n)}{N},
	\end{equation}
	where $gcd(\det(B) n)$ is the greatest common divisor of the nonzero coordinates of $\det(B) n$ and $a \wedge b$ is the greatest common divisor of $a,b \in \Z$. If $|\det(B)| = 1$, the $\limsup$ can be replaced by a real limit and the inequality by an equality.
\end{lemma}

\begin{proof}
	We need to count the number of solutions $(s,t) \in \eint{0}{T-1}^2$ of the equation
	\begin{equation}\label{eq_proof_lem_limit_average_indicator_bound_quantum_variance_problem_P}
		(s-t) B n \in N \Z^d \iff  (s-t) \tilde{n} \in N \adj(B) \Z^d,
	\end{equation}
	where $\adj(B)$ is the adjugate matrix of $B$ and $\tilde{n} = \det(B) n$. The adjugate matrix has integer coefficients and satisfies $\adj(B) B = \det(B) I_d$. If $|\det(B)| = 1$, we see that $|\det(\adj(B))| = 1$ and thus $\adj(B) \Z^d = \Z^d$. Otherwise, we only have $\adj(B) \Z^d \subset \Z^d$. Therefore, the number of solutions of the equation \eqref{eq_proof_lem_limit_average_indicator_bound_quantum_variance_problem_P} is bounded by the number of solutions of the equation
	\begin{equation}\label{eq_proof_lem_limit_average_indicator_bound_quantum_variance_problem_P_prime}
		(s-t) \tilde{n} \equiv 0 \quad \modulo{N},
	\end{equation}
	with equality if $|\det(B)| = 1$. Denote by $S$ the subset of $\eint{1}{d}$ of the indices of the nonzero coordinates of $\tilde{n}$. The equation \eqref{eq_proof_lem_limit_average_indicator_bound_quantum_variance_problem_P_prime} is equivalent to a system of one-dimensional equations
	\begin{equation*}
		\forall i \in S, \quad (s-t) \tilde{n}_i \equiv 0 \quad \modulo{N}  \quad \iff \quad  \forall i \in S, \quad (s-t) \equiv 0 \quad \modulo{\frac{N}{N \wedge \tilde{n}_i}}.
	\end{equation*}
	Using the Chinese remainder theorem, we see that this system is equivalent to
	\begin{equation}\label{eq_proof_lem_limit_average_indicator_bound_quantum_variance}
		(s-t) \equiv 0 \quad \modulo{\bigvee_{i \in S} \frac{N}{N \wedge \tilde{n}_i}},
	\end{equation}
	where $a \vee b$ denotes the least common multiple of $a,b \in \Z$. Let $a,b,c \in \Z^*$. The following identity can be checked by decomposing $a,b,c$ as a product of prime numbers
	\begin{equation*}
		\frac{a}{a\wedge b} \vee \frac{a}{a\wedge c} = \frac{a}{a\wedge b .\wedge c}
	\end{equation*}
	After iterative applications of this rule, we obtain that \eqref{eq_proof_lem_limit_average_indicator_bound_quantum_variance} is equivalent to
	\begin{equation*}
		s \equiv t \quad \modulo{\frac{N}{N \wedge gcd(\det(B) n)}}.
	\end{equation*}
	The number of solutions of this equation in $\eint{0}{T-1}$ is
	\begin{equation*}
		T \cdot \left(\frac{T}{\frac{N}{N \wedge gcd(\det(B) n)}} + \CO(1)\right)\ = \frac{T^2 N \wedge gcd(\det(B) n)}{N} + \CO(T).
	\end{equation*}
	All in all, we get, with equality if $|\det(B)| = 1$,
	\begin{align*}
		\frac{1}{T^2} \sum_{s,t = 0}^{T-1} \mathbb{1}_{N \Z^d}((s-t)B n) \le \frac{N \wedge gcd(\det(B) n)}{N} + \CO(\frac{1}{T}) \xrightarrow[T \to + \infty]{} \frac{N \wedge gcd(\det(B) n)}{N}.
	\end{align*}
\end{proof}

Finally, we provide\footnote{The author claims no novelty as he remembers
	having read this idea of proof somewhere on the internet, not on an AI, but can
	sadly not remember where.} a proof of the folklore fact that given a finite set
of integers $S$, there exist infinitely many primes such that each of the
integers in $S$ is a square modulo those primes.
\begin{lemma}\label{lem_finite_integer_squares_modulo_infinity_primes}
	Let $a_1, \dots, a_n \in \Z$. Then there exist infinitely many primes $p$ such that each $a_i$ is a square modulo $p$.
\end{lemma}
\begin{proof}
	Denote by $S$ the set of primes that divide any of the $a_i$.
	\begin{equation*}
		S := \set{p \text{ prime} \, | \,  \exists i \in \eint{1}{n} \, p \text{ divides } a}.
	\end{equation*}
	The proof is a combination of the quadratic reciprocity law and of Dirichlet's theorem on arithmetic progressions in prime numbers.
	For an integer to be a square modulo $p$, it suffices that each of its prime divisors and $-1$ (to handle negative integers) are squares modulo $p$. Therefore, we need to find primes $p$ such that each element of $S$ and $-1$ are squares modulo $p$.
	For this purpose, we suppose first that $p \equiv 1 \modulo{8}$ to ensure that $2$ is a square modulo $p$ \cite[sec. 1.3.2, thm. 5]{serreCourseArithmetic1973a}, which implies in particular that $p \equiv 1 \modulo{4}$ and thus that $-1$ is also a square modulo $p$. Hence, the quadratic reciprocity law \cite[sec. 1.3.3, thm. 6]{serreCourseArithmetic1973a} gives that a prime $q \neq 2$ is a square modulo $p$ if and only if $p$ is a square modulo $q$, which is in particular the case if $p \equiv 1 \modulo{q}$.
	Hence, we further impose that
	\begin{equation*}
		\forall q \in S \setminus \set{2}, \quad p \equiv 1 \quad \modulo{q}.
	\end{equation*}
	Putting all the conditions together via the Chinese remainder theorem, we are done if we find $p$ such that
	\begin{equation*}
		p \equiv 1 \quad \modulo{8 \prod_{q \in S \setminus \set{2}} q}.
	\end{equation*}
	Thankfully, Dirichlet's theorem on arithmetic progressions \cite[chap. 6]{serreCourseArithmetic1973a} tells us that there are infinitely many primes that satisfy the conditions above, which concludes.
\end{proof}

\section{Parabolic symplectic matrices}\label{section_parabolic_matrices}

The goal of this section is to study the semiclassical measures associated to
symplectic matrices of the form
\begin{equation*}
	P_B = \begin{pmatrix}
		I & B \\ 0 & I
	\end{pmatrix} \qqtqq{where} B \in \Sym_d(\Z) \cap \GL_d(\Q).
\end{equation*}
We define an auxiliary subgroup of symplectic matrices
\begin{equation*}
	\tilde{P}_B = \begin{pmatrix}
		I & 0 \\ B & I
	\end{pmatrix} = J P_{-B} J^{-1} \qqtq{where} B \in \Sym_d(\Z) \cap \GL_d(\Q),
\end{equation*}
and a new orthonormal basis $(f^\theta_j)_{j \in \eint{0}{N-1}^d}$ of $\CH_N(\theta)$ given by
\begin{equation}\label{eq_new_basis_f_theta_j}
	f^\theta_j = M_N(J)^{-1} e^{J\theta}_j = \CF_{(2 \pi N)^{-1}}^{-1} e^{J\theta}_j \qqtq{for} j \in \eint{0}{N-1}^d, \, \theta \in \T^{2d}.
\end{equation}
The distributions $f^\theta_j$ belong to $\CH_N(\theta)$ thanks to \eqref{eq_action_J_CH_N(theta)}.

Since we seek to describe the semiclassical measures of the family $(P_B)_{B \in
			\Sym_d(\Z)}$, we first need to determine which spaces $\CH_N(\theta)$ are
invariant under $M_h(P_B)$. This is done in the next lemma.
\begin{lemma}[Quantizability conditions and expression for $M_h(P_B)$]\label{lem_quantizability_condition_expression_Mh_PB}
	Let $B \in \Sym_d(\Z)$. Then $\CH_N((\theta_x, \theta_\xi))$ is invariant under $M_h(P_B)$ if and only if
	\begin{equation}\label{eq_lem_quantizability_condition_expression_Mh_PB}
		B \theta_\xi \equiv \frac{N \modulo{B_{i,i}}_2}{2} \modulo{\Z^d} \qqtq{where} \modulo{B_{i,i}}_2 = \begin{pmatrix}
			\modulo{b_{1,1}}_2 \\
			\vdots \\
			\modulo{b_{d,d}}_2
		\end{pmatrix} \qtq{and} B = (b_{i,j})_{1 \le i,j \le d},
	\end{equation}
	with $\modulo{b}_2$ equal to $0$ if $b$ is even and $1$ otherwise.
	If $\theta_\xi$ satisfies this equation, the restriction $M_N(P_B)$ to $\CH_N(\theta)$ is given on the basis $(f^\theta_j)_j$ by
	\begin{equation*}
		M_N(P_B) f^\theta_j = e^{-\frac{\pi i}{N} \sprod{B(j + \theta_\xi)}{j+ \theta_\xi}} f^\theta_j.
	\end{equation*}
\end{lemma}

\begin{proof}
	We first prove an analogous lemma for the auxiliary matrices $\tilde{P}_B$.
	We compute the angle $\phi_{\tilde{P}_B}$ defined in Lemma \ref{lem_conditions_quantizability_metaplectic_operators}. Let $(x, \xi) \in \Z^{2d}$. We have
	\begin{align*}
		P(\tilde{P}^{-1}_B(x,\xi)) - \tilde{P}((x,\xi)) & = \sprod{x}{-Bx + \xi} - \sprod{x}{\xi} = - \sprod{Bx}{x},\\
		& \equiv \sum_{1 \le i,j \le d} b_{i,j}x_i x_j \equiv \sum_{i=1}^d b_{i,i} x_i^2 + 2 \sum_{1 \le i < j \le d} b_{i,j} x_i x_j \quad \modulo{2} ,\\
		& \equiv \sum_{i=1}^d b_{i,i} x_i \equiv
		\sigma((0,\modulo{B_{i,i}}_2),(x,\xi)) \quad \modulo{2}.
	\end{align*}
	Therefore, $\phi_{\tilde{P}_B} = (0, \modulo{B_{i,i}}_2)$.
	According to Lemma \ref{lem_conditions_quantizability_metaplectic_operators}, the space $\CH_N(\theta)$ is invariant under $M_h(\tilde{P}_B)$ if and only if
	\begin{equation*}
		(I - \tilde{P}_B) \theta \equiv \frac{N \phi_{\tilde{P}_B}}{2} \modulo{\Z^{2d}}.
	\end{equation*}
	Therefore
	\begin{equation}\label{equiv_quantizability_condition_expression_Mh_tildeP_B}
		M_h(\tilde{P}_B) \CH_N(\theta) = \CH_N(\theta) \iff B \theta_x \equiv \frac{N \modulo{B_{i,i}}_2}{2} \modulo{\Z^d}.
	\end{equation}
	We suppose that $\theta_x$ is a solution of
	\eqref{eq_lem_quantizability_condition_expression_Mh_PB} and we compute the
	action of $M_h(\tilde{P}_B)$ on $e^\theta_j$. We have thanks to the expression
	for the metaplectic operators computed in Lemma
	\ref{lem_metaplectic_op_generator_symp_group} that
	\begin{align*}
		M_{(2\pi N)^{-1}} (\tilde{P}_{B}) e^\theta_j & = \frac{1}{N^d} \sum_{k \in \Z^d} e^{\pi i N \sprod{Bx}{x}} e^{2 \pi i \sprod{k}{\theta_\xi}} \delta\left(x - \frac{j+kN - \theta_x}{N}\right),\\
		& = \frac{1}{N^d} \sum_{k \in \Z^d} e^{\pi i N \sprod{B(k + \frac{j - \theta_x}{N})}{k + \frac{j - \theta_x}{N}}} e^{2 \pi i \sprod{k}{\theta_\xi}} \delta\left(x - \frac{j+kN - \theta_x}{N}\right),\\
		& = e^{\frac{\pi i}{N} \sprod{B(j- \theta_x)}{j- \theta_x}} \frac{1}{N^d} \sum_{k \in \Z^d} e^{\pi i \sprod{Bk}{N k - 2 \theta_x}} e^{2 \pi i \sprod{k}{\theta_\xi}} \delta\left(x - \frac{j+kN - \theta_x}{N}\right) ,\\
		& =  e^{\frac{\pi i}{N} \sprod{B(j- \theta_x)}{j- \theta_x}} e^\theta_j.
	\end{align*}
	We used that there exists an integer $a \in \Z$ such that $B\theta_x = \frac{N \modulo{B_{i,i}}_2}{2} + a$ as $\theta_x$ is a solution of \eqref{eq_lem_quantizability_condition_expression_Mh_PB}. Hence, for $k \in \Z^d$, we get
	\begin{align*}
		\sprod{Bk}{N k - 2 \theta_x} & = \sprod{k}{N B k - 2 B\theta_x} = N\sprod{k}{Bk - \modulo{B_{i,i}}_2}-2\sprod{Bk}{a} ,\\
		& \equiv N \left(\sum_{1 \le i,j \le d} b_{i,j}k_i k_j + \sum_{i=1}^d k_i \modulo{b_{i,i}}_2 \right) \quad \modulo{2} ,\\
		& \equiv N \left( \sum_{i=1}^d b_{i,i} k_i(k_i+1) + 2 \sum_{1 \le i < j \le d} b_{i,j} x_i x_j \right) \quad \modulo{2} ,\\
		& \equiv 0 \quad \modulo{2},
	\end{align*}
	as $k_i(k_i+1)$ is always an even integer. Therefore, $e^{\pi i \sprod{Bk}{N k - 2 \theta_x}} =1$ for any $k \in \Z^{d}$.

	We finally transfer the results above to $P_B = J^{-1} \tilde{P}_{-B} J$. Since
	the metaplectic operators form a representation of the metaplectic group, we
	get
	\begin{equation}\label{eq_lem_quantizability_condition_expression_Mh_tildeP_B_conjugation_relation}
		M_h(P_B) = \pm M_h(J)^{-1} M_h(\tilde{P}_{-B}) M_h(J).
	\end{equation}
	Thus the relation \eqref{eq_action_J_CH_N(theta)} and the equivalence \eqref{equiv_quantizability_condition_expression_Mh_tildeP_B} imply
	\begin{align*}
		M_h(P_B) \CH_N(\theta) = \CH_N(\theta) & \iff M_h(\tilde{P}_B) \CH_N(J\theta) = \CH_N(J \theta) ,\\
		& \iff -B (J\theta)_x \equiv \frac{N \modulo{-B_{i,i}}_2}{2} \modulo{\Z^d} ,\\
		& \iff B \theta_\xi \equiv \frac{N \modulo{B_{i,i}}_2}{2} \modulo{\Z^d}.
	\end{align*}
	This settles the first part of the lemma. The computation done above for $\tilde{P}_B$ and the conjugation relation \eqref{eq_lem_quantizability_condition_expression_Mh_tildeP_B_conjugation_relation} leads to
	\begin{align*}
		M_N(P_B) f^\theta_j & = M_h(J) M_h(\tilde{P}_{-B}) e^{J \theta}_j = e^{\frac{\pi i}{N} \sprod{-B(j- (J\theta)_x)}{j- (J\theta)_x}} e^{J \theta}_j = e^{-\frac{\pi i}{N} \sprod{B(j - \theta_\xi)}{j - \theta_\xi}} f^\theta_j.
	\end{align*}
\end{proof}

We transfer some results proven in the section
\ref{section_semiclassical_measures} on the basis $(e^\theta_j)_j$ to the basis
$(f^\theta_j)_j$.
\begin{lemma}\label{lem_properties_basis_f_theta_j}
	Let $\psi = \sum_{j \in \eint{0}{N-1}^d} \psi_j f^\theta_j \in \CH_N(\theta)$. Then for any $(m,n) \in \Z^{2d}$, we have
	\begin{equation*}
		Wig(\psi)(e^{2 \pi i \sprod{(m,n)}{\cdot}}) = e^{\frac{2 \pi i}{N} \sigma(J\theta,(m,n)) + \frac{\pi i}{N} P(m,n)} \sum_{j \in \eint{0}{N-1}^d} \psi_j \overline{\psi_{\modulo{j+m}_N}} e^{\frac{2 \pi i }{N}\sprod{n}{j}} e^{-\frac{2 \pi i}{N} \sprod{\theta_x}{j-\modulo{j+m}_N}}.
	\end{equation*}
	Moreover, if $(j_N)_N$ is a sequence of elements in $\eint{0}{N-1}^d$ and $(\theta_N)_N$ is a sequence of elements in $\T^{2d}$, then
	\begin{equation*}
		\frac{j_N}{N} \xrightarrow[N \to \infty]{} x \implies \CS \CM(f^{\theta_N}_{j_N}) = \set{\Leb \otimes \delta_{x}  }.
	\end{equation*}
\end{lemma}

\begin{proof}
	Let $\tilde{\psi} = M_h(J) \psi$. The exact Egorov theorem \eqref{eq_exact_Egorov_observable_torus} implies
	\begin{align*}
		Wig(\psi)(e^{2 \pi i \sprod{(m,n)}{\cdot}}) & = Wig(M_h(J)^{-1}\tilde{\psi})(e^{2 \pi i \sprod{(m,n)}{\cdot}}) = Wig(\tilde{\psi})(e^{2 \pi i \sprod{(m,n)}{J^{-1}\cdot}}) ,\\
		& = Wig(\tilde{\psi})(e^{2 \pi i \sprod{(n,-m)}{\cdot}}).
	\end{align*}
	Since $\tilde{\psi} = \sum_{j \in \eint{0}{N-1}^d} \psi_j e^{J \theta}_j$, Lemma \ref{lem_fourier_coefficients_wigner_distribution} leads to
	\begin{align*}
		Wig(\tilde{\psi})(e^{2 \pi i \sprod{(n,-m)}{\cdot}}) = e^{\frac{2 \pi i}{N} \sigma(J\theta,(m,n)) + \frac{\pi i}{N} P(m,n)} \sum_{j \in \eint{0}{N-1}^d} \psi_j \overline{\psi_{\modulo{j+m}_N}} e^{\frac{2 \pi i }{N}\sprod{n}{j}} e^{-\frac{2 \pi i}{N} \sprod{\theta_x}{j-\modulo{j+m}_N}}.
	\end{align*}
	This gives the first part of the lemma.

	The second part of the lemma is a direct consequence of the definition
	\eqref{eq_new_basis_f_theta_j} of $f^\theta_j$, Lemma
	\ref{lem_action_metaplectic_semiclassical_measures}, Lemma
	\ref{lem_semiclassical_measure_e_theta_j} and the equality
	\begin{equation*}
		J^{-1}_{\#} (\delta_x \otimes \Leb) = \Leb \otimes \delta_{x} \qqtq{for} x \in \T^{d}.
	\end{equation*}
\end{proof}

\begin{remark}
	Similarly to the vector $e^\theta_j$, the vector $f^\theta_j$ only depends on the residue class of $j$ modulo $N$ if $\theta_x = 0$ as a consequence of \eqref{eq_transformation_rule_e_theta_j}. In this case, the expression for the Fourier coefficients of the Wigner distribution in the lemma above simplifies significantly into
	\begin{equation}\label{eq_simplified_expression_fourier_coefficients_wigner_distribution_f_theta_j}
		Wig(\psi)(e^{2 \pi i \sprod{(m,n)}{\cdot}}) = e^{\frac{\pi i}{N} \sprod{n}{m-2 \theta_\xi}} \sum_{j \in (\Z / N \Z)^d} e^{\frac{2 \pi i }{N}\sprod{n}{j}} \psi_j \overline{\psi_{j+m}}.
	\end{equation}
\end{remark}

We remark that the equation
\eqref{eq_lem_quantizability_condition_expression_Mh_PB} is trivially satisfied
by $\theta_\xi =0$ if $N$ is even. Therefore, the restriction of $M_{2N}(P_B)$
to $\CH_{2N}(0)$ is always well-defined and given by
\begin{equation*}
	M_{2N}(P_B) f^0_j = e^{-\frac{\pi i }{2N} \sprod{B j}{j}} f^0_j \qqtq{for} j \in \eint{0}{2N-1}^d.
\end{equation*}
It is diagonal in the basis $(f^0_j)_{j \in \eint{0}{2N-1}^d}$ and its eigenvalues and eigenspaces depend on the quadratic form $x \mapsto \sprod{Bx}{x}$ on $\eint{0}{2N-1}^d$. We use this observation to build semiclassical measures associated to $P_B$ and prove Proposition \ref{pro_parabolic_nu_times_lebesgue}.

\begin{proof}
	The proof consists in explicitly constructing eigenvectors of $M_{2N}(P_B)$ having semiclassical measure of the form
	\begin{equation*}
		(\sum_{i=1}^m \alpha_i \Leb \otimes \delta_{\xi_i}) \qtq{where} \xi_i \in \T^d, \, \alpha_i \ge 0 \qtq{and} \sum_{i=1}^m \alpha_i = 1.
	\end{equation*}
	To conclude, it then suffices to use the density of convex combinations of Dirac measures in $\CP(\T^{d})$ for the weak-star topology and the fact that $\CS \CM(P_B)$ is closed for this topology.

	We suppose without loss of generality that the points $\xi_1, \dots, \xi_m$ are
	distinct. We take $2 N = 2^{2n-1}$ and consider the points $a_{i,n} =
		\floor{2^{2n-1}\xi_i}$. For each point $a_{i,n}$ there exists a point $k_{i,n}
		\in \eint{0}{2^{n-1}-1}^d$ such that $c_{i,n} = 2^n \cdot k_{i,n}$ is at a
	distance at most $2^n$, for the $l^\infty$ norm, from $a_{i,n}$. This implies
	that
	\begin{equation*}
		\frac{c_{i,n}}{2N} = \frac{a_{i,n}}{2^{2n-1}} + \CO(\frac{2^n}{2^{2n-1}}) = \frac{\floor{2^{2n-1}\xi_i}}{2^{2n-1}} + \CO(2^{-n}) \xrightarrow[n \to + \infty]{} \xi_i,
	\end{equation*}
	and that
	\begin{equation*}
		\sprod{B c_{i,n}}{c_{i,n}} \equiv  \sprod{B 2^n \cdot k_{i,n}}{2^n \cdot k_{i,n}} \equiv 2^{2n} \sprod{B k_{i,n}}{k_{i,n}} \equiv 0 \modulo{2^{2n}}.
	\end{equation*}
	Therefore, each $f^0_{c_{i,n}}$ is an eigenvector of $M_{2N}(P_B)$ associated to the eigenvalue $1$, as $$e^{\frac{\pi i}{2N} \sprod{B c_{i,n}}{c_{i,n}}} = 1.$$ Moreover, as the points $\xi_1, \dots, \xi_m$ are distinct, the points $a_{1,n}, \dots, a_{m,n}$ are distinct for $n$ large enough and we can thus choose the points $c_{1,n}, \dots, c_{m,n}$ to be distinct. Hence,
	\begin{equation*}
		\psi_n = \sum_{i=1}^m \sqrt{\alpha_i} f^0_{c_{i,n}},
	\end{equation*}
	is an eigenvector of norm $1$ for $M_{2N}(P_B)$ for $n$ large enough. Lemma \ref{lem_properties_basis_f_theta_j} gives that
	\begin{equation*}
		Hus(f^0_{c_{i,n}}) \xrightharpoonup[n \to \infty]{}^* \Leb \otimes .\delta_{\xi_i}
	\end{equation*}
	As the points $\xi_1, \dots, \xi_m$ are distinct, the measures $\delta_{\xi_i}$ and $\delta_{\xi_j}$ are mutually singular as long as $i \neq j$. Lemma \ref{lem_convex_combination_orthogonal_semiclassical_measures} thus leads to
	\begin{equation*}
		Hus(\psi_n) \xrightharpoonup[n \to \infty]{}^* \Leb \otimes (\sum_{i=1}^m \alpha_i \delta_{\xi_i}),
	\end{equation*}
	as $\norm{\psi_n} =1$.
\end{proof}

\subsection{The case \texorpdfstring{$d=1$}{d=1}}

This subsection is dedicated to an in-depth study of the case $d=1$ and to the
proof of Theorem \ref{thm_full_characterization_parabolic_dimension_1}. The
proof is decomposed into two propositions:
\begin{itemize}
	\item we prove in Proposition \ref{pro_projection_parabolic_case_dim1} that all the
	      semiclassical measures for $P_b$ must be of the form presented in the theorem,
	\item we prove conversely in Proposition
	      \ref{pro_projection_parabolic_case_dim1_building_the_measure} that all the
	      measures of this form are attained.
\end{itemize}

The first step in our study consists in solving the equation
\eqref{eq_lem_quantizability_condition_expression_Mh_PB} to determine for which
values of $\theta$ the Hilbert space $\CH_N(\theta)$ is stable under $M_{(2\pi
			N)^{-1}}(P_b)$. The equation is equivalent in dimension one to
\begin{equation*}
	b \theta_\xi \equiv \frac{N \modulo{b}_2}{2} \quad \modulo{\Z}.
\end{equation*}
If either $b$ or $N$ is even, this reduces to $b\theta_\xi \equiv 0 \modulo{\Z}$ and if both $b$ and $N$ are odd, this reduces to $b \theta_\xi \equiv \frac{1}{2} \modulo{\Z}$. Solving each case, we obtain that $M_N(P_b)$ is well defined on $\CH_N(\theta_x, \theta_\xi)$ if and only if
\begin{nalign}\label{cond_quantizability_parabolic_one_dim}
	\theta_\xi &\in \set{0, \frac{1}{b}, \dots, \frac{b-1}{b}} \quad &\text{if $b$ or $N$ are even} ,\\
	\theta_\xi &\in \set{\frac{1}{2b}, \frac{3}{2b}, \dots, \frac{2b-1}{2b}} \quad &\text{if $b$ and $N$ are odd}.
\end{nalign}
We suppose in the following that $\theta_\xi$ always satisfies the conditions \eqref{cond_quantizability_parabolic_one_dim}. A useful consequence is that $2 b \theta_\xi$ is always an integer.

The second step is to determine more precisely the spectrum and the eigenspaces
of $M_N(P_b)$. We deduce from the expression of Lemma
\ref{lem_quantizability_condition_expression_Mh_PB} that the spectrum of
$M_N(P_b)$ is
\begin{equation*}
	\spectre(M_{N}(P_b)) = \set{e^{-\frac{\pi i}{N} b(j - \theta_\xi)^2} \, | \, j \in \Z^d},
\end{equation*}
and that the eigenspace $E_\lambda$ corresponding to an eigenvalue $\lambda \in \spectre(M_{N}(P_b))$ is
\begin{equation*}
	E_\lambda = \Vect \set{f^\theta_j \, | \, e^{-\frac{\pi i}{N} b(j - \theta_\xi)^2} = \lambda }.
\end{equation*}
It will be useful to know under which conditions two vectors $f^\theta_j$ and $f^\theta_{k}$ belong to the same eigenspace. We have for any $j, k \in \eint{0}{N-1}$ that
\begin{align*}
	e^{-\frac{\pi i}{N} b(j - \theta_\xi)^2} = e^{-\frac{\pi i}{N} b(k - \theta_\xi)^2} \iff bj^2 - 2b \theta_\xi j \equiv  bk^2 - 2b \theta_\xi k \quad \modulo{2N}.
\end{align*}
Therefore,
\begin{equation}\label{cond_e_0_j_same_eigenvalue_dim_1}
	\exists \lambda \in \spectre(M_{N}(P_b)), \quad  f^\theta_j, f^\theta_{k} \in E_\lambda \iff bj^2 - 2b \theta_\xi j \equiv  bk^2 - 2b \theta_\xi k \quad \modulo{2N}.
\end{equation}
Note that the vector subspace generated by $f^\theta_j$ depends only on the residue class of $j$ modulo $N$ thanks to \eqref{eq_transformation_rule_e_theta_j}.

The structure of the eigenspaces and the conditions
\eqref{cond_e_0_j_same_eigenvalue_dim_1} heavily constrain the possible
semiclassical measures associated to $P_b$.
\begin{proposition}\label{pro_projection_parabolic_case_dim1}
	Let $b \in \N^*$. Any semiclassical measure $\mu \in \CS \CM(P_b)$ has the form
	\begin{equation}\label{eq_pro_projection_parabolic_case_dim1_form_measure}
		\mu = \gamma \Leb_x \otimes \nu + \sum_{l=0}^{bm-1} \alpha_l (1 + \beta_l \cos(2 \pi m x + \phi_l)) \Leb_x \otimes \delta_{\frac{l}{bm}}.
	\end{equation}
	where $\nu \in \CP(\T)$, $m \in \N$, $\gamma \in [0,1]$, $(\alpha_l, \beta_l, \phi_l)_{l \in \eint{0}{bm-1}} \in ([0,1]^2 \times \R)^{bm}$ and $$\gamma + \sum_{l=0}^{bm-1} \alpha_l = 1.$$
\end{proposition}
In other words, the projection on the position of any semiclassical measure is absolutely continuous with respect to the Lebesgue measure and even better, it has at most one positive Fourier coefficient that is nonzero. Namely, the proposition implies the inclusion \eqref{incl_proof_pro_projection_parabolic_case_dim1_fourier_support} below.

\begin{proof}
	The first step of the proof consists in showing that
	\begin{equation}\label{incl_proof_pro_projection_parabolic_case_dim1_fourier_support}
		\CS \CM(P_b) \subset \set{\mu \in \CP_{P_b}(\T^{2d}) \, | \, \exists ! m \in \N, \, \supp(\widehat{\mu}) \subset \set{-m ,0, m} \times \Z},
	\end{equation}
	using the structure of the eigenspaces of $M_N(P_b)$. This will imply the form claimed in the theorem.

	Let $\mu$ be a semiclassical measure associated to $P_b$. There exists a
	sequence $(\psi_{N_k})_k$ of normed eigenvectors of $M_{N_k}(P_b)$, associated
	to the eigenvalues $(\lambda_k)_k$ on $\CH_{N_k}(\theta_k)$, such that
	\begin{equation*}
		Wig(\psi_{N_k}) \xrightharpoonup[k \to + \infty]{}^* \mu.
	\end{equation*}

	Using Lemma \ref{lem_properties_basis_f_theta_j} and observing that the phase
	in front of the sum tends to one as $N_k$ tends to $+ \infty$, we obtain that
	\begin{align*}
		\widehat{\mu}(m,n) & = \lim_{k \to + \infty} Wig(\psi_{N_k})(e^{-2\pi i \sprod{(m,n)}{\cdot}}),\\
		 & = \lim_{k \to + \infty} \sum_{j \in \eint{0}{N_k-1}^d} \psi_{N_k,j} \overline{\psi_{N_k,\modulo{j-m}_{N_k}}} e^{-\frac{2 \pi i }{N_k}\sprod{n}{j}} e^{-\frac{2 \pi i}{N_k} \sprod{\theta_\xi}{j-\modulo{j-m}_{N_k}}}.
	\end{align*}
	This limit can be nonzero if and only if there exists a sequence of indices $(j_k)_k$ such that $\psi_{N_k,j_k} \overline{\psi_{N_k,\modulo{j_k-m}_{N_k}}} \neq 0$. Looking at the form of the eigenspace $E_{\lambda_k}$ of $M_{N_k}(P_b)$, we see that both $f^{\theta_k}_{j_k}$ and $f^{\theta_k}_{j_k-m}$ belong to $E_{\lambda_k}$. We used that $f^{\theta_k}_{\modulo{j_k-n}_{N_k}}$ is equal to $f^{\theta_k}_{j_k-m}$ up to a phase according to \eqref{eq_transformation_rule_e_theta_j}.

	Suppose, by contradiction, that there exists $m_1 \neq 0$, $m_2 \neq \pm m_1$
	and $n_1, n_2 \in \Z$ such that $\widehat{\mu}(m_1,n_1) \neq 0$ and
	$\widehat{\mu}(m_2,n_2) \neq 0$. By the previous discussion, there exist
	sequences $(j_{k,1})_k$ and $(j_{k,2})_k$ such that
	\begin{equation*}
		f^{\theta_k}_{j_{k,1}}, \quad f^{\theta_k}_{j_{k,2}}, \quad f^{\theta_k}_{j_{k,1} - m_1} \quad f^{\theta_k}_{j_{k,2} - m_2} \; \in E_{\lambda_k}.
	\end{equation*}
	According to \eqref{cond_e_0_j_same_eigenvalue_dim_1}, this implies that
	\begin{equation*}
		\begin{syst}
			bj_{k,1}^2 - 2b \theta_{k,\xi} j_{k,1} &\equiv  b(j_{k,1}-m_1)^2 - 2b \theta_{k,\xi} (j_{k,1}-m_1) \quad \modulo{2N_k} ,\\
			bj_{k,2}^2 - 2b \theta_{k,\xi} j_{k,2} &\equiv  b(j_{k,2}-m_2)^2 - 2b \theta_{k,\xi} (j_{k,2}-m_2) \quad \modulo{2N_k} ,\\
			bj_{k,1}^2 - 2b \theta_{k,\xi} j_{k,1} &\equiv bj_{k,2}^2 - 2b \theta_{k,\xi} j_{k,2} \quad \modulo{2N_k}.
		\end{syst}
	\end{equation*}
	However, this system does not admit solutions $(j_{k,1}, j_{k,2})$ for infinitely many values of $N_k$, for each possible value of $2 b \theta_\xi$, according to Lemma \ref{lem_no_solutions_system_2_Fourier_coeff_parabolic_1d}. We are thus done as $2 b \theta_{k,\xi}$ can only take a finite number of integer values since $\theta_\xi$ must satisfy the conditions \eqref{cond_quantizability_parabolic_one_dim}.

	The second step consists in showing that if $\mu \in \CP_{P_b}(\T^{2d})$ is such
	that there exists a unique $m \in \N$ with $\supp (\widehat{\mu}) \subset
		\set{-m ,0, m} \times \Z $, then $\mu$ is of the form
	\eqref{eq_pro_projection_parabolic_case_dim1_form_measure}. Suppose first that
	$m =0$, this means that $\supp(\widehat{\mu}) \subset \set{0} \times \Z $.
	Denote by $\nu$ the projection of $\mu$ on the second variable of $\T^2$. Its
	Fourier coefficients satisfy $\widehat{\nu}(n) = \widehat{\mu}(0,n)$ for any $n
		\in \Z$ and thus the Fourier coefficients of $\mu$ and $\Leb \otimes \nu$ are
	equal. This implies that $\mu = \Leb \otimes \nu$ and thus $\mu$ is of the form
	\eqref{eq_pro_projection_parabolic_case_dim1_form_measure}.

	We suppose now that $m \neq 0$ and we denote by $\nu$ the projection of $\mu$
	on the second coordinate of the torus $\T^2$. The disintegration of measure
	theorem \cite[thm. 5.3.1 for the precise
		statement]{ambrosioGradientFlowsMetric2008}\cite[III-70 for a
		proof]{dellacherieProbabilitiesPotential1978} gives that there exists a
	$\nu$-almost everywhere uniquely determined family of probability measures
	$(\mu_\xi)_{\xi \in \T}$ on $\T$ that satisfies
	\begin{equation*}
		\int_{\T^2} f(x, \xi) \diff \mu = \int_{\T} \int_{\T} f(x, \xi) \diff \mu_\xi(x) \diff \nu(\xi) \qtq{for} f \text{ Borel measurable}.
	\end{equation*}
	The disintegration theorem gives $\mu_\xi$ as a measure on $\T^2$, but we identified $\mu_\xi$ with its restriction to $\T \times \set{\xi}$ as the disintegration theorem also implies that $\mu_\xi$ is supported in $\T \times \set{\xi}$ for $\nu$-almost all $\xi \in \T$. Let $f$ be Borel measurable. We use the invariance of $\mu$ under $P_b$ to obtain
	\begin{align*}
		\int_{\T^2} f(x, \xi) \diff \mu & = \int_{\T^2} f(x + b \xi, \xi ) \diff \mu =  \int_{\T} \int_{\T} f(x + b \xi, \xi) \diff \mu_\xi(x) \diff \nu(\xi) ,\\
		& = \int_{\T} \int_{\T} f(x, \xi) \diff \tau_{b\xi} \mu_\xi(x) \diff \nu(\xi).
	\end{align*}
	Since the measures $(\mu_\xi)_{\xi \in \T}$ are uniquely determined $\nu$-almost everywhere, there exists a set $N$ of $\nu$-measure zero such that
	\begin{equation*}
		\forall \xi \in \T \setminus N, \quad \mu_{\xi} = \tau_{b\xi} \mu_{\xi}.
	\end{equation*}
	Let $l \in \Z$ and $\xi \in \T \setminus N$. We compute
	\begin{align*}
		\widehat{\mu_\xi}(l) = \int_{\T} e^{-2 \pi i l x} \diff \mu_\xi(x) = \int_{\T} e^{-2 \pi i l (x +b \xi)} \diff \mu_\xi(x) = e^{-2 \pi i b\xi l} \widehat{\mu_\xi}(l).
	\end{align*}
	Therefore,
	\begin{equation*}
		\forall l \in \Z^*, \quad \widehat{\mu_\xi}(l) \neq 0 \implies \xi \in \frac{1}{bl} \Z.
	\end{equation*}
	Let $l, n \in \Z$ with $l \notin \set{-m ,0,m}$. We have, as $\supp (\widehat{\mu}) \subset \set{-m ,0,m} \times \Z$, that
	\begin{align*}
		0 & = \widehat{\mu}(l,n) = \int_{\T} e^{-2 \pi i n \xi} \int_{\T} e^{- 2\pi i l x} \diff \mu_\xi(x) \diff \nu(\xi) = \int_{\T} e^{-2 \pi i n \xi} \widehat{\mu_\xi}(l) \diff \nu(\xi).
	\end{align*}
	As this is valid for any $l \in \Z$, we obtain, by density of the trigonometric polynomials in the continuous functions on the torus, that
	\begin{equation*}
		\forall f \in \CC^0(\T), \qquad \int_{\T} f(\xi) \widehat{\mu_\xi}(m) \diff \nu(\xi) = 0.
	\end{equation*}
	Since $\nu$ is a finite Borel measure on a compact space, it is regular \cite[Thm. 2.18]{rudinRealComplexAnalysis1986} and thus the continuous functions are dense in $L^1(\nu)$ \cite[thm. 3.14]{rudinRealComplexAnalysis1986}. Therefore,
	\begin{equation}\label{eq_proof_pro_projection_parabolic_case_dim1_coeff_fourier_disintegration}
		\forall \xi \in \T \setminus \tilde{N},\, l \notin \set{-m ,0,m} \quad \widehat{\mu_\xi}(l) = 0 \qqtqq{where} \nu(\tilde{N}) = 0.
	\end{equation}
	Thus, if $\xi \in \T \setminus \tilde{N}$ and $\xi \notin \frac{1}{bm}\Z$, we must have $\mu_\xi = \Leb$. Indeed, as $\xi \notin \frac{1}{bm} \Z$, we have $\widehat{\mu_\xi}(\pm m) =0$ and all other Fourier coefficients, except $\widehat{\mu_\xi}(0)$, are zero. Combining everything and denoting $X = \set{0, \dots, (bm-1)/bm}$, we obtain
	\begin{equation*}
		\mu = \nu(X^c) \Leb \otimes \tilde{\nu}  + \sum_{l=0}^{bm-1} \nu(\frac{l}{bm}) \mu_{\frac{l}{bm}} \otimes \delta_{\frac{l}{bm}} \qqtqq{where} \tilde{\nu} = \frac{\mathbb{1}_{X^c} \nu }{\nu(X^c)}.
	\end{equation*}
	No problem arises in the definition of $\nu(l/bm) \mu_{l/bm}$ in the expression above, as 
	either~$\nu(l/bm) \neq 0$ and $\mu_{l/bm}$ is well-defined, or $\nu(l/bm) = 0$.  Let $l \in \eint{0}{bm-1}$ such that $\nu(\frac{l}{bm}) \neq 0$. We have thanks to \eqref{eq_proof_pro_projection_parabolic_case_dim1_coeff_fourier_disintegration} that $\widehat{\mu_{l/bm}}$ is supported on $\set{-m, 0,m}$. As $\mu_{l/bm}$ is a probability measure, we have $\widehat{\mu_{l/bm}}(0)=1$ and $\widehat{\mu_{l/bm}}(-m) = \overline{\widehat{\mu_{l/bm}}(m)} = \beta_{l}' e^{-i \phi_l}$ with $\beta_{l}' \ge 0$. Since those are the only nonzero Fourier coefficients of $\mu_{l/bm}$, we get
	\begin{align*}
		\mu_{l/bm} & = (\beta_{l}' e^{-i \phi_l} e^{-2 \pi i m x} + 1 + \beta_{l} e^{i \phi_l} e^{2 \pi i m x}) \Leb ,\\
		& = (1+2 \beta_{l}' \cos(2 \pi m x + \phi_l)) \Leb.
	\end{align*}
	As $\mu_{l/bm}$ is a positive measure, we must have $ \beta_{l} = 2 \beta_{l}' \le 1$. Finally, we obtain the desired expression
	\begin{equation*}
		\mu = \nu(X^c) \Leb \otimes \tilde{\nu}  + \sum_{l=0}^{bm-1} \nu(\frac{l}{bm}) (1+ \beta_l \cos(2 \pi m x + \phi_l)) \Leb \otimes \delta_{\frac{l}{bm}} \qtq{where} \tilde{\nu} = \frac{\mathbb{1}_{X^c} \nu }{\nu(X^c)}.
	\end{equation*}
\end{proof}

We now prove a reverse result, namely that all the measures in Proposition
\ref{pro_projection_parabolic_case_dim1} are indeed semiclassical measures for
$P_b$.

\begin{proposition}\label{pro_projection_parabolic_case_dim1_building_the_measure}
	Let $b \in \N^*$. For any $\nu \in \CP(\T)$, $\gamma \in [0,1]$, $m \in \N^*$, $(\alpha_l, \beta_l, \phi_l)_{l \in \eint{0}{bm-1}} \in ([0,1]^2 \times \R)^{bl}$ with $\gamma + \sum_{l=0}^{bm-1} \alpha_l = 1$, we have that
	\begin{equation*}
		\gamma  \Leb_x \otimes \nu + \sum_{l=0}^{bm-1} \alpha_l (1 + \beta_l \cos(2 \pi m x + \phi_l)) \Leb_x \otimes \delta_{\frac{l}{bm}}  \in \CS \CM(P_b).
	\end{equation*}
\end{proposition}

\begin{proof}
	Note that the case $\sum_{l=0}^{bm-1} \alpha_l = 0$ has already been treated in Proposition \ref{pro_parabolic_nu_times_lebesgue}. Thus, we suppose that $\sum_{l=0}^{bm-1} \alpha_l > 0$ in the following. We first explicitly construct a sequence of eigenvectors $(\psi_k)_{k} \in \CH_{N_k}(\theta_k)$ of $M_{N_k}(P_b)$ whose semiclassical measure is
	\begin{equation*}
		\mu = \sum_{l=0}^{bm-1} \alpha_l (1 + \beta_l \cos(2 \pi m x + \phi_l)) \Leb_x \otimes \delta_{\frac{l}{bm}}.
	\end{equation*}
	A simple computation gives
	\begin{equation*}
		\widehat{\mu}(r,s) = \sum_{l=0}^{bm-1} \alpha_l e^{-2 \pi i s \frac{l}{bm}} (\beta_l e^{-i \phi_l} \delta_{r,-m} + \delta_{0,r} + \beta_l e^{i \phi_l} \delta_{r,m}).
	\end{equation*}
	The first step consists in solving the equation
	\begin{equation}\label{eq_proof_pro_projection_parabolic_case_dim1_building_the_measure}
		bj^2_k -2 b \theta_{k,\xi}j \equiv b(j_k-m)^2 -2 b \theta_{k,\xi}(j_k-m) \quad \modulo{2N_k},
	\end{equation}
	with $\theta_k$ that satisfies the conditions \eqref{cond_quantizability_parabolic_one_dim}. There are four cases to consider:
	\begin{enumerate}
		\item if $b$ and $m$ are even, we choose $\theta = 0$,
		\item if $b$ is even and $m$ is odd, we choose $\theta_{k,\xi} = 1/2 = (b/2)/b$,
		\item if $b$ is odd and $m$ is even, we choose $N_k$ even and $\theta_{k,\xi} = 0$,
		\item if $b$ and $m$ are odd, we choose $N_k$ odd and $\theta_{k,\xi} = \frac{b}{2b}
			      = 1/2$.
	\end{enumerate}
	We always choose $\theta_{k,x} = 0$ to ensure that $f^\theta_j = f^\theta_{j+N_k}$ per \eqref{eq_transformation_rule_e_theta_j} and \eqref{eq_new_basis_f_theta_j}. In all cases, we obtain that the equation \eqref{eq_proof_pro_projection_parabolic_case_dim1_building_the_measure} is equivalent to
	\begin{equation*}
		2bm \left(j_k - \frac{m + 2 \theta_{k,\xi}}{2}\right) \equiv 0 \quad \modulo{2N_k}.
	\end{equation*}
	We then choose $N_k = bm^2 k$ in all cases with $k$ odd. The solutions of \eqref{eq_proof_pro_projection_parabolic_case_dim1_building_the_measure} are then
	\begin{equation*}
		j_{k,l} = \frac{m + 2 \theta_{k,\xi}}{2} + l \frac{N_k}{bm} \qqtq{where} l \in \eint{0}{bm-1}.
	\end{equation*}
	Suppose first that $m$ is even and remark that
	\begin{align*}
		b j_{k,l}^2 & \equiv b(\frac{m}{2} + l \frac{b m^2k}{bm})^2 \equiv  b\left(\frac{m}{2}\right)^2 + bm^2k l(lk+1)  \quad \modulo{2 bm^2 k} ,\\
		& \equiv  b\left(\frac{m}{2}\right)^2 \quad \modulo{2b m^2 k}.
	\end{align*}
	We used that $l(kl+1)$ is even as either $l$ is even, or $l$ is odd and thus $(kl+1)$ is even. If $m$ is odd, we have similarly
	\begin{align*}
		bj_{k,l}^2 - bj_{k,l} & \equiv bj_{k,l}(j_{k,l}- 1) \equiv b \frac{m-1}{2}\frac{m+1}{2} + b m^2 k l (kl + 1) \quad \modulo{2bm^2k} ,\\
		 & \equiv b \frac{m-1}{2}\frac{m+1}{2} \quad \modulo{2bm^2k}.
	\end{align*}
	We have thus proved that in all cases and for any $l \in \eint{0}{bm-1}$
	\begin{equation}\label{eq_proof_pro_projection_parabolic_case_dim1_building_the_measure_second}
		bj_{k,0}^2 - 2b \theta_{k,\xi} j_{k,0} \equiv bj_{k,l}^2 - 2b \theta_{k,\xi} j_{k,l} \equiv b(j_k-m)^2 -2 b \theta_{k,\xi}(j_k-m)  \quad \modulo{2N_k}.
	\end{equation}
	We define a sequence of vectors $(\psi_k)_k$ in $\CH_{N_k}(\theta_k)$ by
	\begin{equation*}
		\psi_k = \sum_{l=0}^{bm-1} \sqrt{\alpha_l} \left(\sqrt{\frac{1+ \sqrt{1-\beta_l}}{2}} e^{i \phi_l} f^{\theta_k}_{j_{k, l}} + \sqrt{\frac{1- \sqrt{1-\beta_l}}{2}} f^{\theta_k}_{j_{k,l}-m} \right).
	\end{equation*}
	Each $\psi_k$ is an eigenvector of $M_{N_k}(P_b)$ thanks to \eqref{cond_e_0_j_same_eigenvalue_dim_1} and \eqref{eq_proof_pro_projection_parabolic_case_dim1_building_the_measure_second}. As all $(j_{k,l})_{l \in \eint{0}{bm-1}}$ are different, we get
	\begin{equation}\label{eq_proof_pro_projection_parabolic_case_dim1_building_the_measure_norm_psi_k}
		\norm{\psi_k}^2 = \sum_{l=0}^{bm-1} \alpha_l\left(\frac{1+ \sqrt{1-\beta_l}}{2} + \frac{1 -\sqrt{1-\beta_l}}{2} \right) =\sum_{l=0}^{bm-1} \alpha_l.
	\end{equation}
	Let $(\psi_{k,j})$ be the components of $\psi_k$ in the basis $(f^{\theta_k}_j)_j$. The equation \eqref{eq_simplified_expression_fourier_coefficients_wigner_distribution_f_theta_j} gives
	\begin{align*}
		Wig(\psi_k)(e^{-2 \pi i \sprod{(r,s)}{\cdot}}) & = e^{\frac{\pi i}{N_k} r (s+2 \theta_{k,\xi})} \sum_{j \in \Z/N_k\Z} e^{-\frac{2 \pi i}{N_k} sj} \psi_{k,j} \overline{\psi_{k,j-r}} \qqtq{for} r,s \in \Z.
	\end{align*}
	If $r \notin \set{-m,0,m}$, this sum is zero if $k$ is large enough, as all $(j_{k,l})_l$ and $(j_{k,l}-r)_l$ are different from each other modulo $N_k$.
	If $r = 0$, this sum becomes
	\begin{align*}
		e^{\frac{\pi i}{N_k} r (s+ 2 \theta_{k,\xi})} \sum_{l=0}^{bm-1} \alpha_l \left( e^{-\frac{2 \pi i}{N_k} sj_{k,l}} \left|\sqrt{\frac{1+ \sqrt{1-\beta_l}}{2}} e^{i \phi_l} \right|^2 + e^{-\frac{2 \pi i}{N_k} sj_{k,l}} \left|\sqrt{\frac{1- \sqrt{1-\beta_l}}{2}} \right|^2  \right),
	\end{align*}
	and after some simplifications, we get
	\begin{align*}
		Wig(\psi_k)(e^{-2 \pi i \sprod{(r,0)}{\cdot}}) & = e^{\frac{\pi i}{N_k} r (s+ 2 \theta_{k,\xi})} \sum_{l=0}^{bm-1} \alpha_l e^{-\frac{2 \pi i}{N_k} s(\frac{m + 2 \theta_{k,\xi}}{2} + l \frac{N_k}{bm})} ,\\
		& \xrightarrow[k \to + \infty]{} \sum_{l=0}^{bm-1} \alpha_l e^{-2 \pi i s\frac{l}{bm}} = \widehat{\mu}(0,s).
	\end{align*}
	For the case $r = m$, we remark that for $N_k$ large enough, the only indices $j$ such that both $j$ and $j-m$ belong to the set $\set{j_{k,l}, j_{k,l } -m \, | \, l \in \eint{0}{bm-1}}$ are $(j_{k,l})_{l \in \eint{0}{bm-1}}$, where everything is understood modulo $N_k$. Moreover,
	\begin{align*}
		Wig(\psi_k)(e^{-2 \pi i \sprod{(m,s)}{\cdot}}) & = e^{\frac{\pi i}{N_k} r (s+ 2 \theta_{k,\xi})} \sum_{l=0}^{bm-1} \alpha_l e^{-\frac{2 \pi i}{N_k} s j_{k,l}} \sqrt{\frac{1+ \sqrt{1-\beta_l}}{2}} e^{i \phi_l} \sqrt{\frac{1- \sqrt{1-\beta_l}}{2}} ,\\
		& = e^{\frac{\pi i}{N_k} r (s+ 2 \theta_{k,\xi})} \sum_{l=0}^{bm-1} \alpha_l \beta_l e^{i \phi_l} e^{-\frac{2 \pi i}{N_k} s(\frac{m - 2 \theta_{k,\xi}}{2} + l \frac{N_k}{bm})},\\
		& \xrightarrow[k \to + \infty]{} \sum_{l=0}^{bm-1} \alpha_l e^{-2 \pi i s \frac{l}{bm}} \beta_l e^{i \phi_l} = \widehat{\mu}(m,s).
	\end{align*}
	We obtain a similar result for $r=-m$ and finally that for any $r,s \in \Z$
	\begin{equation*}
		Wig(\psi_k)(e^{-2 \pi i \sprod{(r,s)}{\cdot}}) \xrightarrow[k \to + \infty]{}  \widehat{\mu}(r,s) = \mu(e^{-2 \pi i \sprod{(r,s)}{\cdot}}).
	\end{equation*}
	This implies that $Wig(\psi_k)$ converges as a distribution to $\mu$ by density of the trigonometric polynomials in $\CC^\infty(\T^{2d})$.

	We now fix a finite set of distinct points $(\xi_i)_{i \in \eint{1}{r}}$ in $\T
		\setminus \set{0, \dots, \frac{bm-1}{bm}}$ and construct a sequence of
	eigenvectors $(\tilde{\psi}_k)_{k} \in \CH_{N_k}(\theta_k)$ of $M_{N_k}(P_b)$
	whose semiclassical measure is
	\begin{equation*}
		\sum_{i=1}^r \gamma_i \Leb \otimes \delta_{\xi_i}  + \sum_{l=0}^{bm-1} \alpha_{l} (1 + \beta_l \cos(2 \pi m x + \phi_l))\Leb \otimes \delta_{\frac{l}{bm} } \qtq{with} \sum_{i=1}^r \gamma_i + \sum_{l=0}^{bm-1} \alpha_l = 1.
	\end{equation*}
	We first consider the case $m$ even. Recall that in this case, we take $\theta_{k,\xi} = 0$ and $N_k = bm^2k$. The first step is to find elements $s_{k,i} \in \Z_{N_k}$ such that
	\begin{equation}\label{eq_proof_condition_construction_using_equidistribution_even}
		\frac{s_{k,i}}{N_k} \xrightarrow[k \to + \infty]{} \xi_i \qqtqq{and} bs^2_{k,i} \equiv b j_{k,0}^2 \quad \modulo{2N_k}.
	\end{equation}
	The second equation is there to ensure that $f^{\theta_{k}}_{s_{k,i}}$ is in the same eigenspace of $M_{N_k}(P_b)$ as the $(f^{\theta_{k}}_{j_{k,l}})_{l \in \eint{0}{bm-1}}$. As $j_{k,0}^2 = (m/2)^2$, we get
	\begin{align*}
		bs^2_{k,i} \equiv b j_{k,0}^2 \quad \modulo{2N_k} \iff b(s_{k,i} - \frac{m}{2})(s_{k,i} + \frac{m}{2}) \equiv 0 \quad \modulo{2bm^2 k}.
	\end{align*}

	We apply the following lemma to choose a good sequence of $N_k$.
	\begin{lemma}
		Let $a,b \in \Z^*$ and denote $\CA_N = \set{x \in \Z/N\Z \, | \, b(x-a)(x+a) \equiv 0 \modulo{N}}$. There exists a subset $S$ of density one of $\N^*$ such that
		\begin{equation*}
			\frac{1}{|\CA_N|}\sum_{x \in \CA_N} \delta_{\frac{x}{N}} \xrightharpoonup[N \to + \infty, N \in S]{}^* \Leb_{[0,1]}.
		\end{equation*}
	\end{lemma}
	\begin{proof}
		The lemma is a direct consequence of \cite[example 1]{kowalskiEquidistributionChineseRemainder2021a}, the fact that Cesàro convergence implies convergence along a subset of density one, and that the convergence of the discrepancy to $0$ implies convergence to the Lebesgue measure.
	\end{proof}
	Using this lemma, we obtain a subset $S$ of the integers of density one such that
	\begin{equation*}
		\lambda_N := \frac{1}{|\CA_N|}\sum_{x \in \CA_N} \delta_{\frac{x}{N}} \xrightharpoonup[\substack{N \to + \infty ,\\ N \in S}]{}^* \Leb_{[0,1]},
	\end{equation*}
	where $\CA_N = \set{x \in \Z/N\Z \, | \, b(x-\frac{m}{2})(x+\frac{m}{2}) \equiv 0 \modulo{N}}$.
	As the set of integers of the form $2bm^2k$ with $k$ odd is of density $\frac{1}{4bm^2}$, the set $\tilde{S}$ defined by
	\begin{equation*}
		\tilde{S} := S \cap \set{n \in \N \, | \, \exists r \in \N, \, n =2bm^2(2r+1)},
	\end{equation*}
	must be infinite. Finally, the equidistribution of the measures $(\lambda_N)_{N \in \tilde{S}}$ implies that there exist sequences $(s_{k,i})_{N_k \in \tilde{S}}$ such that
	\begin{equation*}
		\frac{s_{k,i}}{N_k} \xrightarrow[k \to + \infty]{} \xi_i \qqtqq{and} b(s_{k,i} - \frac{m}{2})(s_{k,i} + \frac{m}{2}) \equiv 0 \quad \modulo{2bm^2 k}.
	\end{equation*}

	Now that we have solutions of
	\eqref{eq_proof_condition_construction_using_equidistribution_even}, we define
	an eigenvector $\tilde{\psi}_k$ of $M_{N_k}(P_b)$ by
	\begin{align*}
		\tilde{\psi}_k = \sum_{i=1}^r \sqrt{\gamma_i} f^{\theta_k}_{s_{k,i}} + \psi_k.
	\end{align*}
	By Lemma \ref{lem_properties_basis_f_theta_j} and our conditions \eqref{eq_proof_condition_construction_using_equidistribution_even}, we have that
	\begin{equation*}
		Wig(e^{\theta_k}_{s_{k,i}}) \xrightharpoonup[k \to +\infty]{}^* \Leb \otimes \delta_{\xi_i} \qqtqq{for} i \in \eint{1}{r}.
	\end{equation*}
	Recall that we proved in the first part that $\norm{\psi_k}^2 = \sum_{l=0}^{bm-1} \alpha_l$ \eqref{eq_proof_pro_projection_parabolic_case_dim1_building_the_measure_norm_psi_k}. Moreover,
	\begin{equation*}
		Wig(\frac{\psi_k}{\sqrt{\sum_{l=0}^{bm-1} \alpha_l}}) \xrightharpoonup[k \to +\infty]{}^* \frac{1}{\sum_{l=0}^{bm-1} \alpha_l} \sum_{l=0}^{bm-1} \alpha_l (1 + \beta_l \cos(2 \pi m x + \phi_l)) \Leb_x \otimes \delta_{\frac{l}{bm}}.
	\end{equation*}
	Note that the right-hand side is a probability measure, hence we rewrite $\tilde{\psi}_k$ as
	\begin{equation*}
		\tilde{\psi}_k = \sum_{i=1}^r \sqrt{\gamma_i} e^{\theta_k}_{s_{k,i}} + \sqrt{\sum_{l=0}^{bm-1} \alpha_l} \frac{\psi_k}{\sqrt{\sum_{l=0}^{bm-1} \alpha_l}}.
	\end{equation*}
	Since we supposed that $\xi_i \notin \set{0, \dots, \frac{bm-1}{bm}}$, all the measures above are mutually singular and Lemma \ref{lem_convex_combination_orthogonal_semiclassical_measures} implies that
	\begin{equation*}
		Wig \left(\frac{\tilde{\psi}_k}{\norm{\tilde{\psi}_k}}\right) \xrightharpoonup[k \to +\infty]{}^* \Leb \otimes \left(\sum_{i=1}^r \gamma_i \delta_{\xi_i}\right) + \sum_{l=0}^{bm-1} \alpha_l (1 + \beta_l \cos(2 \pi m x + \phi_l)) \Leb_x \otimes \delta_{\frac{l}{bm}}.
	\end{equation*}
	We have thus proven that
	\begin{equation*}
		\Leb \otimes \left(\sum_{i=1}^r \gamma_i \delta_{\xi_i}\right) + \sum_{l=0}^{bm-1} \alpha_l (1 + \beta_l \cos(2 \pi m x + \phi_l)) \Leb_x \otimes \delta_{\frac{l}{bm}}.
	\end{equation*}
	This is true for any choice of distinct points $(\xi_i)_{i \in \eint{1}{r}}$ in $\T \setminus \set{0, \dots, \frac{bm-1}{bm}}$. As the set of semiclassical measures $\CS \CM(P_b)$ is closed for the weak-star topology of probability measures on $\T$ and the set of convex combinations of Dirac measures supported on a dense subset of $\T$ is dense for this topology in the set of all probability measures on $\T$, we obtain that
	\begin{equation*}
		\gamma \Leb \otimes \nu  + \sum_{l=0}^{bm-1} \alpha_l (1 + \beta_l \cos(2 \pi m x + \phi_l)) \Leb_x \otimes \delta_{\frac{l}{bm}} \in \CS \CM(P_b) \qqtq{for} \nu \in \CP(\T).
	\end{equation*}

	It remains to treat the case $m$ odd. Recall that in the case $m$ odd, we take
	$\theta_{k,\xi} = 1/2$ and $N_k = bm^2k$ with $k$ odd. The proof will be the
	same except that we need to find points $s_{k,i} \in \Z_{N_k}$ such that
	\begin{equation}\label{eq_proof_condition_construction_using_equidistribution_odd}
		\frac{s_{k,i}}{N_k} \xrightarrow[k \to + \infty]{} \xi_i \qqtqq{and} bs^2_{k,i} - b s_{k,i} \equiv b j_{k,0}^2 - b j_{k,0} \quad \modulo{2N_k}.
	\end{equation}
	The equation is equivalent to
	\begin{align*}
		 & \quad b s_{k,i} (s_{k,i} -1) \equiv b \frac{m-1}{2}\frac{m+1}{2} \quad \modulo{2bm^2k} ,\\
		 & \iff b(s_{k,i} + \frac{m-1}{2})(s_{k,i} -\frac{m+1}{2}) \equiv 0 \quad \modulo{2bm^2k}.
	\end{align*}
	The proof then follows the same line as the case $m$ even.
\end{proof}

\subsection{Quantum variance and partial quantum ergodicity}

This section is dedicated to the proof of Theorem
\ref{thm_upper_bound_quantum_variance_position_parabolic_matrices} and Theorem
\ref{thm_lower_bound_quantum_variance_position_parabolic_matrices}.

\begin{proof}
	We suppose without loss of generality that $\overline{a} = 0$.
	Let $T \in \N$. As a consequence of the exact Egorov theorem \eqref{eq_exact_Egorov_observable_torus} and the fact that the $\psi_{j}$ are eigenvectors of $M_N(P_B)$, we obtain that
	\begin{equation*}
		\sprod{\Op^w_N(a) \psi_j }{\psi_j} = \sprod{M_N(P_B)^{-1}\Op^w_N(a) M_N(P_B) \psi_j }{\psi_j} = \sprod{\Op^w_N(a \circ P_B) \psi_j }{\psi_j}.
	\end{equation*}
	Therefore, defining $\jap{a}_T = \frac{1}{T} \sum_{s=0}^{T-1} a\circ P_B^s$, we have
	\begin{equation*}
		\sprod{\Op^w_N(a) \psi_j }{\psi_j} = \frac{1}{T} \sum_{s=0}^{T-1} \sprod{\Op^w_N(a\circ P_B^s) \psi_j }{\psi_j} = \sprod{\Op^w_N(\jap{a}_T) \psi_j }{\psi_j}.
	\end{equation*}
	Moreover, as $\norm{\psi_j} =1$, a direct application of the Cauchy-Schwarz inequality gives
	\begin{equation*}
		\left|\sprod{\Op^w_N(\jap{a}_T) \psi_j }{\psi_j} \right|^2 \le \sprod{\Op^w_N(\jap{a}_T)^* \Op^w_N(\jap{a}_T) \psi_j}{\psi_j}.
	\end{equation*}
	Combining all these observations leads to
	\begin{align*}
		V(a,N, (\psi_j)_j) & \le \frac{1}{N^d} \sum_{j\in \eint{1}{N}^d}\sprod{\Op^w_N(\jap{a}_T)^* \Op^w_N(\jap{a}_T) \psi_j}{\psi_j} ,\\
		& = \frac{1}{N^d}  \tr(\Op^w_N(\jap{a}_T)^* \Op^w_N(\jap{a}_T)).
	\end{align*}
	Note that
	\begin{align*}
		\jap{a}_T(x,\xi) &= \frac{1}{T} \sum_{t=0}^{T-1} \sum_{m \in \Z^d \setminus \set{0}} a_m e^{2 \pi i \sprod{m}{x + tB\xi}} = \sum_{(m,n) \in \Z^{2d}\setminus \set{0}} \frac{a_m}{T} \sum_{t=0}^{T-1} \delta_{n,t B m} e^{2 \pi i \sprod{(m,n)}{(x,\xi)}},
	\end{align*}
	thus the Fourier coefficients of $\jap{a}_T$ are
	\begin{equation*}
		\widehat{\jap{a}_T}(m,n) = a_m \left( \frac{1}{T} \sum_{t=0}^{T-1} \delta_{n,t B m} \right) \qqtq{for} (m,n) \in \Z^{2d} \setminus \set{0}.
	\end{equation*}
	The computation of the Hilbert-Schmidt norm in Lemma \ref{lem_trace_quantum_translation_HS_norm_Weyl_quantization} leads to
	\begin{align*}
		V(a,N, (\psi_j)_j) & \le \sum_{(m,n), (k,l) \in \Z^{2d}} |\widehat{\jap{a}_T}(m,n)| |\widehat{\jap{a}_T}(k,l)| \mathbb{1}_{N \Z^{2d}}(m-k, n-l),\\
		 & \le \sum_{(m,n), (k,l) \in \Z^{2d} \setminus \set{0}} |a_m| |a_k| \left(\frac{1}{T^2}\sum_{s,t= 0}^{T-1} \delta_{n, s Bm} \delta_{l, t Bk} \right) \mathbb{1}_{N \Z^{2d}}(m-k, n-l) ,\\
		 & \le \sum_{m, k \in \Z^{d} \setminus \set{0}} |a_m| |a_k| \frac{1}{T^2}\sum_{s,t= 0}^{T-1}  \mathbb{1}_{N \Z^{2d}}(m-k,s Bm-tBk) ,\\
		 & \le \sum_{m, k \in \Z^{d} \setminus \set{0}} |a_m| |a_k| \mathbb{1}_{N \Z^{2d}}(m-k) \frac{1}{T^2}\sum_{s,t = 0}^{T-1} \mathbb{1}_{N \Z^{2d}}((s-t) Bm).
	\end{align*}
	We let $T$ go to $+ \infty$ and we deduce from Lemma \ref{lem_limit_average_indicator_bound_quantum_variance} that
	\begin{align*}
		V(a,N, (\psi_j)_j) & \le \sum_{m, k \in \Z^{d} \setminus \set{0}} |a_m| |a_k| \mathbb{1}_{N \Z^{2d}}(m-k) \frac{N \wedge gcd(\det(B) m)}{N} ,\\
		 & \le \frac{|\det(B)|}{N} \sum_{m, k \in \Z^{d} \setminus \set{0}} \norm{m}_{\infty} |a_m| |a_k| ,\\
		 & \le C_{s,d} \frac{|\det(B)|}{N} \norm{a}_{H^{s+1}(\T^d)} \norm{a}_{H^s(\T^d)},
	\end{align*}
	for any $s > d/2$. We used that for any $s > d/2$ and any $\alpha \ge 0$,
	\begin{align*}
		\sum_{m \in \Z^d \setminus \set{0}} \norm{m}^\alpha |a_m| \le \norm{a}_{H^{s+\alpha}(\T^d)} \sqrt{\sum_{m \in \Z^d \setminus \set{0}} \frac{1}{\norm{m}^{2s}}} \le C_{s,d} \norm{a}_{H^{s+\alpha}(\T^d)}.
	\end{align*}
\end{proof}

\begin{remark}\label{rem_proof_thm_upper_bound_quantum_variance_position_parabolic_matrices}
	The last step of the proof is probably not optimal and it seems feasible to improve the Sobolev norm used. One could use a better bound on $N \wedge gcd(\det(B)m)$, for example by saying that it is smaller than the smallest nonzero coordinate of $\det(B) m$, or by studying more precisely the arithmetic properties of $N \wedge gcd(\det(B)m)$ and seeing that it is often just one. Finally, we got rid of $\mathbb{1}_{N \Z^{2d}}(m-k)$ and there is probably a better way than just bounding it by $1$.
\end{remark}

The deduction of Corollary \ref{cor_partial_quantum_ergodicity_position}
from Theorem
\ref{thm_upper_bound_quantum_variance_position_parabolic_matrices} is standard
and we let the reader adapt the arguments of \cite[sec.
	5]{colindeverdiereErgodiciteFonctionsPropres1985} to the present context.

It remains to prove Theorem
\ref{thm_lower_bound_quantum_variance_position_parabolic_matrices} that gives
a lower bound for the quantum variance.

\begin{proof}
	The goal is to build a sequence of eigenbases $(\psi_{k,j})_{k \in \N, j \in \eint{0}{N_k-1}^d}$ of $M_{N_k}(P_B)$ and a point $m \in \Z^d \times \set{0}$ such that
	\begin{equation*}
		V(e^{2 \pi i \sprod{m}{\cdot}}, N_k,(\psi_{k,j})_{j}) \ge \frac{1}{2^{d+1}N_k}.
	\end{equation*}
	We choose $N_k = 2 p_k$ where $p_k$ is the $k$-th prime number and
	\begin{equation*}
		m = (\mathbf{2}, \mathbf{0}) \qqtq{where} \mathbf{x} = (x, \dots, x) \in \R^d \text{ for } x \in \R.
	\end{equation*}
	In particular, $\CH_{N_k}(0)$ is invariant under $M_{N_k}(P_B)$ according to Lemma \ref{lem_quantizability_condition_expression_Mh_PB}. Note that for any $j \in \Z_{N_k}^d$, the equation \eqref{eq_U_h(z)=quantif_e(isigma(z,.))} implies
	\begin{align*}
		\sprod{\Op^w_{N_k}(e^{2 \pi i \sprod{(\mathbf{2}, \mathbf{0})}{\cdot}}) (f^0_j + f^0_{j-\mathbf{2}})}{f^0_j + f^0_{j-\mathbf{2}}} & = \sprod{\Op^w_{N_k}(e^{2 \pi i \sigma((\mathbf{0}, \mathbf{2}),\cdot)}) (f^0_j + f^0_{j-\mathbf{2}})}{f^0_j + f^0_{j-\mathbf{2}}}.
	\end{align*}
	Using the exact Egorov theorem \eqref{eq_exact_Egorov_torus}, the definition \eqref{eq_new_basis_f_theta_j} of $f^\theta_j$ and \eqref{eq_action_U_N((m,n)/N)_on_e_theta_j}, this is equal to
	\begin{align*}
		\sprod{U_N( \mathbf{0}, \mathbf{\frac{2}{N}}) M_N(J)^{-1} (e^0_j + e^0_{j-\mathbf{2}})}{ M_N(J)^{-1}(e^0_j + e^0_{j-\mathbf{2}})},
	\end{align*}
	which simplifies into
	\begin{equation*}
		\sprod{U_N( \mathbf{\frac{2}{N}}, \mathbf{0} ) (e^0_j + e^0_{j-\mathbf{2}})}{(e^0_j + e^0_{j-\mathbf{2}})} = \sprod{e^0_{j+\mathbf{2}} + e^0_{j}}{e^0_j + e^0_{j-\mathbf{2}}}.
	\end{equation*}
	Suppose that $N_k >2$, then $j$, $j-\mathbf{2}$ and $j+\mathbf{2}$ are all different in $(\Z/N_k \Z)^d$ and thus
	\begin{equation}\label{eq_proof_lower_bound_quantum_variance}
		\sprod{\Op^w_{N_k}(e^{2 \pi i \sprod{(\mathbf{2}, \mathbf{0})}{\cdot}}) (f^0_j + f^0_{j-\mathbf{2}})}{f^0_j + f^0_{j-\mathbf{2}}} = 1.
	\end{equation}
	Therefore, we search for eigenvectors of $M_{N_k}(P_B)$ of the form
	\begin{equation*}
		\frac{1}{\sqrt{2}}(f^0_j + f^0_{j-\mathbf{2}}).
	\end{equation*}
	The expression of $M_{N_k}(P_B)$ in Lemma \ref{lem_quantizability_condition_expression_Mh_PB} implies that such a vector is an eigenvector if and only if
	\begin{equation*}
		Q(j-\mathbf{2}) \equiv Q(j) \quad \modulo{2N_k} \qqtq{where} Q(x) = \sprod{Bx}{x}.
	\end{equation*}
	This is equivalent to
	\begin{equation*}
		4 \sprod{j}{B \mathbf{1}} \equiv 4 Q(\mathbf{1}) \quad \modulo{4 p_k} \iff \sprod{j}{B \mathbf{1}} \equiv Q(\mathbf{1}) \quad \modulo{p_k}.
	\end{equation*}
	We suppose that $k$ is large enough to ensure that $Q(\mathbf{1})$ and $B \mathbf{1}$ are nonzero modulo $p_k$. The equation above defines an affine hyperplane in $\F_{p_k}^d$ and thus its set of solutions $X_k$ has cardinal $p_k^{d-1}$. Let $j \in X_k$, then
	\begin{equation*}
		\sprod{j-\mathbf{2}}{B\mathbf{1}} \equiv Q(\mathbf{1}) - 2 Q(\mathbf{1}) \equiv - Q(\mathbf{1}) \quad \modulo{p_k}.
	\end{equation*}
	Suppose that $j- \mathbf{2}$ also belongs to $X_k$. Then we would have
	\begin{equation*}
		Q(\mathbf{1}) \equiv \sprod{j-\mathbf{2}}{B^T \mathbf{1}} \equiv - Q(\mathbf{1}) \quad \modulo{p_k}.
	\end{equation*}
	This is impossible because $Q(\mathbf{1})$ is nonzero modulo $p_k$. Therefore, $X_k \cap (X_k -\mathbf{2}) = \emptyset$. We define
	\begin{equation*}
		\psi_{k,j} = \frac{1}{\sqrt{2}}(f^0_j + f^0_{j-\mathbf{2}}) \qqtq{for} j \in X_k.
	\end{equation*}
	They are normed eigenvectors of $M_{N_k}(P_B)$ by construction and are orthogonal to each other since $X_k \cap (X_k -\mathbf{2}) = \emptyset$. We complete the family $(\psi_{k,j})_{j \in X_k}$ in an orthonormal eigenbasis of $M_{N_k}(P_B)$ on $\CH_{N_k}(0)$. Finally, we deduce from the equality \eqref{eq_proof_lower_bound_quantum_variance} that
	\begin{align*}
		V(e^{2 \pi i \sprod{(\mathbf{2}, \mathbf{0})}{\cdot}}, N_k) & \ge \frac{1}{2^d p_k^d} \sum_{j \in X_k} |\sprod{\Op^w_{N_k}(e^{2 \pi i \sprod{(\mathbf{2}, \mathbf{0})}{\cdot}}) \psi_{k,j}}{\psi_{k,j}}|^2 ,\\
		& = \frac{1}{2^d p_k^d} \sum_{j \in X_k} \frac{1}{4} = \frac{p_k^{d-1}}{2^{d+2}p_k^d} = \frac{1}{2^{d+1}N_k}.
	\end{align*}
\end{proof}

\subsection{Examples of semiclassical measures for parabolic matrices in higher dimensions}

We now consider parabolic symplectic matrices of the form
\begin{equation*}
	P_{2B} = \begin{pmatrix}
		I & 2B \\ 0 & I
	\end{pmatrix} \qqtqq{where} B \in M_{2d}(\Z) \qtq{,} B^T = B, \det(B) \neq 0.
\end{equation*}
They are easier to work with as the associated metaplectic operator can always be restricted to $\CH_{N}(0)$ without any parity condition on $N$ as can be seen from Lemma \ref{lem_quantizability_condition_expression_Mh_PB}. Moreover, the metaplectic operator then satisfies
\begin{align*}
	M_{(2\pi N)^{-1}} (P_{2B}) f^0_j = e^{-\frac{2 \pi i}{N} Q(j)} f^0_j,
\end{align*}
where $Q$ is the quadratic form associated to $B$, defined by $Q(x) = \sprod{Bx}{x}$ for any $x \in \Q^d$. We highlight that the coefficients of $Q$ are integers as the coefficients of $B$ are integers. Thus, still denoting by $E_\lambda$ the eigenspaces of $M_N(P_{2B})$, we have, for any $j,k \in \eint{0}{N-1}^d$ that
\begin{equation}\label{cond_e_0_j_same_eigenvalue_higher_dim_2B}
	Q(j) \equiv Q(k) \quad \modulo{N}  \iff  \exists \lambda \in \spectre(M_{N}(P_{2B})), \quad  f^\theta_j, f^\theta_{k} \in E_\lambda.
\end{equation}

We prove Theorem \ref{thm_parabolic_case_infinite_fourier_coef} that states
that starting from the dimension $2$, there always exists a semiclassical measure
whose projection on the position has all its Fourier coefficients positive.
\begin{proof}
	Let $B_0(M)^* = \set{n \in \Z^d \setminus \set{0} \, | \, \norm{n} \le M}$ and let $Q$ be the integral quadratic form associated to $B$. Suppose for the moment that for any $M > 0$, there exist infinitely many primes such that the system $(E_M)$ defined below admits solutions $(j_n)_{n \in B_0(M)^*} \in \F_p^d$.
	\begin{equation*}
		(E_M) \quad
		\forall n \in B_0(M)^*, \quad Q(j_n) = Q(j_n - n) = 0.
	\end{equation*}
	Let $(p_M)_M$ be an increasing sequence of primes such that $(E_M)$ has a solution $(j_{M,n})_{n \in B_0(M)^*}$ in $\F_{p_M}^d$.
	We consider
	\begin{equation*}
		\psi_M = \sum_{n \in B_0(M)^*} \frac{1}{2\norm{n}^{d+1}} (f^0_{j_{M,n}} +  f^0_{j_{M,n}-n}) = \sum_{j \in \F_{p_M}^d} \alpha_j f^0_j,
	\end{equation*}
	where the last equality defines the $(\alpha_j)_{j \in \F^d_{p_M}}$.
	As the $(j_{M,n})_{n \in B_0(M)^*}$ are solutions of $(E_M)$ in $\F_{p_M}^d$, we see that $\psi_M$ is an eigenvector of $M_{p_M}(P_{2B})$ thanks to \eqref{cond_e_0_j_same_eigenvalue_higher_dim_2B} and it is clearly nonzero. Moreover,
	\begin{align*}
		\norm{\psi_M} & \le \sum_{n \in B_0(M)^*}  \frac{1}{\norm{n}^{d+1}} \le \sum_{n \in \Z^d \setminus \set{0}} \frac{1}{\norm{n}^{d+1}} = C_d < + \infty.
	\end{align*}
	Therefore, $\tilde{\psi}_M = \psi_M/\norm{\psi_M}$ is a normed eigenvector of $M_{p_M}(P_{2B})$.

	By the usual computation starting from
	\eqref{eq_simplified_expression_fourier_coefficients_wigner_distribution_f_theta_j},
	we find, as all the coefficients $(\alpha_j)_j$ are real and nonnegative, that
	\begin{align*}
		|Wig(\tilde{\psi}_M)(e^{- 2 \pi i \sprod{(n,0)}{\cdot}})| & = \frac{1}{\norm{\psi_M}^2} \sum_{j \in \F_{p_M}^d} \alpha_j \overline{\alpha_{j-n}} \ge \frac{1}{\norm{\psi_M}^2} \alpha_{j_{M,n}} \alpha_{j_{M,n}-n} \ge \frac{1}{C_d^2} \frac{1}{4 \norm{n}^{2d+2}}.
	\end{align*}
	Up to an extraction, we can suppose that $Wig(\tilde{\psi}_M)$ converges weak-star to a semiclassical measure $\mu$ of $P_{2B}$ and we get that
	\begin{equation*}
		\mu(n,0) = \lim_{M \to \infty} Wig(\tilde{\psi}_M)(e^{- 2 \pi i \sprod{(n,0)}{\cdot}}) \ge \frac{1}{4 C_d^2\norm{n}^{2d+2}} > 0.
	\end{equation*}
	We used that the phase in front of the sum in \eqref{eq_simplified_expression_fourier_coefficients_wigner_distribution_f_theta_j} goes to $1$ as $p_M$ goes to $+\infty$.

	It remains to prove that the system $(E_M)$ admits solutions in $\F_p^d$ for
	infinitely many primes. We start by reducing the system to a standard form.
	Let $\delta = \det(B)$ be the determinant of $B$, which is also the
	discriminant of the quadratic form $Q$. Let $p$ be a prime such that $\delta$
	is a square in $\F_p$. By the reduction theory of quadratic forms over $\F_p$
	\cite[sec. 4.1.7, prop. 5]{serreCourseArithmetic1973a}, we know that there
	exist a matrix $S \in \GL_d(\F_p)$ such that
	\begin{equation*}
		Q(Sx) = \sum_{i=1}^d x_i^2.
	\end{equation*}
	Hence, under the change of variables, the system $(E_M)$ is equivalent to
	\begin{equation*}
		\forall n  \in S^{-1} B_0(M)^*, \quad \sum_{i=1}^{d} j_{n,i}^2 = \sum_{i=1}^{d} (j_{n,i} - n_i)^2 = 0.
	\end{equation*}
	Remark that
	\begin{equation*}
		\sum_{i=1}^{d} j_{n,i}^2 = \sum_{i=1}^{d} (j_{n,i} - n_i)^2 \iff 2 \sum_{i=1}^{d} j_{n,i} n_i = \sum_{i=1}^d n_i^2 = Q(Sn).
	\end{equation*}
	Hence, the system is equivalent to
	\begin{equation*}
		\forall n  \in S^{-1} B_0(M)^*, \quad \sum_{i=1}^{d} j_{n,i}^2 = 0   \qqtqq{and} 2 \sum_{i=1}^{d} j_{n,i} n_i = Q(Sn).
	\end{equation*}
	Consider $n \in B_0(M)^*$ such that its first coordinate is nonzero. The corresponding two equations are then equivalent to
	\begin{equation*}
		\left(Q(Sn) - 2 \sum_{i=2}^d n_i j_{n,i}\right)^2 + 4 n^2_1 \sum_{i=2}^{d} j_{n,i}^2 = 0  \qqtqq{and} j_{n,1} = \frac{1}{2n_1}Q(Sn) - \frac{1}{n_1} \sum_{i=2}^d n_i j_{n,i}.
	\end{equation*}
	If $n_1 = 0$, we do this reduction for another coordinate of $n$ and at least one of them is nonzero since $n \neq 0$. Moreover, we can divide by any nonzero coordinate of $n$ in $\F_p$ if $p$ is large enough to be prime to all of those coordinates. Therefore, the whole system is equivalent to
	\begin{equation*}
		\forall n  \in S^{-1} B_0(M)^*, \quad \begin{syst}
			& \left(Q(Sn) - 2 \sum_{\substack{i =1 ,\\i \neq  f(n)}}^d n_i j_{n,i}\right)^2 + 4 n^2_{f(n)} \sum_{\substack{i =1 ,\\i \neq  f(n)}}^d j_{n,i}^2 = 0,\\
			& j_{n,f(n)} = \frac{1}{2n_{f(n)}}Q(Sn) - \frac{1}{n_{f(n)}} \sum_{\substack{i =1 ,\\i \neq  f(n)}}^d n_i j_{n,i},
		\end{syst}
	\end{equation*}
	where $f$ is chosen to ensure that $n_{f(n)} \neq 0$. If $d = 2$, the first equation associated to each $n$ is a degree two polynomial. If $d>2$, we reduce ourselves to the case of degree two polynomials by setting enough variables $j_{n,i}$ to zero. We need to take $p$ large enough to ensure that the leading coefficients of the polynomials are relatively prime to $p$.

	The final step consists in finding primes such that all the obtained degree-two
	polynomials have a root modulo $p$. A polynomial of degree two admits a root in
	a field $k$ if and only if its discriminant admits a square root in $k$.
	Therefore, we are done if we can find primes $p$ such that $\delta$ and the
	discriminants of all the polynomials are squares modulo $p$. This is always
	possible thanks to Lemma
	\ref{lem_finite_integer_squares_modulo_infinity_primes}. We have thus proven
	that the system $(E_M)$ has a solution modulo $p$ for infinitely many
	primes.
\end{proof}

We now prove Theorem \ref{thm_weird_measure_parabolic_isotropic} that tells us
that as long as the quadratic form $Q$ associated to $B$ is isotropic over
$\Q$, then there exists a semiclassical measure $\mu$ associated to $P_{2B}$ of
the form
\begin{equation*}
	\mu = \delta_{\T_E} \otimes \Leb_{\T_{E^\bot}} \otimes \delta_{\T_F} \otimes \Leb_{\T_{F^\bot}}.
\end{equation*}
\begin{proof}
	Recall that we fixed $B \in \Sym(\Z^d)$ such that the associated quadratic form $Q$ is isotropic over $\Q$ with isotropy index $r$.
	By the reduction theory of quadratic forms over a field of characteristic different from $2$ \cite[sec. 4.1.6]{serreCourseArithmetic1973a}, there exists an invertible matrix $S \in \GL_d(\Q)$ and rational numbers $a_{2r+1}, \dots, a_d \in \Q^*$ such that
	\begin{equation}\label{eq_proof_thm_weird_measure_parabolic_isotropic_reduction_quad_form}
		\forall x \in \Q^d, \quad Q(Sx) = \sum_{k=1}^r x_{2k-1} x_{2k} + \sum_{k=2r+1}^d a_k x_k^2.
	\end{equation}
	The first part corresponds to the isotropic part of $Q$ and the second one to its anisotropic part. As we supposed that $Q$ is isotropic over $\Q$, the dimension $r$ of the maximal rational isotropic subspace of $Q$ is positive.

	We consider in the following primes $p$ that are relatively prime to the
	determinant $\det(S)$, the denominators of any of the coefficients of $S$, and
	the denominators of the numbers $a_{2r+1}, \dots, a_d \in \Q$. In particular,
	the equation
	\eqref{eq_proof_thm_weird_measure_parabolic_isotropic_reduction_quad_form} is
	again valid in $\F_p$ and the matrix $S$ is invertible in $\F_p$. Denote by $X$
	the linear subspace of $\F_p^d$ defined by
	\begin{equation*}
		X = \set{(x_1, 0, x_{3}, 0, \dots, x_{2r-1}, 0, 0, \dots, 0) \, | \, x_{2k-1} \in \F_p  \text{ for } k \le r}.
	\end{equation*}
	As any $x \in X$ satisfies $Q(Sx) = 0$, the vector $\psi_p$ defined below is a normed eigenvector of $M_{p}(P_{2B})$ on $\CH_{p}(0)$.
	\begin{equation*}
		\psi_p = \frac{1}{\sqrt{|X|}} \sum_{x \in X} f^0_{Sx} = \sum_{j \in \F_p^d} \alpha_j f^0_j \qqtq{where} \alpha_{Sj} = \begin{syst}
			&\frac{1}{\sqrt{|X|}} \qtq{if} j \in X ,\\
			&0 \qtq{otherwise}.
		\end{syst}
	\end{equation*}
	The relation \eqref{eq_simplified_expression_fourier_coefficients_wigner_distribution_f_theta_j} gives
	\begin{align*}
		Wig(\psi_p)(e^{-2 \pi i \sprod{(m, n)}{\cdot}}) & = e^{\frac{\pi i}{p} \sprod{m}{n}} \sum_{j \in \F_p^d} e^{-\frac{2 \pi i}{p} \sprod{n}{j}} \alpha_{j} \overline{\alpha_{j - m}},\\
		& = e^{\frac{\pi i}{p} \sprod{m}{n}} \sum_{j \in \F_p^d} e^{-\frac{2 \pi i}{p} \sprod{S^{T}n}{j}} \alpha_{Sj} \overline{\alpha_{S(j - S^{-1}m)}} ,\\
		& = \frac{e^{\frac{\pi i}{p} \sprod{m}{n}}}{|X|} \sum_{j \in X, j-S^{-1}m \in X} e^{-\frac{2 \pi i}{p} \sprod{S^{T}n}{j}}.
	\end{align*}
	Note that $x$ and $x-S^{-1}m$ can both belong to $X$ only if the reduction $\modulo{S^{-1}m}_p$ of $S^{-1}m$ modulo $p$ is also in $X$. Hence, if this is not the case, the sum is $0$. Suppose now that $\modulo{S^{-1}m}_p$ is in $X$ and let $\tilde{n} = S^T n$. Then, using $|X| = p^r$, we find that the sum is equal to
	\begin{align*}
		\frac{e^{\frac{\pi i}{p} \sprod{m}{n}}}{p^r} \sum_{j_1, j_3, \dots, j_{2r-1} \in \F_p} e^{-\frac{2 \pi i}{p} \tilde{n}_1 j_1} \cdots e^{-\frac{2 \pi i}{p} \tilde{n}_{2r-1} j_{2r-1}} = 
		\begin{syst}
			e^{\frac{\pi i}{p} \sprod{m}{n}} \qtq{if} \modulo{\tilde{n}}_p \in X ,\\
			0 \qtq{otherwise}.
		\end{syst}
	\end{align*}
	Finally, we obtain that
	\begin{equation*}
		Wig(\psi_p)(e^{-2 \pi i \sprod{(m, n)}{\cdot}}) = \begin{syst}
			&e^{\frac{\pi i}{p} \sprod{m}{n}} \qtq{if} \modulo{S^{-1}m}_p, \modulo{S^{T}n}_p \in X ,\\
			&0 \qtq{otherwise}.
		\end{syst}
	\end{equation*}
	Denote by $E$ the linear subspace of $\Q^{d}$ defined by
	\begin{equation*}
		E = \set{(x_1, 0, x_{3}, 0, \dots, x_{2r-1}, 0, 0, \dots, 0) \, | \, x_{2k-1} \in \Q  \text{ for } k \le r}.
	\end{equation*}
	Let $m,n \in \Z^d$. Since we can write $S^{-1}m = \frac{\tilde{m}}{a}$ with $\tilde{m} \in \Z^d$ and $a$ relatively prime to $p$, the condition $\modulo{S^{-1}m}_p \in X$ is equivalent, if $p$ is larger than any coordinate of $\tilde{m}$, to $S^{-1}m \in E$. Indeed,
	\begin{equation*}
		\modulo{m}_p \in X \iff \frac{1}{a} \modulo{\tilde{m}}_p \in X \iff \modulo{\tilde{m}}_p \in X \iff \tilde{m} \in E \iff m \in E
	\end{equation*}
	The reasoning is similar for $S^{T}n$. Therefore, we find that for any $(m,n) \in \Z^{2d}$
	\begin{equation*}
		Wig(\psi_p)(e^{-2 \pi i \sprod{(m,n)}{\cdot}}) \xrightarrow[p \to + \infty]{} \begin{syst}
			&1 \qtq{if} S^{-1}m,S^{T}n \in E ,\\
			&0 \qtq{otherwise}.
		\end{syst}
	\end{equation*}
	By density of the trigonometric polynomials in $\CC^\infty(\T^2)$, we get that $Wig(\psi_p)$ converges as a distribution to a semiclassical measure $\mu$ of $P_{2B}$ whose Fourier coefficients are
	\begin{equation*}
		\widehat{\mu}(m,n) = \begin{syst}
			&1 \qtq{if} m \in S E \text{ and } n \in S^{-T} E ,\\
			&0 \qtq{otherwise}.
		\end{syst}
	\end{equation*}
	Denote $\T^{l}_{V}$ the image in the torus of $V \otimes \R$ for $V$ a subspace of $\Q^{2d}$ of dimension $l$. We deduce from the Fourier coefficients of $\mu$ that
	\begin{equation*}
		\mu = \delta_{\T^r_{SE}} \otimes \Leb_{\T^{d-r}_{(SE)^\bot}} \otimes \delta_{\T^r_{S^{-T}E}} \otimes \Leb_{\T^{d-r}_{(S^{-T}E)^\bot}}.
	\end{equation*}
	One can check that the different subspaces verify the properties claimed in the theorem using the fact that $\begin{pmatrix}
			S & 0 \\ 0 & S^{-T}
		\end{pmatrix}$ is symplectic.
\end{proof}

We finally prove Theorem \ref{thm_bound_sum_fourier_coeff_parabolic} that gives
an upper bound on the sum of Fourier coefficients of the projection on
position of semiclassical measures associated to parabolic matrices.
\begin{proof}
	Let $M \ge 0$. Suppose by contradiction that the upper bound does not hold. Then there exist $\epsilon >0$ and a semiclassical measure $\mu \in \CS \CM (P_B)$ such that
	\begin{equation*}
		\sum_{n \in B_0(M)^*} |\widehat{\pi_\# \mu}(n)| \ge H(M)(1+\epsilon).
	\end{equation*}
	Denote by $(\psi_{N})_N$ the sequence of normed eigenvectors of $M_N(P_B)$ on $\CH_N(\theta_N)$ whose Wigner distribution $Wig(\psi_N)$ converges weak-star to $\mu$. Up to reducing $\epsilon$ a bit, we can suppose that for any $N$ larger than some fixed $N_0$, we have
	\begin{equation}\label{ineq_proof_control_sum_coeff_Fourier_lower_bound_sum_wigner}
		\sum_{n \in B_0(M)^*} |Wig(\psi_N)(e^{-2 \pi i \sprod{(n,0)}{\cdot}})| \ge H(M)(1+\epsilon).
	\end{equation}
	Decompose $\psi_N$ in the basis $(f^{\theta_N}_j)_j$ as
	\begin{equation*}
		\psi_N = \sum_{j \in S_N} \alpha_j f^{\theta_N}_j \qqtq{with} \sum_{j \in S_N} |\alpha_j|^2 = 1 \qtq{and} S_N \subset \eint{0}{N-1}^d.
	\end{equation*}
	Lemma \ref{lem_properties_basis_f_theta_j} gives
	\begin{align*}
		\sum_{n \in B_0(M)^*} |Wig(\psi_N)(e^{-2 \pi i \sprod{(n,0)}{\cdot}})| & \le \sum_{n \in B_0(M)^*} \sum_{j, \modulo{j-n}_N \in S_N} |\alpha_j| |\alpha_{\modulo{j-n}_N}| ,\\
		& \le \sum_{n \in B_0(M)^*} \sum_{j, \modulo{j-n}_N \in S_N} \frac{|\alpha_j|^2 + |\alpha_{\modulo{j-n}_N}|^2}{2}.
	\end{align*}
	Denote by $S_{N,j}$ the set
	\begin{equation*}
		S_{N,j} = \set{n \in B_0(M)^* \, | \, j, \modulo{j-n}_N \in S_N}.
	\end{equation*}
	Then the sum of Fourier coefficients of the Wigner distribution is bounded by
	\begin{align*}
		\sum_{j \in \eint{0}{N-1}^d} |S_{N,j}| \frac{|\alpha_j|^2 + |\alpha_{\modulo{j-n}_N}|^2}{2} & \le \max_{j \in \eint{0}{N-1}^d} |S_{N,j}| \sum_{j \in \eint{0}{N-1}^d} \frac{|\alpha_j|^2 + |\alpha_{\modulo{j-n}_N}|^2}{2} ,\\
		& \le \max_{j \in \eint{0}{N-1}^d} |S_{N,j}|.
	\end{align*}
	This inequality combined with \eqref{ineq_proof_control_sum_coeff_Fourier_lower_bound_sum_wigner} leads to the existence of an integer $N_0 \in \N$ such that
	\begin{equation*}
		\forall N \ge N_0, \, \exists j_{N} \in S_{N} \quad |S_{N,j_{N}}| \ge H(M) + 1.
	\end{equation*}
	We used that $H(M)$ and $|S_{N,j}|$ are integers. Recall from Lemma \ref{lem_quantizability_condition_expression_Mh_PB} that
	\begin{equation*}
		M_N(P_B) e^{\theta_N}_j = e^{-\frac{\pi i}{N} Q(j- \theta_{n,\xi})} e^{\theta_N}_j = e^{-\frac{\pi i}{N} Q(\theta_{n,\xi})} e^{-\frac{\pi i}{N} (Q(j) - \sprod{j}{2B\theta_{n,\xi}})} e^{\theta_N}_j.
	\end{equation*}
	Therefore, there exist $a_N \in \Z/2N\Z$ such that
	\begin{equation*}
		\forall j \in S_N, \quad Q(j) - \sprod{j}{2B\theta_{N,\xi}} \equiv a_N \quad \modulo{2N}.
	\end{equation*}
	This equation makes sense in $\Z/2N\Z$ because $2B \theta_{N, \xi}$ is an element of $\Z^d$. Indeed, according to Lemma \ref{lem_quantizability_condition_expression_Mh_PB}, $\theta_{N, \xi}$ must be a solution in $\T^d$ of
	\begin{equation*}
		B \theta_{N, \xi} \equiv \frac{N \modulo{B_{i,i}}_2}{2} \quad \modulo{\Z^d} \iff \theta_{N, \xi} \in \left(B^{-1} \frac{N \modulo{B_{i,i}}_2}{2} + B^{-1} \Z^d \right) \cap \T^d.
	\end{equation*}
	It is clear from the equation on the left that $2B \theta_{N, \xi}$ is in $\Z^d$ and we record from the characterization on the right that $\theta_{N, \xi}$ can only take a finite number of values in $\T^d$. Note also that a computation similar to the one done at the end of the proof of Lemma \ref{lem_quantizability_condition_expression_Mh_PB} implies that the residue class of $Q(j) - \sprod{j}{2B\theta_{N, \xi}}$ modulo $2N$ only depends on the residue class of $j$ modulo $N$.
	Denote by $X_N \subset B_0(M)^*$ the set
	\begin{equation*}
		X_N = \set{n \in B_0(M)^* \, |\, j_{N} \in S_N \text{ and } \modulo{j_{N} - n}_N \in S_N}.
	\end{equation*}
	Since $(X_N)_N$ is a sequence of subsets of the finite set $B_0(M)^*$, we can suppose, up to an extraction, that $X_N = \tilde{X}$ is independent of $N$. We can apply the same reasoning to $(\theta_{N, \xi})_N$ since this sequence only takes a finite number of values in $\T^d$. All in all, we obtain that there exists $\tilde{\theta_{x}} \in \T^d$ and a subset $X \subset B_0(M)^*$ such that $|X| \ge H(M) + 1$ and the following system has solutions for infinitely many $N$
	\begin{equation*}
		\exists j \in (\Z/N \Z)^d ,\,  \forall n \in X, \quad Q(j) - \sprod{j}{2B\theta_{x}} \equiv Q(j-n) - \sprod{j-n}{2B\theta_x} \quad \modulo{2N}.
	\end{equation*}

	We can now use Lemma \ref{lem_empty_over_C_empty_over_ZN}, to obtain the
	existence of a solution over $\C$ of the same system. In other words,
	\begin{equation*}
		\exists j \in \C ,\,  \forall n \in X, \quad Q(j) - \sprod{j}{2B\theta_{x}} = Q(j-n) - \sprod{j-n}{2B\theta_x}.
	\end{equation*}
	Since $Q(j) - \sprod{j}{2B\theta_{x}} = Q(j- \theta_x) - Q(\theta_x)$, this is equivalent to
	\begin{equation*}
		\exists j \in \C ,\,  \forall n \in X, \quad Q(j) = Q(j-n).
	\end{equation*}
	This would imply that $H(M) \ge |X|$ by definition of $H$, which is in direct contradiction with the fact that $|X| \ge H(M) + 1$.

	It remains to find a bound on $H(M)$. Recall that $H(M)$ is the cardinal of the
	largest subset $X$ of $B_0(M)^*$ such that the system
	\begin{equation*}
		\forall n \in X, \quad Q(j) = Q(j-n).
	\end{equation*}
	admits a solution $ j \in \C^{d}$. The system is equivalent to
	\begin{equation*}
		\forall n \in X, \quad 2 \sprod{j}{Bn} = \sprod{Bn}{n}.
	\end{equation*}
	Therefore if $n \in X$, its opposite $-n$ cannot belong to $X$, otherwise
	\begin{equation*}
		-Q(n) = -\sprod{Bn}{n} = 2 \sprod{j}{-Bn} = \sprod{B(-n)}{(-n)} = \sprod{Bn}{n} = Q(n).
	\end{equation*}
	which would imply $n = 0$ since $Q$ is anisotropic over $\Q$. This is in contradiction with $X \subset B_0(M)^*$. As $B_0(M)^*$ is stable under the multiplication by $-1$ and $0$ does not belong to it, we deduce that
	\begin{equation*}
		H(M) \le \frac{|B_0(M)^*|}{2}.
	\end{equation*}
\end{proof}

\subsection{Link with the semiclassical measure of the Laplacian on the torus}

The Laplacian $-\Delta_{\T^d}$ on $\T^{d}$ is a self-adjoint unbounded
diagonalizable operator on $L^{2}(\T^d)$. An orthonormal eigenbasis is given by
$(e^{2 \pi i \sprod{n}{\cdot}})_{n \in \Z^d}$ and the eigenspace $E_\lambda$
associated to the eigenvalue $4 \pi^2 \lambda$ is
\begin{equation*}
	E_\lambda = \Vect \set{e^{2 \pi i \sprod{n}{\cdot}} \, | \, \norm{n}^2 = \lambda}.
\end{equation*}
The dimension of $E_\lambda$ is equal to the number of ways to write $\lambda$ as a sum of $d$ squares.

To each normed vector $\psi$ in $L^2(\T^d)$ is associated a Wigner distribution
$Wig(\psi)$ given by
\begin{equation*}
	Wig_h(\psi)(a) = \sprod{\Op^w_h(a)\psi}{\psi}_{L^2(\T^{d})} \qqtq{for} a \in \CC^\infty(\T^{2d}).
\end{equation*}
This is well defined. Indeed, writing $a = \sum_{n \in \Z^{2d}} a_n e^{2 \pi i \sigma(n,\cdot)}$ and using \eqref{eq_quantization_obs_torus}, we see that
\begin{align*}
	|Wig_h(\psi)(a)| & = |\sum_{n \in \Z^{2d}} a_n \sprod{U_h(2 \pi h n)\psi}{\psi}_{L^2(\T^{d})}| \le \norm{\psi}^2 \sum_{n \in \Z^{2d}} |a_n| \le C_d \max_{|\alpha| \le 2d+2} \norm{\partial^\alpha a}_\infty.
\end{align*}
since the quantum translations $U_h(2 \pi h n)$ are unitary on $L^2(\T^d)$, as can be checked from \eqref{eq_action_quantum_translation_on_functions}.

A semiclassical measure for the Laplacian is defined as an accumulation point
of a sequence $(Wig_{\frac{1}{8 \pi \sqrt{\lambda}}}(\psi_\lambda))_{\lambda \in \N}$ in the space of distributions where $(\psi_\lambda)_\lambda$ is a
sequence of normed eigenvectors of the Laplacian of eigenvalues $4\pi^2\lambda$. We denote the set of semiclassical measures for the Laplacian by
$\CS \CM(-\Delta_{\T^d})$. It is a classical result \cite[thm. 5.6]{zworskiSemiclassicalAnalysis2012} that any semiclassical measure is
supported on the energy shell $\set{(x, \xi) \in \T^d \times \R^d \, | \,
\norm{\xi}^2 = \frac{1}{16}}$ and is invariant under the geodesic flow on it. The
value $1/16$ for the energy shell comes from our choice of scaling $h =
\frac{1}{8 \pi \sqrt{\lambda}}$, instead of the usual $h = \lambda^{-1/2}$.

We now associate an eigenvector of $M_N(P_{2I})$ on $\CH_{N}(0)$ to every
eigenvector of $-\Delta_{\T^d}$. The factor $2$ could probably be removed by
considering the restriction of $M_h(P_I)$ to different $\CH_N(\theta)$, but we
keep it for the sake of simplicity. Let $\psi_\lambda \in E_\lambda$ be a
normed eigenvector of $-\Delta_{\T^d}$ of eigenvalue $4\pi^2 \lambda$. It can be
written as
\begin{equation*}
	\psi_\lambda = \sum_{\substack{j \in \Z^d}} \alpha_j e^{2 \pi i \sprod{j}{\cdot}} \qqtq{with} \sum_{\substack{j \in \Z^d ,\\ \norm{j}^2 = \lambda}} |\alpha_j|^2 = 1 \qtq{and} \alpha_j = 0 \text{ if } \norm{j}^2 \neq \lambda.
\end{equation*}
We associate to it the vector $S_N(\psi_\lambda)$, where $S$ stands for sampling, defined by
\begin{equation*}
	S_N(\psi_\lambda) = \psi_\lambda \sum_{k \in \Z^d} e^{2 \pi i N \sprod{k}{x}} \in \CH_N(0).
\end{equation*}
Its decomposition in the basis $(f^0_j)_{j \in \eint{0}{N-1}^d}$ is
\begin{equation*}
	S_N(\psi_\lambda) = \sum_{j \in \Z^d } \alpha_{j} f^0_j = \sum_{k \in (\Z / N \Z)^d} \beta_k f^0_k \qqtq{where} \beta_k = \sum_{\substack{j \in \Z^d ,\\ j \equiv k \modulo{N}}} \alpha_j,
\end{equation*}
as a direct computation starting from \eqref{eq_new_basis_f_theta_j} gives
\begin{equation*}
	f^0_j (x) = e^{2 \pi i \sprod{j}{x}} \sum_{k \in \Z^d} e^{2 \pi i N \sprod{k}{x}}.
\end{equation*}
We also used that the vector $f^0_j$ only depends on the residue class of $j$ modulo $N$ according to \eqref{eq_transformation_rule_e_theta_j} and \eqref{eq_new_basis_f_theta_j}.
Note that $S_N(\psi_\lambda)$ may not be normed anymore as there could be multiple occurrences of the same $f^0_j$ as there could be multiple terms in the sum defining $\beta_k$. However, such a problem does not occur if we take $N^2 \ge 8 \lambda$ as then any solution of $\norm{j}^2 = \lambda$ satisfies $\norm{j} \le \sqrt{\lambda} \le \frac{N}{2 \sqrt{2}}$ and thus solutions have different reduction modulo $N$ as long as they are different.
Finally, we observe that $S_N(\psi_\lambda)$ is an eigenvector of $M_N(P_{2I})$. Indeed, Lemma \ref{lem_quantizability_condition_expression_Mh_PB} leads to
\begin{equation*}
	M_N(P_{2I}) S_N(\psi_\lambda) = \sum_{\substack{j \in \Z^d ,\\ \norm{j}^2 = \lambda}} \alpha_{j} e^{-\frac{2 \pi i}{N} \norm{j}^2} f^0_j = e^{-\frac{2 \pi i}{N} \lambda} S_N(\psi_\lambda).
\end{equation*}

We now turn to the comparison of the Wigner distribution of $\psi_\lambda$ and
$S_N(\psi_\lambda)$. Let $(m,n) \in \Z^{2d}$. On the one hand, the equations
\eqref{eq_action_quantum_translation_on_functions} and
\eqref{eq_U_h(z)=quantif_e(isigma(z,.))} lead to
\begin{align*}
	Wig_h(\psi_\lambda)(e^{2 \pi i \sigma((m,n),\cdot)}) & = \sprod{U_h(2\pi h(m,n)) \psi_\lambda}{\psi_\lambda}_{L^2(\T^{d})} ,\\
	& = \sum_{j,k \in \Z^d} \alpha_j \overline{\alpha_k} \sprod{U_h(2 \pi h(m,n))e^{2 \pi i \sprod{j}{\cdot}}}{e^{2 \pi i \sprod{k}{\cdot}}}_{L^2(\T^{d})},\\
	& = \sum_{j,k \in \Z^d} \alpha_j \overline{\alpha_{k}} e^{-2 \pi i \sprod{\pi h m}{n+2j}} \sprod{e^{2\pi i \sprod{j+n}{\cdot}}}{e^{2 \pi i \sprod{k}{\cdot}}}_{L^2(\T^{d})} ,\\
	& = \sum_{j \in \Z^d} \alpha_j \overline{\alpha_{j+n}} e^{-2 \pi i \sprod{\pi h m}{n+2j}}.
\end{align*}
On the other hand, the equation \eqref{eq_simplified_expression_fourier_coefficients_wigner_distribution_f_theta_j} gives
\begin{align*}
	Wig_N(S_N(\psi_\lambda))(e^{2 \pi i \sigma((m,n),\cdot)}) & = Wig_N(S_N(\psi_\lambda))(e^{2 \pi i \sprod{(n,-m)}{\cdot}}) ,\\
	& = e^{-\frac{\pi i }{N} \sprod{n}{m}} \sum_{k \in (\Z / N \Z)^d} e^{-\frac{2 \pi i}{N} \sprod{m}{j}} \beta_k \overline{\beta_{k+n}} ,\\
	& = \sum_{k \in (\Z / N \Z)^d} \sum_{\substack{j \in \Z^d ,\\
	j \equiv k \modulo{N}}} \sum_{\substack{l \in \Z^d ,\\
	l \equiv k + n \modulo{N}}} e^{-\frac{\pi i}{N}\sprod{m}{n +2j}} \alpha_j \overline{\alpha_{l}} ,\\
	& = \sum_{\substack{j,l \in \Z^d ,\\ \norm{j}^2 = \norm{l}^2 = \lambda ,\\ j \equiv l+n \modulo{N}}} e^{-\frac{\pi i}{N}\sprod{m}{n +2j}} \alpha_j \overline{\alpha_{l}}.
\end{align*}
We fix $N(\lambda) = \floor{4 \sqrt{\lambda}} + 1 $, $h(\lambda) = (2 \pi N(\lambda))^{-1}$ and suppose that $\norm{n}_\infty < N(\lambda)/2$. Observe that if $\norm{j}^2 = \norm{l}^2 = \lambda$, then
\begin{equation*}
	\norm{j-l-n}_\infty < \norm{j}_2 + \norm{l}_2 + \frac{N(\lambda)}{2} \le  2\sqrt{\lambda} + \frac{N(\lambda)}{2} \le N(\lambda).
\end{equation*}
Therefore, the condition $j \equiv l+n \modulo{N}$ is equivalent to the equality $j = l+n$ and we get
\begin{equation*}
	Wig_{h(\lambda)}(\psi_\lambda)(e^{2 \pi i \sigma((m,n), \cdot)}) = Wig_{N(\lambda)}(S_{N(\lambda)}(\psi_\lambda))(e^{2 \pi i \sigma((m,n), \cdot)}).
\end{equation*}
Hence, since the trigonometric polynomials are dense in $\CC^\infty(\T^{2d})$, we get that for any probability measure $\mu$ on $\set{(x, \xi) \in \T^d \times \R^d \, | \, \norm{\xi}^2 = \frac{1}{16}}$
\begin{equation*}
	Wig_{h(\lambda) }(\psi_\lambda) \xrightharpoonup[\lambda \to + \infty]{}^* \mu \iff Wig_{N(\lambda)}(S_{N(\lambda)}(\psi_\lambda)) \xrightharpoonup[\lambda \to + \infty]{}^* \tilde{\mu},
\end{equation*}
where $\tilde{\mu}$ is the restriction of $\mu$ to $\T^d \times [0,1]^d$ identified with $\T^{2d}$. Our choice for the dependence of $h$ and $N$ on $\lambda$ was made to ensure that $N = \frac{1}{2 \pi h}$ is an integer and that any semiclassical measure for $-\Delta_{\T^d}$ is supported on a set in frequency that is included in $[0,1]^d$ and that does not intersect its translates by $\Z^d$. This allows us to consider the measure on the torus $\T^{2d}$ without needing to periodize in frequency.

The discussion above is summarized in the proposition below.
\begin{proposition}\label{pro_link_semiclassical_laplacian_P_2I}
	Let $(\psi_\lambda)_{\lambda}$ be a sequence of normed eigenvectors of $-\Delta_{\T^d}$ associated to the eigenvalue $4 \pi^2 \lambda$. Then
	\begin{equation*}
		S_N(\psi_\lambda) = \psi_\lambda \sum_{k \in \Z^d} e^{2 \pi i N\sprod{k}{\cdot}} \in \CH_N(0),
	\end{equation*}
	is an eigenvector of $M_N(P_{2I})$. It is normed if $N^2 \ge 8 \lambda$.
	We get that for any probability measure $\mu$ on $\set{(x, \xi) \in \T^d \times \R^d \, | \, \norm{\xi}^2 = \frac{1}{16}}$
	\begin{equation*}
		Wig_{h(\lambda) }(\psi_\lambda) \xrightharpoonup[\lambda \to + \infty]{}^* \mu \iff Wig_{N(\lambda)}(S_{N(\lambda)}(\psi_\lambda)) \xrightharpoonup[\lambda \to + \infty]{}^* \tilde{\mu},
	\end{equation*}
	where $\tilde{\mu}$ is the restriction of $\mu$ to $\T^d \times [0,1]^d$ identified with $\T^{2d}$, $N(\lambda) = \floor{4 \sqrt{\lambda}} + 1$ and $h(\lambda) = \frac{1}{2 \pi N(\lambda)}$.
\end{proposition}
\noindent The inclusion part of Corollary \ref{cor_semicla_laplacian_torus_included_semicla_P_2I} follows immediately. This inclusion is strict:
\begin{itemize}
	\item for $d=1$, as the only quantum limit for the Laplacian on the circle $\T^1$ is the Lebesgue measure and this is not the case for the quantum limits of $P_{2}$ according to Theorem \ref{thm_full_characterization_parabolic_dimension_1}.
	\item for $d \ge 2$, by Remark \ref{rem_infinite_fourier_coeff_parabolic} thanks to the example built in Theorem \ref{thm_parabolic_case_infinite_fourier_coef}.
\end{itemize}

\printbibliography

\end{document}